\documentclass[a4paper,10pt]{article}

\usepackage{amsmath, amsthm, amssymb, amsfonts, amscd}
\usepackage[colorlinks=true,citecolor=blue]{hyperref}
\usepackage[utf8]{inputenc}
\usepackage[a4paper,margin=1in]{geometry}
\usepackage{setspace}
\usepackage{header}
\longtrue
\definecolor{orcidlogocol}{HTML}{A6CE39}

\theoremstyle{plain}

\newtheorem{theorem}{Theorem}[section]
\newtheorem{corollary}[theorem]{Corollary}
\newtheorem{lemma}[theorem]{Lemma}

\theoremstyle{definition}

\theoremstyle{remark}

\title{%
    Mean field limits for interacting diffusions with colored noise:\\
    phase transitions and spectral numerical methods
}

\newcommand{\email}[1]{\href{#1}{#1}}

\newcommand{\orcid}[1]{\href{https://orcid.org/#1}{\includegraphics[width=.4cm]{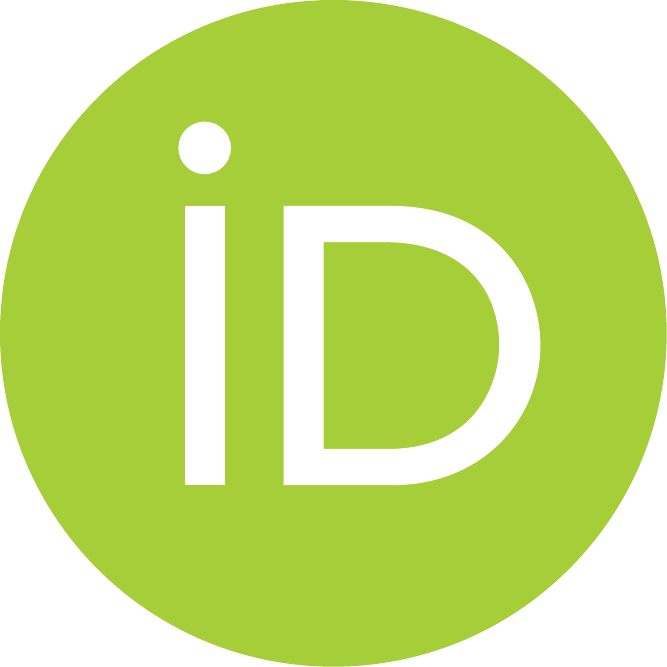}}}
\author{%
S.N. Gomes\thanks{Mathematics Institute, University of Warwick (\email{susana.gomes@warwick.ac.uk})}%
\hspace{2mm}\orcid{0000-0002-8731-367X}
\and G.A. Pavliotis\thanks{Department of Mathematics, Imperial College London (\email{g.pavliotis@imperial.ac.uk})}%
\hspace{2mm}\orcid{0000-0002-3468-9227}%
\and U. Vaes\thanks{Department of Mathematics, Imperial College London (\email{u.vaes13@imperial.ac.uk})}%
\hspace{2mm}\orcid{0000-0002-7629-7184}%
}
\date{\today}

\begin{document}
\maketitle

\begin{abstract}
    In this paper we consider systems of weakly interacting particles driven by colored noise in a bistable potential,
    and we study the effect of the correlation time of the noise on the bifurcation diagram for the equilibrium states.
    We accomplish this by solving the corresponding McKean--Vlasov equation using a Hermite spectral method,
    and we verify our findings using Monte Carlo simulations of the particle system.
    We consider both Gaussian and non-Gaussian noise processes,
    and for each model of the noise we also study the behavior of the system in the small correlation time regime
    using perturbation theory.
    The spectral method that we develop in this paper can be used for solving both linear and nonlinear, local and nonlocal (mean field) Fokker--Planck equations,
    without requiring that they have a gradient structure.%
    \\[.1cm]


    \noindent \textbf{AMS}: %
        35Q70, 
        35Q83, 
        35Q84, 
        65N35, 
        65M70, 
        82B26, 
\end{abstract}

\newcommand{\m}{m}
\renewcommand{\l}{\ell}

\section{Introduction}
Systems of interacting particles appear in a wide variety of applications,
ranging from plasma physics and galactic dynamics~\cite{BinneyTremaine2008} to mathematical biology~\cite{Farkhooi2017,Lucon2016},
the social sciences~\cite{GPY2017,Motsch2014},
active media~\cite{bain2017critical},
dynamical density functional theory (DDFT)~\cite{MR2968853,goddard2012general}
and machine learning~\cite{2018arXiv180504035L,2018arXiv180500915R,2018arXiv180501053S}.
They can also be used in models for
cooperative behavior~\cite{Dawson1983},
opinion formation~\cite{GPY2017},
and risk management~\cite{GPY2012},
and also in algorithms for global optimization~\cite{MR3597012}.

In most of the existing works on the topic,
the particles are assumed to be subject to thermal additive noise
that is modeled as a white noise process,
i.e.\ a mean-zero Gaussian stationary process that is delta-correlated in time.
There is extensive literature studying the behavior of these systems;
we mention for example works on the rigorous passage to the mean field limit~\cite{Oelschlager1984},
the long-time behavior of solutions
(see~\cite{Dawson1983,shiino1987} for a case of a ferromagnetic (quartic) potential,
and~\cite{GomesEtAl2018} for more general potentials),
multiscale analysis~\cite{GomesPavliotis2017},
and phase transitions~\cite{Tugaut2014}.

In a more realistic scenario,
the system has memory and the hypothesis of Markovianity does not hold~\cite{horsthemke1984noise,igarashi1992velocity,igarashi1988non}.
This memory can be modeled by using \emph{colored} noise,
i.e.\ noise with a nonzero correlation time (or, more precisely, a nonsingular autocorrelation function),
which is the approach we take in this paper.
For simplicity, we will assume that the noise is additive and that it can be represented by a finite-dimensional Markov process,
as in the recent study~\cite{duongmean} on mean field limits for non-Markovian interacting particles.

In this paper we will study the dynamics of a system of interacting particles of the Desai--Zwanzig type,
interacting via a quadratic Curie--Weiss potential.
The system of interacting particles is modeled by a system of stochastic differential equations (SDEs):
\begin{equation}
    \label{eq:sde-desai-zwanzig}
    \frac{\d X_t^i}{\d t} = -\left(\derivative*{1}[V]{x}(X_t^i)  + \theta \, \left(X_t^i-\frac{1}{N}\sum_{j=1}^N X_t^j\right)\right) + \sqrt{2\beta^{-1}} \, \xi_t^i, \qquad i = 1, \dotsc, N,
\end{equation}
where $N$ is the number of particles,
$V(\cdot)$ is a confining potential,
$\theta$ is the interaction strength,
$\beta$ is the inverse temperature of the system,
and $\xi_t^i$ are independent, identically distributed (i.i.d.) noise processes.

Before discussing the Desai--Zwanzig model with colored noise,
we present a brief overview of known results~\cite{Dawson1983,shiino1987} for the white noise problem.
When $\xi_t^i$ are white noise processes,
we can pass to the mean field limit $N \to \infty$ in \cref{eq:sde-desai-zwanzig}
and obtain a nonlinear and nonlocal Fokker--Planck equation,
known in the literature as a McKean--Vlasov equation,
for the one-particle distribution function $\rho(x,t)$:
\begin{equation}
\label{eq:mckean_white_noise}%
\derivative{1}[\rho]{t} = \derivative{1}{x} \left(V'(x) \, \rho + \theta \, \left(x - \int_{\real} x \, \rho(x, t) \, \d x \right) \, \rho + \beta^{-1} \, \derivative{1}[\rho]{x}\right).
\end{equation}
The McKean--Vlasov equation~\eqref{eq:mckean_white_noise} is a gradient flow with respect to the quadratic Wasserstein metric for the free energy functional
\begin{equation}
    \label{eq:free_energy}%
    \mathcal F[\rho] = \beta^{-1} \int_{\real} \rho(x) \, \ln \rho(x) \, \d x + \int_{\real} V(x) \, \rho(x) \, \d x + \frac{\theta}{2} \int_{\real} \int_{\real} F(x - y) \, \rho(x) \, \rho(y) \, \d x \, \d y,
\end{equation}
where $F(x) := x^2/2$ is the interaction potential.
The long-time behavior of solutions depends on the number of local minima of the confining potential $V$~\cite{Tugaut2014}.
It follows directly from \cref{eq:mckean_white_noise} that any steady-state solution $\rho_{\infty}(x)$ solves,
together with its first moment,
the following system of equations:
\begin{subequations}
\label{eq:mckean_steady_white}
\begin{align}
    \label{eq:mckean_steady_white_noise_rho} & \derivative{1}{x} \left(V'(x) \, \rho_{\infty}(x) + \theta \, \left(x - m \right) \, \rho_{\infty}(x) + \beta^{-1} \, \derivative{1}[\rho_{\infty}]{x}(x) \right) = 0, \\
    \label{eq:mckean_steady_white_noise_m} & m = \int_{\real} x \, \rho_{\infty} (x)\, \d x.
\end{align}
\end{subequations}
Since \cref{eq:mckean_steady_white_noise_rho} is, for $m$ fixed, the stationary Fokker--Planck equation associated with the overdamped Langevin dynamics in the confining potential
\begin{equation}
\label{eq:effective-potential}
V_{\textrm{eff}}(x;m, \theta) = V(x) + \frac{\theta}{2}(x-m)^2,
\end{equation}
solutions can be expressed explicitly as
\begin{equation}
    \label{eq:one_parameter_family}
    \rho_{\infty}(x;m,\beta,\theta) := \frac{1}{\mathcal Z(m, \beta, \theta)} \e^{-\beta V_{\textrm{eff}}(x; m, \theta)},
\end{equation}
where $\mathcal Z(m, \beta, \theta)$ is the normalization constant (partition function);
see~\cite{Dawson1983,GomesEtAl2018,GomesPavliotis2017} for more details.
By substitution in \cref{eq:mckean_steady_white_noise_m},
a scalar fixed-point problem is obtained for $m$,
the \emph{self-consistency equation}:
\begin{equation}
    \label{eq:self_consistency_white_noise}%
    m = \int_{\real} x \, \rho_{\infty}(x; m, \beta, \theta) \, \d x =: R(m, \beta, \theta).
\end{equation}
The stability of solutions to \cref{eq:mckean_steady_white}
depends on whether they correspond to a local minimum (stable) or to a local maximum/saddle point (unstable) of the free energy functional.
The free energy along the one-parameter family~\eqref{eq:one_parameter_family},
with parameter $m$, can be calculated explicitly~\cite{GomesPavliotis2017},
\begin{align*}
    \mathcal F[\rho_{\infty}(\, \cdot \,; m, \beta, \theta)] = -\beta^{-1} \, \ln \mathcal Z(m, \beta, \theta) - \frac{\theta}{2} \big( R(m,\beta,\theta) - m \big)^2,
\end{align*}
from which we calculate that
\begin{align*}
    \label{eq:free_energy_derivative}%
    \derivative{1}{m} \, \mathcal F[\rho_{\infty}(\, \dummy \,; m, \beta, \theta)] =
    - \beta \theta^2 \big( R(m, \beta, \theta) - m \big) \, \mathrm{Var}(\rho_{\infty}(\cdot \,;\, m,\beta,\theta)),
\end{align*}
where, for a probability density $\psi$,
\[
    \mathrm{Var}(\psi) := \int_{\real} \left( x - \int_{\real} \psi(x) \, \d x \right)^2 \, \psi(x) \, \d x.
\]
Though incomplete,
this informal argument suggests that the stability of a steady-state solution can also be inferred from the slope of $R(m, \beta, \theta) - m$ at the corresponding value of $m$:
if this slope is positive, the equilibrium is unstable, and conversely.
The self-consistency map and the free energy of $\rho_{\infty}(x; m, \beta, \theta)$, for a range of values of $m$,
are illustrated in \cref{fig:white_noise_free_energy} for the bistable potential $V(x) = \frac{x^4}{4} - \frac{x^2}{2}$.
\begin{figure}[ht]
    \centering
    \includegraphics[width=0.8\linewidth]{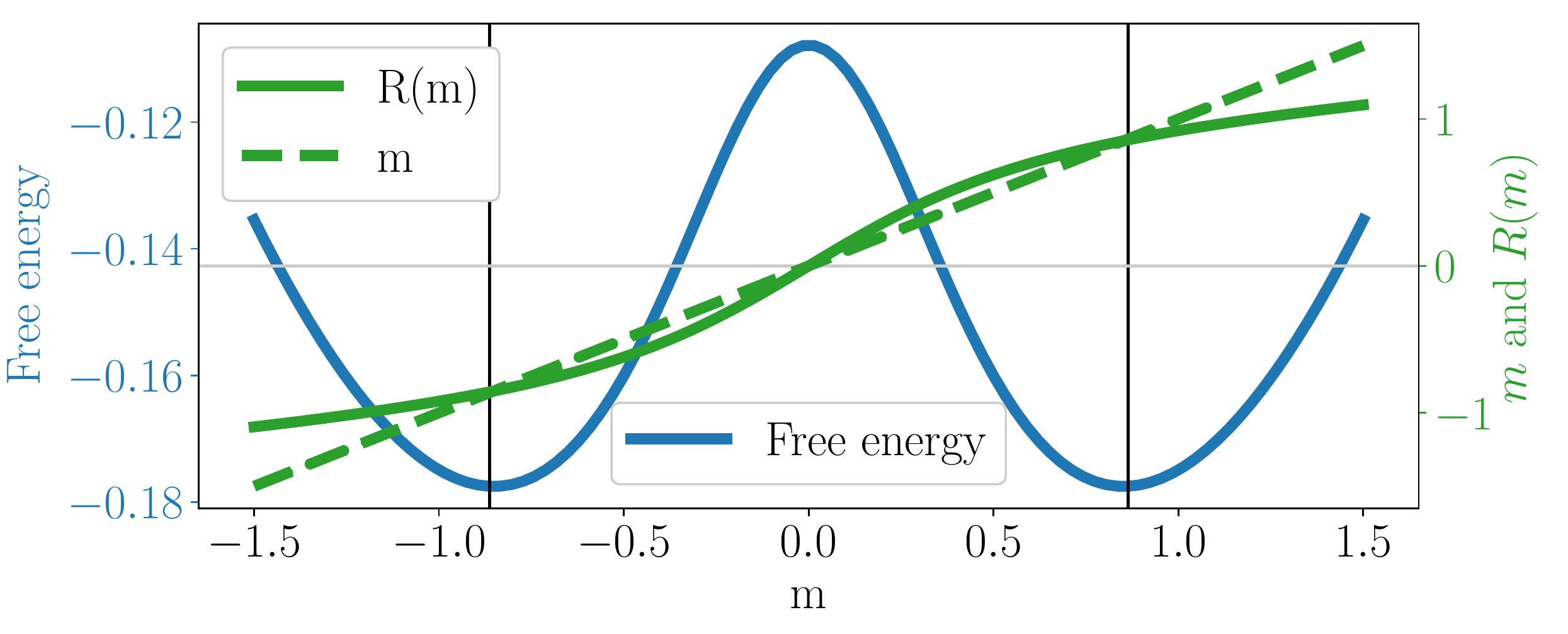}
    \caption{%
        Free energy~\eqref{eq:free_energy} of the one-parameter family~\eqref{eq:one_parameter_family} of probability densities that solve \cref{eq:mckean_steady_white_noise_rho} for some value of $m$ (in blue),
        and associated first moment $R(m)$ (in green), for fixed $\theta = 1$ and $\beta = 5$.
        Along the one-parameter family, $m = 0$ is a local maximum of the free energy,
        and it therefore corresponds to an unstable steady state of the McKean--Vlasov equation.
    } \label{fig:white_noise_free_energy}
\end{figure}
It is well-known that, when $V(\cdot)$ is an even potential,
\cref{eq:mckean_white_noise} possesses a unique, mean-zero steady-state solution for sufficiently large temperatures (i.e., small $\beta$).
As the temperature decreases,
this solution loses its stability and two new solutions of the self-consistency equation emerge,
corresponding to a pitchfork bifurcation;
see \cite{Dawson1983,GomesPavliotis2017} for details.

As mentioned above,
in this paper we focus on the case where the noise processes $\xi_t^i$ in \cref{eq:sde-desai-zwanzig} have a nonzero correlation time,
and in particular we assume that each noise process can be represented using a (possibly multi-dimensional) SDE,
in which case \cref{eq:sde-desai-zwanzig} leads to a Markovian system of SDEs in an extended phase space.
The colored noise will be modeled by either an Ornstein--Uhlenbeck process, harmonic noise~\cite[Example 8.2]{pavliotis2011applied},
or a non-Gaussian reversible diffusion process.

Though more realistic,
the use of colored noise presents us with some difficulties.
First,
the introduction of an extra SDE for the noise breaks
the gradient structure of the problem;
while we can still pass formally to the limit $N \to \infty$ in \cref{eq:sde-desai-zwanzig}
and obtain a McKean--Vlasov equation for the associated one-particle distribution function,
it is no longer possible to write a free energy functional, such as \cref{eq:free_energy}, that is dissipated by this equation.
Second, the McKean--Vlasov equation is now posed in an extended phase space,
which increases the computational cost of its numerical solution via PDE methods.
And third,
it is no longer possible to obtain an explicit expression for the one-parameter family of (possible) stationary solutions to the mean field equation,
as was possible in \cref{eq:one_parameter_family},
which renders the calculation of steady states considerably more difficult.

When the correlation time of the noise is small,
the latter difficulty can be somewhat circumvented by constructing
an approximate one-parameter family of solutions through appropriate asymptotic expansions in terms of the correlation time,
from which steady-state solutions of the McKean--Vlasov dynamics can be extracted by solving a self-consistency equation similar to~\eqref{eq:self_consistency_white_noise},
see \cref{eq:approximate_self_consistency}.
Outside of the small correlation time regime, however,
finding the steady-states of the McKean--Vlasov equation requires
a numerical method for PDEs in all but the simplest cases.

In this work,
we propose a novel Hermite spectral method
for the time-dependent and steady-state equations,
applicable to the cases of both white and colored noise.
Discretized in a basis of Hermite functions,
the McKean--Vlasov equation becomes a system of ordinary differential equations
with a quadratic nonlinearity originating from the interaction term.
In contrast with other discretization methods for PDEs,
the use of (possibly rescaled) Hermite functions for the problem under consideration leads to an efficient numerical method,
first because Hermite functions have very good approximation properties in $L^2$,
but also because all the differential operators appearing in the McKean Vlasov equation
lead to sparse matrices in Hermite space,
with a small bandwidth related to the polynomial degree of $V$
(provided that a suitable ordering of the multi-indices is employed).
To solve the finite-dimensional system of equations obtained after discretization of the time-dependent equation,
we employ either the Runge--Kutta 45 method (RK45) or
a linear, semi-implicit time-stepping scheme.

\iflong%
To assess the performance of our numerical method,
we compare its efficiency in the white noise case
with that of the finite volume scheme developed in~\cite{carrillo2015finite},
the applicability of which depends on the existence a gradient structure of \cref{eq:mckean_white_noise}.
\fi%
We also verify that our results agree with known analytical solutions in simple settings,
and with explicit asymptotic expansions in the small correlation time regime.
We then use our spectral method,
together with asymptotic expansions and Monte Carlo (MC) simulations of the particle system,
to construct the bifurcation diagram of the first moment of the steady-state solutions as a function of the inverse temperature.

For the reader's convenience, we summarize here the main results of this paper:
\begin{enumerate}
    \item The systematic study of the effect of colored noise,
        both Gaussian and non-Gaussian, on the long-time behavior of the McKean--Vlasov mean field equation,
        including the effect of colored noise on the structure and properties of phase transitions.
    \item The development and analysis of a spectral numerical method for the solution of linear or nonlinear,
        local or nonlocal Fokker--Planck-type equations.
        In particular, our method does not depend on an underlying gradient structure for the PDE.
\end{enumerate}

The rest of the paper is organized as follows.
In \cref{sec:model},
we present the models for the colored noise and
we derive formally the mean field McKean--Vlasov equation associated with the interacting particle system.
In \cref{sec:numerics},
we present the numerical methods used to (a) solve the time-dependent and steady-state Fokker--Planck (or McKean--Vlasov) equations
and (b) solve the finite-dimensional system of interacting diffusions~\eqref{eq:sde-desai-zwanzig}.
\iflong
In~\cref{ssub:asymptotic_analysis_for_the_galerkin_formulation},
we study the performance of our numerical method in the small correlation time regime,
and we verify numerically the convergence rates to the white noise solution in the limit where the correlation time tends to 0.
\fi
In \cref{sec:results},
we describe our methodology for
constructing the bifurcation diagrams and we present the associated results.
\Cref{sec:conclusions} is reserved for conclusions and perspectives for future work.

\section{The model}%
\label{sec:model}
We consider the following system of weakly interacting diffusions,
\begin{equation}
    \label{eq:main-sde}%
    \d X_t^i = - \left( V'(X_t^i) + \theta\left(X_t^i - \frac{1}{N}\sum_{j=1}^N X_t^j\right)\right) \, \d t + \sqrt{2\beta^{-1}} \, \eta^i_t  \, \d t, \quad 1 \leq i \leq N,
\end{equation}
where the noise processes $\eta^i_t$ are independent, mean-zero, second-order stationary processes
with almost surely continuous paths and autocorrelation function $K(t)$.
In the rest of this paper,
we will assume that the interaction strength $\theta$ is fixed and equal to 1
and we will use the inverse temperature $\beta^{-1}$ as the bifurcation parameter.
We will consider two classes of models for the noise:
Gaussian stationary noise processes with an exponential correlation function,
and non-Gaussian noise processes that we construct by using the overdamped Langevin dynamics in a non-quadratic potential.

\paragraph{Gaussian noise}
Stationary Gaussian processes in $\real^n$ with continuous paths and an exponential autocorrelation function
are solutions to an SDE of Ornstein--Uhlenbeck type:
\begin{equation}
    \label{eq:noise-general}%
    \d \vect{Y}_t^i =  \mat A \, \vect{Y}_t^i \, \d t + \sqrt{2} \, \mat D \, \d \vect W_t^i, \qquad i=1,\dotsc,N,
\end{equation}
where $\mat A, \mat D$ are $n\times n$ matrices satisfying Kalman's rank condition~\cite[Chapter 9]{lorenzi2006analytical},
and $\vect W_t^i$, $1 \leq i \leq n$, are independent white noise processes in $\real^n$.
We assume here that the noise is obtained by projection as $\eta_t^i = \ip{\vect Y_t^i}{\vect y_{\eta}}$,
where $\ip{\cdot}{\cdot}$ denotes the Euclidean inner product,
for some vector $\vect y_{\eta} \in \real^n$.
Throughout this paper we will consider two particular examples,
namely the scalar OU process and the harmonic noise~\cite[Chapter 8]{pavliotis2011applied}.
\begin{description}
    \item [(OU)] Scalar Ornstein--Uhlenbeck process:
        \begin{equation*}
            \label{eq:OU-noise}
            \d \eta_t^i = - \eta_t^i \, \d t + \sqrt{2} \, \d W_t^i.
        \end{equation*}
        The associated autocorrelation function is
        \begin{equation*}
            \label{eq:OU-autocorr}
            K_{OU}(t) = e^{-|t|}.
        \end{equation*}
\item [(H)] Harmonic noise:
    \begin{equation*}
        \mat A = \begin{pmatrix} 0 & 1 \\ -1  & - \gamma \end{pmatrix}, \quad
        \mat D = \begin{pmatrix} 0 & 0 \\ 0  & \sqrt{\gamma} \end{pmatrix}, \quad
        \vect y_{\eta} = \begin{pmatrix} 1 \\ 0 \end{pmatrix}.
    \end{equation*}
    In this case the noise is the solution to the Langevin equation,
    with the first and second components of $\vect Y$ corresponding to the position and velocity, respectively.
    Throughout this paper we will assume $\gamma = 1$ for simplicity.
    The associated autocorrelation function of $\eta^i$ is given by
    \begin{equation*}
        \label{eq:H-autocorr}
        K_{H}(t) =  e^{-\frac{|t|}{2}}\left(\cos\left(\frac{\sqrt{3}}{2} t\right) + \frac{\sqrt{3}}{3}\sin\left(\frac{\sqrt 3}{2} t\right)\right).
    \end{equation*}
\end{description}

\paragraph{Non-Gaussian noise} In this case,
instead of \cref{eq:noise-general} we consider
\begin{equation*}
    \label{eq:noise-NL}
    \d \eta_t^i =  - \derivative*{1}[V_{\eta}]{\eta} (\eta_t^i) \, \d t + \sqrt{2} \, \d {W}_t^i,
\end{equation*}
where now $V_{\eta}$ is a smooth non-quadratic confining potential satisfying the mean-zero condition:
\begin{equation}
    \label{eq:centering_condition_potential}
    \int_{\real} \eta \, \e^{-V_{\eta}(\eta)} \, \d \eta = 0.
\end{equation}
We consider the following choices for $V_{\eta}$:
\begin{description}
    \item [(B)] The bistable potential $V_{\eta}(\eta) = \eta^4/4 - \eta^2/2$.
    \item [(NS)] The shifted tilted bistable potential
        \begin{equation}
            \label{eq:potential_for_non_symmetric_case}
            V_{\eta}(\eta) = \frac{{(\eta-\alpha)}^4}{4} - \frac{{(\eta - \alpha)}^2}{2} + (\eta - \alpha),
        \end{equation}
        with the constant $\alpha \approx 0.885$ such that \cref{eq:centering_condition_potential} is satisfied.
\end{description}

\iflong
\begin{remark}
For the two Gaussian noise processes we consider,
it would have been equivalent (by a change of variables)
to include the inverse temperature $\beta$ in the noise equation~\cref{eq:noise-general} rather than in~\cref{eq:main-sde}.
This is not the case for non-Gaussian noise processes,
for which including the temperature in the noise equation leads to an effective diffusion coefficient,
in the limit as the correlation time tends to 0,
with a nonlinear dependence on $\beta$.
\end{remark}
\fi


\subsection{Mean field limit}
\label{ssub:mean-field}
\revision{%
    For weakly interacting diffusions,
    the derivation of the mean field McKean--Vlasov PDE is a standard, well-known result~\cite{Dawson1983,shiino1987,Oelschlager1984}.
    When $\xi_t^i$ in~\eqref{eq:sde-desai-zwanzig} are colored noise processes,
    it is also possible to pass to the mean field limit $N \to \infty$ in \cref{eq:sde-desai-zwanzig}
    and to obtain a McKean--Vlasov equation
    for the one-particle distribution function $\rho(x, \vect y, t)$:
}
\begin{subequations}
\label{eq:mckean_mean-field_dynamic}
\begin{equation}
\label{eq:FP-general}
\derivative{1}[\rho]{t} = \derivative{1}{x} \left(V' \rho + \theta \, (x - m(t)) \, \rho - \sqrt{2\beta^{-1}} \, \ip{\vect y_i}{\vect y_{\eta}} \, \rho\right)
+ \mathcal L_{\vect y}^* \rho,
\end{equation}
with the dynamic constraint
\begin{equation}
    \label{eq:self-consistency}
    m(t) = \int_{\real} \int_{\real^n} x \, \rho(x,\vect y,t) \, \d \vect y \, \d x.
\end{equation}
\end{subequations}
Here $\vect y$ are the noise variables,
denoted by $(\eta, \lambda)$ in the case of harmonic noise and just $\eta$ otherwise,
and
\begin{equation*}
\label{eq:noise-operator}
\mathcal{L}_{\vect y}^* \rho =
\begin{cases}
\derivative{1}{ {\eta} }\left(\eta \, \rho + \derivative{1}[\rho]{ {\eta} }\right), & \textrm{for scalar OU noise,}\\
\derivative{1}{ {\lambda} } \left(\lambda \, \rho + \derivative{1}[\rho]{ {\lambda} } \right) + \left(\eta \, \derivative{1}[\rho]{ {\lambda} } - \lambda \, \derivative{1}[\rho]{ {\eta} }\right), & \textrm{for harmonic noise,}\\
\derivative{1}{ {\eta} } \left(\derivative*{1}[V_{\eta}]{ {\eta} } \, \rho + \derivative{1}{ {\eta} } \rho\right), & \textrm{for non-Gaussian noise}.
\end{cases}
\end{equation*}
A formal derivation of the mean field limit is presented in~\cite{thesis_urbain},
and this derivation can be justified rigorously using the results in~\cite{Eberle,Monmarche}.

The main goal of this paper is the study of the effect of colored noise on
the structure of the bifurcation diagram for the McKean--Vlasov equation with colored noise, \cref{eq:FP-general,eq:self-consistency}.
In other words, we want to gain insight into the number of solutions to the following stationary PDE and associated constraint (self-consistency equation):
\begin{subequations}
\begin{equation}\label{eq:SS-general}
    \derivative{1}{x} \left(V'(x) \, \rho + \theta \, (x - m) \, \rho - \sqrt{2\beta^{-1}} \, \ip{\vect y_i}{\vect y_{\eta}} \, \rho\right)
    + \mathcal L_{\vect y}^* \rho = 0,
\end{equation}
\begin{equation}
    \label{eq:self-consistency-stationary}
    m = \int_{\real} \int_{\real^n} x \, \rho(x,\vect y) \, \d \vect y \, \d x.
\end{equation}
\end{subequations}
Although there still exists,
for fixed $\beta$ and fixed $\theta$,
a one-parameter family of solutions to~\eqref{eq:SS-general} (with parameter $m$),
which we will denote by $\{\rho_{\infty}(x, \vect y; m, \beta, \theta)\}_{m \in \real}$,
no closed form is available for these solutions.
This is because the detailed balance condition no longer holds in the presence of colored noise,
i.e.\ the probability flux at equilibrium does not vanish.
\revision{%
    Here, by probability flux, we mean the argument of the divergence in the Fokker--Planck operator;
}
see~\cite[Section 4.6]{pavliotis2011applied}.


\subsection{The white noise limit}%
\label{sub:the_white_noise_limit}

To study the limit of small correlation time,
it will be convenient to rescale the noise as
\begin{equation*}
    \label{eq:scaling_colored_noise}
    \eta^i_t \rightarrow \zeta \, \eta_{t/\varepsilon^2}^i / \varepsilon,
\end{equation*}
where $\varepsilon$ is a time scaling parameter,
and $\zeta$ is a model-dependent parameter ensuring that the autocorrelation function of the rescaled noise,
given by $\zeta^2 \, K(t/\varepsilon^2)/\varepsilon^2$, satisfies
\[
    \int_{0}^{\infty} \zeta^2 K(t/\varepsilon^2) / \varepsilon^2 \, \d t = \int_{0}^{\infty} \zeta^2 K(t) \, \d t = \frac{1}{2}.
\]
Then the autocorrelation of the noise converges to a Dirac delta when $\varepsilon \to 0$,
and it can be shown that, in this limit,
the solution of~\cref{eq:main-sde} converges to that of
\begin{equation*}
    \label{eq:model:homogenized-sde}
    \d X_t^i =\left( - V'(X_t^i) - \theta\left(X_t^i - \frac{1}{N}\sum_{j=0}^N X_t^j\right)\right) \, \d t + \sqrt{2\beta^{-1}} \, \d W^i_t, \quad i = 1, \dotsc  N,
\end{equation*}
where $W^i$, $i = 1, \dotsc  N$, are independent Wiener processes;
see~\cite{MR0476129} and~\cite[Chapter 11]{pavliotis2008multiscale}.
While not strictly necessary,
including the parameter $\zeta$ is convenient to obtain simpler formulas.
The value of $\zeta$ for each of the noise models considered in this paper is presented in~\cref{table:values_of_zeta}.
For the models \textbf{B} and \textbf{NS}, $\zeta$ was calculated numerically and rounded to three significant figures in this table.

\begin{table}[ht]
    \centering
    \caption{Value of $\zeta$}
    \label{table:values_of_zeta}
    \begin{tabular}{l|l|l|l|l}
        Model   & \textbf{OU}  & \textbf{H}   & \textbf{B}  & \textbf{NS} \\
        \hline
        $\zeta$ & $1/\sqrt{2}$ & $1/\sqrt{2}$ & $0.624$     & $0.944$
    \end{tabular}
\end{table}

In view of the convergence of the solution of the finite-dimensional particle system when $\varepsilon \to 0$,
we expect that also the $x$-marginals of the steady-state solutions to the McKean--Vlasov equation with colored noise,
obtained by solving \cref{eq:SS-general,eq:self-consistency-stationary},
should converge to their white-noise counterparts as $\varepsilon \to 0$.
It turns out that this is the case and,
using asymptotic techniques from~\cite{horsthemke1984noise},
it is possible to approximate the solutions $\rho_{\infty}(x, \vect y; m, \beta, \theta)$ to \cref{eq:SS-general} by a power series expansion in $\varepsilon$;
using a superscript to emphasize the dependence on $\varepsilon$,
\begin{equation}
\label{eq:epsilon-expansion}
\rho_\infty^{\varepsilon}(x,\vect y; m, \beta, \theta) = p_0(x,\vect y; m, \beta, \theta) + \varepsilon \, p_1(x,\vect y; m, \beta, \theta) + \varepsilon^2 \, p_2(x,\vect y; m, \beta, \theta) + \dotsb,
\end{equation}
From \cref{eq:epsilon-expansion}, we obtain a power series expansion for the $x$-marginal by integrating out the noise variable:
\begin{equation}
    \label{eq:x_marginal_asymptotic_expansion}
    \begin{aligned}
    \rho^\varepsilon_\infty(x;m,\beta,\theta)
    &= \int_{\real^n} \rho^{\varepsilon}_{\infty}(x, \vect y; m, \beta, \theta) \, \d \vect y\\
    &=: \rho_{\infty}(x; m, \beta, \theta) + \varepsilon \, p_1(x; m, \beta, \theta) + \varepsilon^2 \, p_2(x; m, \beta, \theta) + \dotsb,
    \end{aligned}
\end{equation}
The methodology to obtain expressions for the terms works by
substituting \cref{eq:epsilon-expansion} in \cref{eq:SS-general} and grouping the terms in powers of $\varepsilon$ in the resulting equation.
This leads to a sequence of equations that can be studied using standard techniques.
Details of the analysis leading to an explicit expression of the first nonzero correction in~\eqref{eq:x_marginal_asymptotic_expansion}
can be found in~\cite[Section 8]{horsthemke1984noise} for the particular case of the \textbf{OU} noise,
and in~\cite{thesis_urbain} for the other noise models we consider.

The order of the first nonzero correction in this expansion depends on the model:
it is equal to 1 for model \textbf{NS}, to 2 for models \textbf{OU} and \textbf{B}, and to 4 for model \textbf{H}.
In all cases, the first nontrivial term in the series expansion~\eqref{eq:x_marginal_asymptotic_expansion} can be calculated explicitly
(possibly up to constant coefficients that have to be calculated numerically).
For scalar Ornstein--Uhlenbeck noise, for example,
we have,
omitting the dependence of $V_{\textrm{eff}}$ (the effective potential defined in \cref{eq:effective-potential}) on $m$ and $\theta$ for notational convenience,
\begin{align}
    \label{eq:p-asymptotic-OU}
    \rho^\varepsilon_\infty(x;m,\beta, \theta)
    = \rho_{\infty}(x;m,\beta, \theta)\left[1+\varepsilon^2\left(C_{OU} - \frac{\beta}{2} \left(V_{\textrm{eff}}'{\left (x \right )}\right)^{2} + V_{\textrm{eff}}''{\left (x \right )}\right)\right] + \mathcal O(\varepsilon^4)%
\end{align}
Here $C_{OU}$ is a constant such that the correction integrates to $0$.
Similar expressions can be obtained for the other models; see~\cite{thesis_urbain}.

Taking into account only the first nontrivial correction,
the order of which we denote by $\delta$,
the steady-state solutions to the McKean--Vlasov equation with colored noise can be approximated by solving the approximate self-consistency equation
\begin{align}
    m &= R_0(m, \beta, \theta) + \varepsilon^{\delta} \, R_{\delta}(m, \beta, \theta) \nonumber \\
      \label{eq:approximate_self_consistency} &:= \int_{\real} x \, \rho_{\infty}(x; m, \beta, \theta) \, \d x + \varepsilon^{\delta} \int_{\real} x \, p_{\delta}(x; m, \beta, \theta) \, \d x \\
      &\approx R(m,\beta) := \int_{\real} x \, \rho^{\varepsilon}_{\infty}(x; m, \beta, \theta) \, \d x. \nonumber
\end{align}
We show in \cref{fig:rm-vs-m} that the equation $R_0(m, \beta, \theta) + \varepsilon^2 R_2(m, \beta, \theta) = m$,
for fixed $\beta = 10$, $\theta = 1$ and $\varepsilon = 0.1$,
admits three solutions in the case of \textbf{OU} noise,
similarly to the case of white noise.
This figure was generated using the asymptotic expansion~\eqref{eq:p-asymptotic-OU}.
\begin{figure}
\centering
\includegraphics[width=0.5\linewidth]{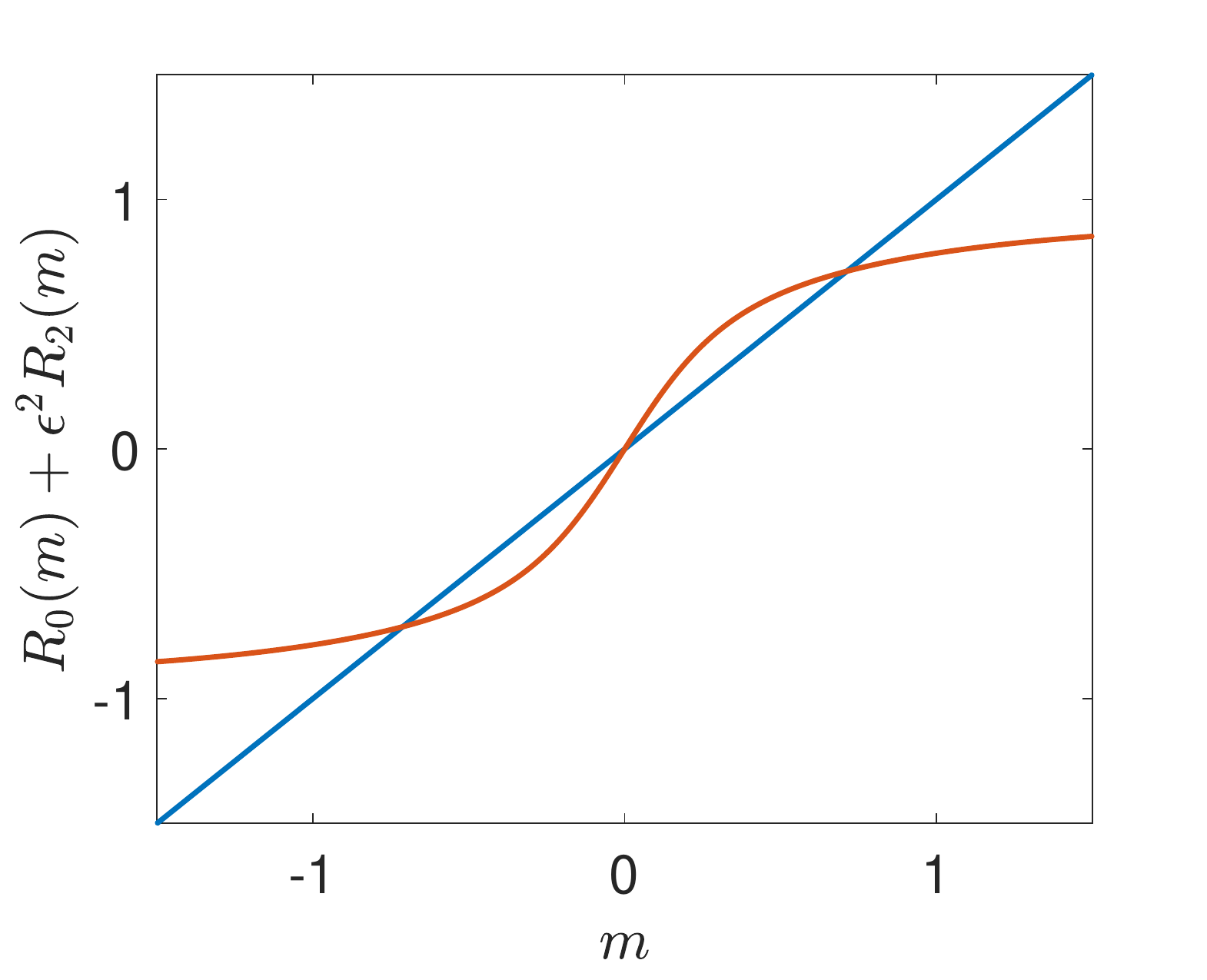}
\caption{%
    Truncated asymptotic expansion of the self-consistency map, $R_0 + \varepsilon^2 \, R_2$,
    as a function of $m$ (red line) compared to $y=m$ (blue line) for the scalar Ornstein--Uhlenbeck noise,
    with $\beta = 10, \, \theta = 1, \, \varepsilon = 0.1$.
}
\label{fig:rm-vs-m}
\end{figure}

\section{The numerical method}\label{sec:numerics}
In this section,
we describe the spectral numerical method that we will use in order to solve the time-dependent McKean--Vlasov equation, \cref{eq:FP-general,eq:self-consistency},
as well as the steady-state equation, \cref{eq:SS-general,eq:self-consistency-stationary}.
Before looking at colored noise,
we consider the case of white noise,
for which our method can be tested against the results in~\cite{GomesPavliotis2017},
which were obtained using the finite volume scheme developed in~\cite{carrillo2015finite}.

\subsection{Linear Fokker--Planck equation with white noise}%
\label{sub:white_noise_case}

We start by presenting the methodology used in the absence of an interaction term,
in which case \cref{eq:FP-general} reduces to a linear Fokker--Planck equation:
\begin{equation}
    \label{eq:fokker_planck_white_noise}
    \derivative{1}[\rho]{t} = \derivative{1}{x} \left(\derivative*{1}[V]{x} \,\rho +  \beta^{-1} \, \derivative{1}[\rho]{x} \right) =: \mathcal L_x^* \rho, \quad \rho(x, t=0) = \rho_0(x).
\end{equation}
We assume that $V(\cdot)$ is a smooth confining potential and,
consequently, the unique invariant distribution is given by $\rho_s = \frac{1}{\mathcal Z} \e^{-\beta V}$, where $\mathcal Z$ is the normalization constant~\cite[Proposition 4.2]{pavliotis2011applied}.
The Fokker--Planck operator in \cref{eq:fokker_planck_white_noise} is unitarily equivalent to a Schrödinger operator; see~\cite{abdulle2017spectral} and~\cite[Section 4.9]{pavliotis2011applied}.
Defining $u = \rho/\sqrt{\rho_s}$, the function $u$ satisfies
\begin{equation}
\begin{aligned}
    \label{eq:schrodinger_equation}
    \derivative{1}[u]{t} = \sqrt{\rho^{-1}_s} \, \mathcal L_x^* \, \left(\sqrt{\mathstrut \rho_s} \, u\right) = \beta^{-1} \, \derivative{2}[u]{x^2} + \left(\frac{1}{2} \, V''(x) - \frac{\beta}{4} \, |V'(x)|^{2} \right) \, u =: \mathcal H_x u,
\end{aligned}
\end{equation}
with the initial condition $u(x,t=0) = \rho_0 / \sqrt{\rho_s} =: u_0$.
Several works made use of Hermite spectral methods to study equations of this type,
e.g.\ ~\cite{abdulle2017spectral,MR1933042,gagelman2012spectral}.
The Schrödinger operator on the right-hand side of \cref{eq:schrodinger_equation} is selfadjoint in $\lp{2}{\real}$ and it has nonpositive eigenvalues.
Under appropriate growth assumptions on the potential $V(x)$ as $x \to \infty$,
it can be shown that its eigenfunctions decrease more rapidly than any exponential function in the $\lp{2}{\real}$ sense,
in that they satisfy $\e^{\mu \abs{x}} \, \varphi(x) \in \lp{2}{\real}$ for all $\mu \in \real$;
see~\cite{gagelman2012spectral} and also~\cite{agmon2014lectures} for a detailed study.
Under appropriate decay assumptions at infinity on the initial condition,
we expect the solution to \cref{eq:schrodinger_equation} to also decrease rapidly as $\abs{x} \to \infty$.

We denote by $\poly(d)$ the space of polynomials of degree less than or equal to $d$,
and by $\ip{\cdot}{\cdot}$ the usual $\lp{2}{\real}$ inner product.
For a quadratic potential $V_q = \frac{1}{2} \left(\frac{x}{\sigma}\right)^2$,
with $\sigma$ a scaling parameter,
the Galerkin method we employ consists in finding $u_d(t) \in \e^{-V_q/2} \, \poly(d)$ such that
\begin{subequations}
\label{eq:galerkin_approximation_white_noise_both}%
\begin{alignat}{2}
    \label{eq:galerkin_approximation_white_noise}%
    \ip{\derivative{1}[u_d]{t}}{w_d} &= \ip{\mathcal H_x u_d}{w_d}[\hat d]
                                     && \qquad \forall w_d \in \e^{-V_q/2} \, \poly(d), \quad \forall t > 0, \\
    \label{eq:galerkin_approximation_white_noise_initial_condition}%
    \ip{u_d(0)}{w_d} &= \ip{u_0}{w_d}[\hat d]
                  && \qquad \forall w_d  \in \e^{-V_q/2} \, \poly(d).
\end{alignat}
\end{subequations}
Here the subscript $\hat d \geq d$ on the right-hand side of
\cref{eq:galerkin_approximation_white_noise,eq:galerkin_approximation_white_noise_initial_condition} indicates that
the inner product is performed using a numerical quadrature with $\hat d + 1$ points.
With appropriately rescaled Gauss--Hermite points,
inner products calculated using the quadrature are \emph{exact} for functions in $\e^{-V_q/2} \, \poly(\hat d)$,
\begin{equation*}
    \ip{v_d}{w_d}[\hat d] = \ip{v_d}{w_d} \quad  \forall v_d, w_d  \in \e^{-V_q/2} \, \poly(\hat d),
\end{equation*}
which is why we did not append the subscript $\hat{d}$ to the inner products in the left-hand side of \cref{eq:galerkin_approximation_white_noise,eq:galerkin_approximation_white_noise_initial_condition}.
When $V$ is a polynomial,
it is possible to show using the recursion
\iflong
relations~\cref{eq:hermite_polynomials_1d_recursion_derivative,eq:hermite_polynomials_1d_recursion_adjoint_derivative} in
\cref{sec:hermite_polynomials},
\else
relations~(A.1) and~(A.2) in
\cite[Appendix A]{2019arXiv190405973G}
\fi
that
the inner product $\ip{\mathcal H_x u_d}{w_d}[\hat d]$ on the right-hand side of \cref{eq:galerkin_approximation_white_noise} is \emph{exactly} $\ip{\mathcal H_x u_d}{w_d}$
when $\hat d \geq d + \deg(\abs*{\derivative*{1}[V]{x}}^2)$.
This is the approach we take in all the numerical experiments presented in this paper,
and we will therefore omit the subscript $\hat d$ in \cref{eq:galerkin_approximation_white_noise} from now on.

The natural basis of $\poly(d)$ (from which a basis of $e^{-V_q/2} \, \poly(d)$ follows)
to obtain a finite-dimensional system of differential equations from the variational formulation~\eqref{eq:galerkin_approximation_white_noise}
is composed of rescaled Hermite polynomials $H_i^{\sigma}(x) := H_i(x/\sigma)$,
$0 \leq i \leq d$,
where $H_i(x)$ are the Hermite polynomials orthonormal for the Gaussian weight $\mathcal N(0, 1)$;
the corresponding basis functions of $\e^{-V_q/2} \, \poly(d)$ are then rescaled Hermite functions.
The fundamental results on Hermite polynomials, Hermite functions and the related approximation results that are used in this paper
are summarized in
\iflong
\cref{sec:hermite_polynomials}.
\else
\cite[Appendix A]{2019arXiv190405973G}.
\fi
\iflong
In general,
in addition to the rescaling parametrized by $\sigma$,
a translation could be applied in order to generate more suitable basis functions (e.g.\ when most of the mass of $\e^{-\beta V}$ is itself localized away from $x = 0$),
but for simplicity we confine ourselves in this work to the case where $V_q$ is symmetric around $x = 0$.
\fi
\iflong
\begin{remark}
Although \cref{eq:schrodinger_equation} and the associated variational formulation~\eqref{eq:galerkin_approximation_white_noise}
are convenient for analysis purposes,
for numerical purposes it is useful to perform a second unitary transformation;
defining $v = \e^{V_q/2} \, u =: u/\sqrt{\rho_q}$, the function $v$ satisfies
\begin{equation}
\begin{aligned}
    \label{eq:bk_like_equation_white_noise}
    \derivative{1}[v]{t} =  \sqrt{\rho_s^{-1} \, \rho_q^{-1}} \, \mathcal L_x^* \, \left(\sqrt{\rho_s \, \rho_q} \, v \right)
                         = \frac{1}{\beta}\left( - V_q' \, \derivative{1}{x} + \derivative{2}{x^2}
                         +  \frac{\beta}{2} \, V'' - \frac{1}{2} \, V_q'' - \frac{1}{4} \, |\beta V'|^{2}  + \frac{1}{4} \, |V_q'|^{2} \right) v,
\end{aligned}
\end{equation}
with the initial condition $v(x,t=0) = \rho_0 / \sqrt{\rho_s \, \rho_q}$,
and to approximate the solution to this equation in $\poly(d)$.
While clearly equivalent,
this approach enables us to work directly with Hermite polynomials,
for which a range of free and open-source software tools are available,
e.g.\ in the \emph{NumPy} package for scientific computing in \emph{Python}~\cite{scipy}.
Building upon the tools provided by \emph{NumPy},
we have developed a thin \emph{Python} library, available online~\cite{hermipy},
that offers the possibility of fully automating spectral methods based on Hermite polynomials.
We note that, when $V_q = \beta V$,
\cref{eq:bk_like_equation_white_noise} is merely the backward Kolmogorov equation corresponding to~\cref{eq:fokker_planck_white_noise}.
\end{remark}
\fi

It is possible to prove the convergence of the method presented above in the limit as $d \to \infty$
given appropriate additional assumptions on the confining potential $V(\cdot)$.
For simplicity we will make the following assumption,
which is satisfied for the bistable potential that we consider in this work,
but we note that less restrictive conditions would be sufficient.
\begin{assumption}
\label{assumption:potential}
The confining potential $V(\cdot)$ is a polynomial of (even) degree greater than or equal to 2.
Consequently, it satisfies
\begin{equation*}
    \label{eq:spectral_gap_assumption}
    C_1 (1 + \abs{x}^2) \leq C_2 + W :=  C_2 + \left(\frac{\beta}{4} \abs{V'}^2 - \frac{1}{2} V''\right) \leq C_3 (1 + \abs{x}^{2k}),
\end{equation*}
for constants $C_1, C_2, C_3 > 0$ and a natural number $k \geq 1$.
\end{assumption}
\revision{%
    We will denote by $\sobolev{m}{\real}[\mathcal H_x]$ the Hilbert space obtained by completion of $C^{\infty}_c(\real)$,
    the space of smooth compactly supported functions,
    with the inner product
    \begin{equation*}
        \ip{u}{v}[m,\mathcal H_x] := \ip{(-\mathcal H_x + 1)^\m u}{v}.
    \end{equation*}
    The norm associated with this Sobolev-like space will be denoted by $\norm{\cdot}[m, \mathcal H_x]$.
}
\begin{theorem}
    \label{thm:convergence_as_d_goes_to_infinity}
    Suppose that \cref{assumption:potential} holds and that
    the initial condition $u_0$ is smooth and belongs to $\sobolev{m}{\real}[\mathcal H_x]$
    for some natural number $\m \geq 2k$,
    where $k$ is as in \cref{assumption:potential}.
    Then for any $d \geq \m - 1$,
    any final time $T$ and for all $\alpha > 0$,
    it holds that
    \begin{equation*}
        \label{eq:theorem_result}
        \sup_{t \in [0, T]} \norm{u(t) - u_d(t)}^2 \leq C_{\alpha} \,\e^{\alpha T} \, {\frac{(d-\m+1)!}{(d-2k+1)!}} \norm{u_0}[m,\mathcal H_x],
    \end{equation*}
    for a constant $C_{\alpha}$ not depending on $d$, $u_0$, or $T$,
    and where $\norm{\cdot}$ denotes the $\lp{2}{\real}$ norm.
\end{theorem}
\begin{proof}
    See \cref{sec:proof_convergence_theorem}.
\end{proof}
\begin{remark}
\revision{
    \label{remark:not_optimal}
    \Cref{thm:convergence_as_d_goes_to_infinity} is not optimal.
    One one hand, it overestimates the error for large times:
    both the numerical and exact solutions converge to stationary solutions as $t \to \infty$,
    so we expect the error $\norm{u(t) - u_d(t)}^2$ to tend to finite limit when $t \to \infty$.
    Although the error between the stationary solutions can be bounded similarly to the transient error,
    see~\cref{remark:stationary_error},
    we have not obtained a result that combines both errors;
    we plan to return to this interesting question in future work.
    On the other hand, the bound on the transient error of~\cref{thm:convergence_as_d_goes_to_infinity} is probably not sharp.
    Indeed, when the initial condition $u_0$ is smooth and,
    together with all its derivatives,
    decreases exponentially as $x \to \infty$,
    \cref{thm:convergence_as_d_goes_to_infinity} implies only that
    the error decreases faster than any negative power of $d$.
    In most practical examples, however,
    we observed numerically that the convergence is in fact exponential.
}
\end{remark}
\revision{
\begin{remark}
    \label{remark:initial_condition}
    The condition that $u_0 \in \sobolev{m}{\real}[m, \mathcal H_x]$ is quite restrictive.
    It requires in particular that $u_0 \in \lp{2}{\real}$,
    which is equivalent to requiring that $\rho_0 \in \lp{2}{\real}[\rho_s^{-1}]$,
    because $u_0 = \rho_0 / \sqrt{\rho_s}$ by definition.
    Though natural from an $L^2$-theory perspective, see~\cite[Sec. 4.5]{pavliotis2011applied} and~\cite{MR1812873},
    this condition excludes a large class of initial conditions.
    If $V(x)$ behaves as $|x|^4$ as $|x| \to \infty$,
    then it excludes Gaussian initial conditions, for example.
\end{remark}
}

\iflong
\begin{remark}
An alternative manner of solving \cref{eq:fokker_planck_white_noise} numerically
is to approximate the eigenvalues and eigenfunctions of the operator $\mathcal H_x$
in terms of Hermite functions,
after which integration in time becomes trivial.
This approach requires finding (or approximating) the $d + 1$ solutions $(\varphi_d^i, \lambda_d^i)$,
$\varphi_d^i \in \e^{-V_q/2} \, \poly(d)$, $0 \leq i \leq d$,
of the eigenvalue problem
\begin{equation}
    \label{eq:eigen_approximation_white_noise}
    \ip{\mathcal H_x \, \varphi_d^i}{w_d} = \lambda_d^i \ip{\varphi_d^i}{w_d} \quad \forall w_d \in \e^{-V_q/2} \, \poly(d).
\end{equation}
Once these have been calculated,
an approximation of the solution to~\cref{eq:schrodinger_equation} is obtained simply as
\begin{equation}
    u_d(t) = \sum_{i=0}^{d} c_{0i} \, \varphi_d^i \, \e^{\lambda^i_d \, t}, \quad c_{0i} = \ip{u_0}{\varphi_d^i}_{\hat d}.
\end{equation}
It is readily seen that $u_d$,
defined by this equation,
is also the unique solution of~\eqref{eq:galerkin_approximation_white_noise}.
This is the approach taken in~\cite{goudonefficient}
for the calculation of drift and diffusion coefficients
in the diffusion approximation of kinetic equations.
While equivalent,
this methodology is more computationally expensive
because it requires the full solution of the eigenvalue problem~\eqref{eq:eigen_approximation_white_noise}.
\end{remark}
\fi

\subsection{McKean--Vlasov equation with white noise}%
In the presence of an interaction term,
the Fokker--Planck equation becomes nonlinear:
\begin{equation}
    \label{eq:fokker_planck_white_noise_nonlinear}
    \derivative{1}[\rho]{t} = \derivative{1}{x} \left(\derivative*{1}[V]{x} \,\rho + \theta (x - m(t)) \, \rho +  \beta^{-1} \derivative{1}[\rho]{x} \right) =: (\mathcal L_x^m)^* \rho , \quad m(t) = \int_{\real} x \, \rho \, \d x.
\end{equation}
For this equation the weighted $\lp{2}{\real}[\e^{V}]$ energy estimate of the linear case~\eqref{eq:proof:energy_estimate_linear} does not hold,
and there is therefore no longer a natural space for the Galerkin approximation.
\revision{
    Because of this,
    and since we would like to employ the spectral numerical method with Gaussian initial conditions,
    which is not possible with a variational formulation of the type~\eqref{eq:galerkin_approximation_white_noise_both} in view of~\cref{remark:initial_condition},
    we will use Hermite functions to approximate the solution to \cref{eq:fokker_planck_white_noise_nonlinear} directly,
    i.e.\ we will look for an approximate solution in the space $\e^{-V_q/2} \, \poly(d)$.
}
The variational formulation corresponding to the Galerkin approximation is then
to find $\rho \in \e^{-V_q/2} \, \poly(d)$ such that
\begin{subequations}
\label{eq:galerkin_approximation_white_noise_nonlinear}
\begin{alignat}{2}
    \ip{\derivative{1}[\rho_d]{t}}{w_d} &= \ip{(\mathcal L_x^{m_d})^* \rho_d}{w_d} \qquad  & \forall w_d \in \e^{-V_q/2} \, \poly(d), \\
    \label{eq:galerkin_approximation_self_consistency}
    \quad m_d &= \frac{\ip{x}{\rho_d}[\hat d]}{\ip{1}{\rho_d}[\hat d]} \approx \frac{\int_{\real} x \, \rho_d \, \d x}{\int_{\real} \rho_d \, \d x}, \\
    \label{eq:galerkin_approximation_white_noise_nonlinear_initial_condition}
    \ip{\rho_d(0)}{w_d} &= \ip{\rho_0}{w_d}[\hat d] \qquad & \forall w_d  \in \e^{-V_q/2} \, \poly(d).
\end{alignat}
\end{subequations}
Dividing by $\ip{1}{\rho_d}[\hat d]$ in \cref{eq:galerkin_approximation_self_consistency} is useful to account for changes in the total mass of $\rho_d$,
which can compromise the accuracy of the method when $d$ is low,
but doing so becomes unnecessary for large enough $d$.
In contrast with the operator $\mathcal H_x$ in \cref{eq:galerkin_approximation_white_noise},
the operator $(\mathcal L_x^{m_d})^*$ is not selfadjoint in $\lp{2}{\real}$,
and therefore the associated stiffness matrix is not symmetric.
In addition, the quadratic form $\ip{(\mathcal L_x^m)^* \cdot}{\cdot}$ is not necessarily negative for the usual $\lp{2}{\real}$ inner product,
and indeed we observe numerically that the eigenvalue with smallest real part of the discrete operator is often negative,
although small when $d$ is large enough.
\iflong
This is illustrated in~\cref{fig:eigen_values_of_smallest_real_part}
for the same parameters as in the subsequent convergence study.
\begin{figure}[htpb]
    \begin{minipage}[b]{.47\linewidth}
        \centering%
        \includegraphics[width=0.99\linewidth]{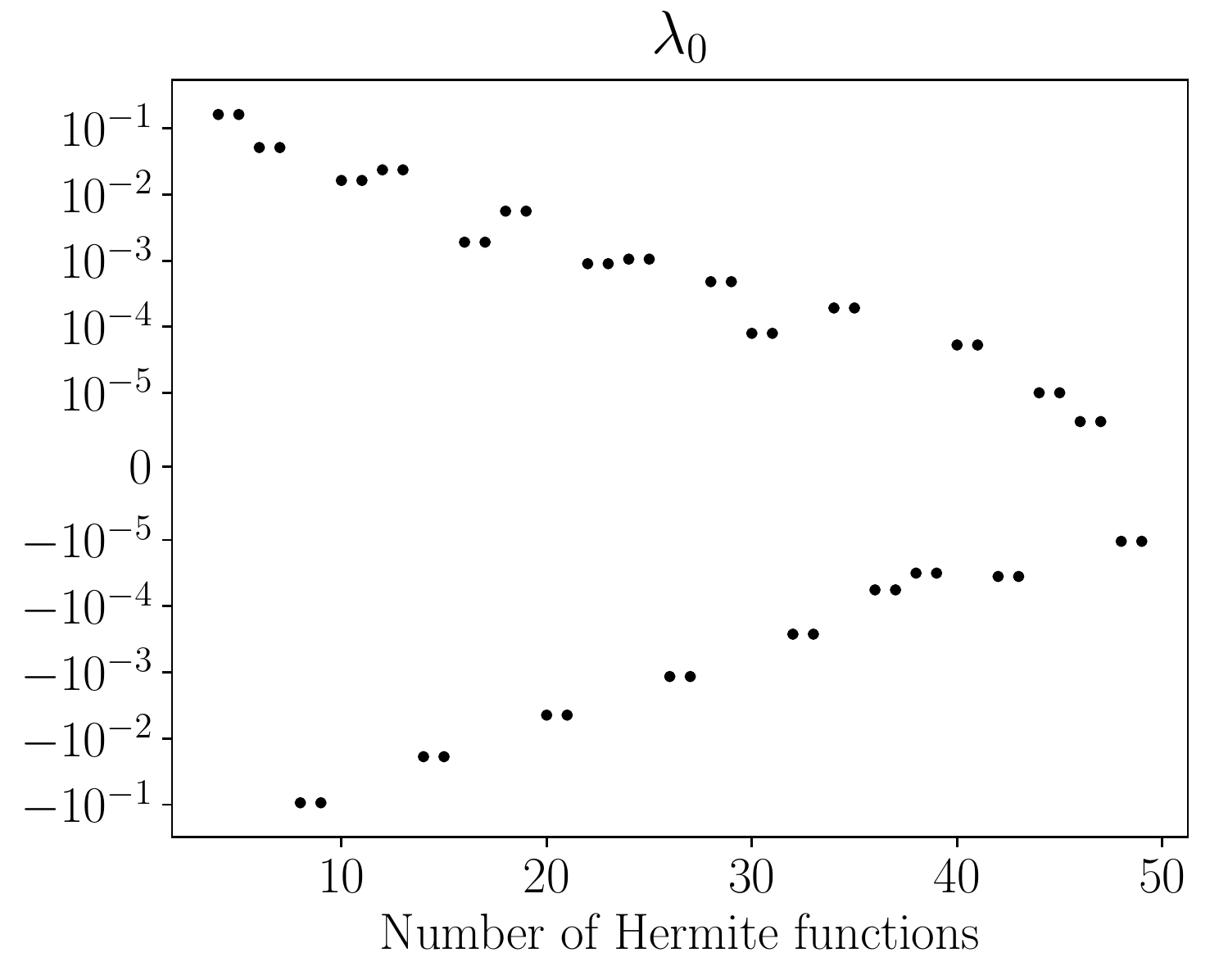}
        \subcaption{%
            \scriptsize Eigenvalue with smallest real part of $-\hat \Pi_d \, (\mathcal L_x^m)^*\, \hat \Pi_d$ for $m = 0$.
            A mixed scale, linear in the interval $[- 10^{-5}, 10^{-5}]$ and logarithmic elsewhere, is used for the $y$ axis.
        }%
        \label{fig:eigen_values_of_smallest_real_part}
    \end{minipage}\hspace{.05\linewidth}%
    \begin{minipage}[b]{.47\linewidth}
        \centering%
        \includegraphics[width=0.95\linewidth]{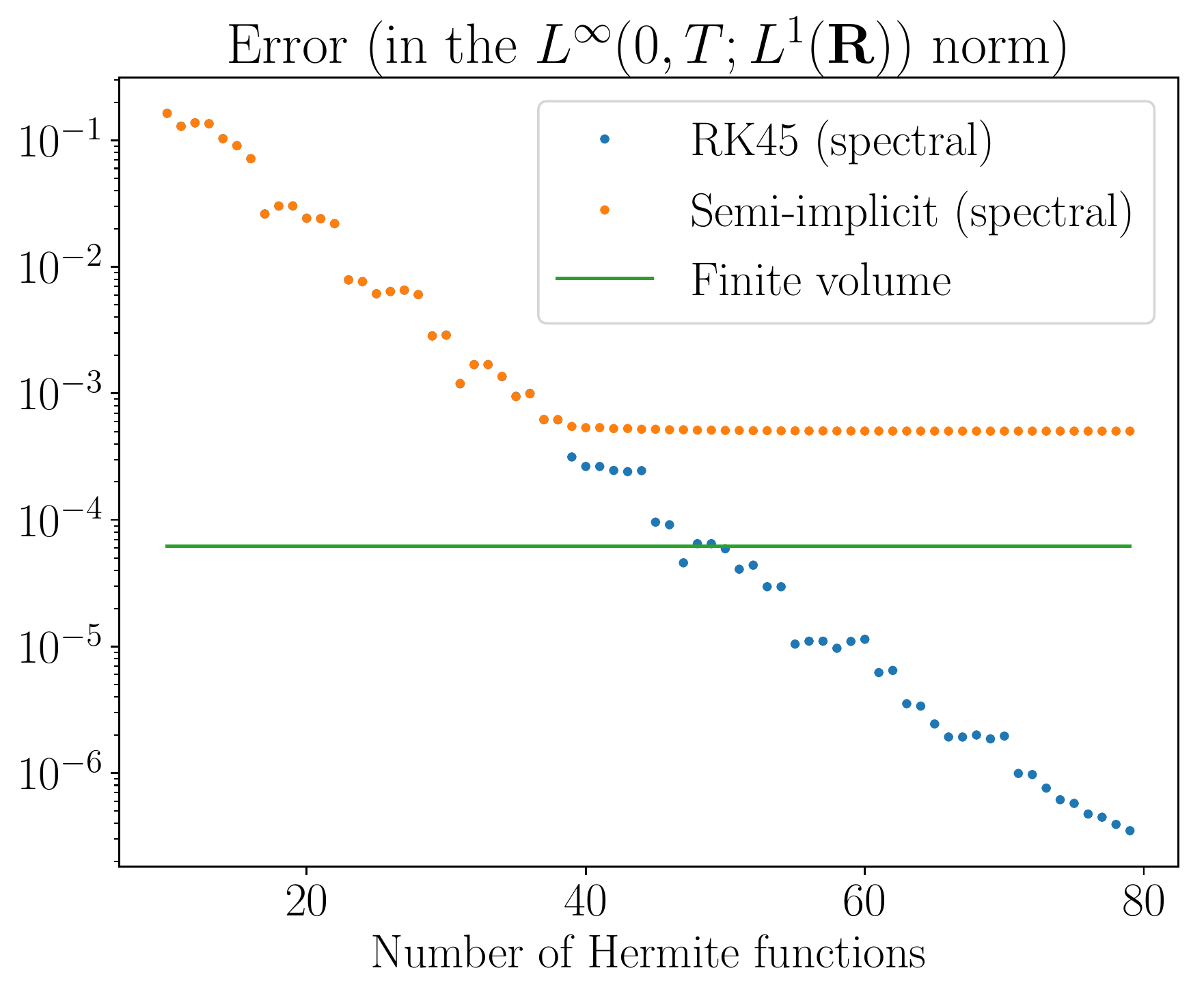}
        \subcaption{%
            \scriptsize Error of the Hermite Galerkin discretization,
            with either RK45 or the semi-implicit method~\eqref{eq:galerkin_approximation_white_noise_nonlinear_time_stepping},
            against degree of approximation ($d$),
            and comparison with the error of the finite volume scheme from~\cite{carrillo2015finite}.
        }%
        \label{fig:error_convergence_to_carrillo}
    \end{minipage}%
    \caption{%
        Study of the Galerkin approximation~\eqref{eq:galerkin_approximation_white_noise_nonlinear}.
    }
\end{figure}
\fi

For the integration in time,
we used either the RK45 method (using the \verb?solve_ivp? method from the \emph{SciPy integrate} module),
or a linear semi-implicit method obtained by
treating $m_d$ explicitly and the other terms implicitly at each time step.
The former is most useful when an accurate time-dependent solution is required,
while the latter enables the use of larger time steps
and is therefore more convenient when only the steady-state solution is sought,
as will be the case for the construction of bifurcation diagrams.
Denoting the time step by $\Delta t$ and the Galerkin approximation of $\rho_d(n \, \Delta t)$ by $\rho_d^n$,
the semi-implicit method is based on obtaining $\rho_d^{n+1}$ by solving:
\begin{subequations}
\label{eq:galerkin_approximation_white_noise_nonlinear_time_stepping}
\begin{align}
    \ip{\rho_d^{n+1} - \rho_d^n}{w_d} &= \Delta t \, \ip{(\mathcal L_x^{m_d^n})^* \rho_d^{n+1}}{w_d} \qquad \forall w_d \in \e^{-V_q/2} \, \poly(d), \\
    \quad m_d^{n+1} &= \frac{\ip{x}{\rho_d^{n+1}}[\hat d]}{\ip{1}{\rho_d^{n+1}}[\hat d]}.
\end{align}
\end{subequations}
\iflong
\paragraph{Convergence study}%
\label{par:convergence_study}
The analysis of the Hermite spectral method for general types of McKean--Vlasov equations will not be presented here.
For the purposes of this work,
it will be sufficient to present a detailed numerical study of the convergence of the method.
To study empirically the validity
of the Galerkin method~\eqref{eq:galerkin_approximation_white_noise_nonlinear}
and of the associated time-stepping scheme~\eqref{eq:galerkin_approximation_white_noise_nonlinear_time_stepping},
we compare our method with the positivity preserving, entropy decreasing finite volume method proposed in~\cite{carrillo2015finite}
for nonlinear, nonlocal gradient PDEs\footnote{%
    We reiterate the fact that one of the main advantages of our numerical method is that it \emph{does not} require that the PDE has a gradient flow structure.
}.
The parameters used here are $\beta = 3$, $\theta = 1$,
and the initial condition was the Gaussian $\mathcal N(10^{-1}, 1)$.
The same time points were used for the finite volume method and semi-implicit Galerkin method
(with a mean time step of approximately 0.002),
and for RK45 the absolute and relative tolerances were both set to $10^{-11}$.
For the finite volume method,
600 equidistant mesh points were used between $x = -6$ and $x = 6$.

\Cref{fig:error_convergence_to_carrillo}
presents the $\lp{\infty}{0,T;\lp{1}{\real}}$ norm of the errors associated with the solutions obtained,
for values of $d$,
the degree of Hermite polynomials used,
ranging from $10$ to $80$.
A very accurate solution,
obtained by using our spectral method with $d = 120$,
was employed for the calculation of the errors.
We observe that, as $d$ increases initially,
the solutions obtained using the semi-implicit~\eqref{eq:galerkin_approximation_white_noise_nonlinear_time_stepping} and the RK45 methods
are indistinguishable and converge exponentially fast.
From $d \approx 40$,
the accuracy of the semi-implicit method no longer improves,
indicating that the error introduced by the time-stepping scheme dominates from that point on.
From $d \approx 50$,
the Galerkin/RK45 approximation becomes more precise than the finite volume method.
We therefore conclude that an accuracy as good as
that obtained using the finite volume scheme
can be reached with roughly \emph{ten times fewer unknowns}
using the spectral discretization~\eqref{eq:galerkin_approximation_white_noise_nonlinear}.
Our spectral method also enjoys a low computational cost:
it ran in only about a minute with a Intel i7-3770 processor,
even for a value of $d$ as high as 80,
whereas the finite volume simulation took over an hour.

\Cref{fig:comparison_carrillo_solutions_in_time}
presents snapshots of the solutions at different times.
We observe that,
although the number of Hermite functions employed in the expansion is relatively low (=25),
the solutions are in extremely good agreement.
\begin{figure}[ht]
    \centering
    \includegraphics[width=0.32\linewidth]{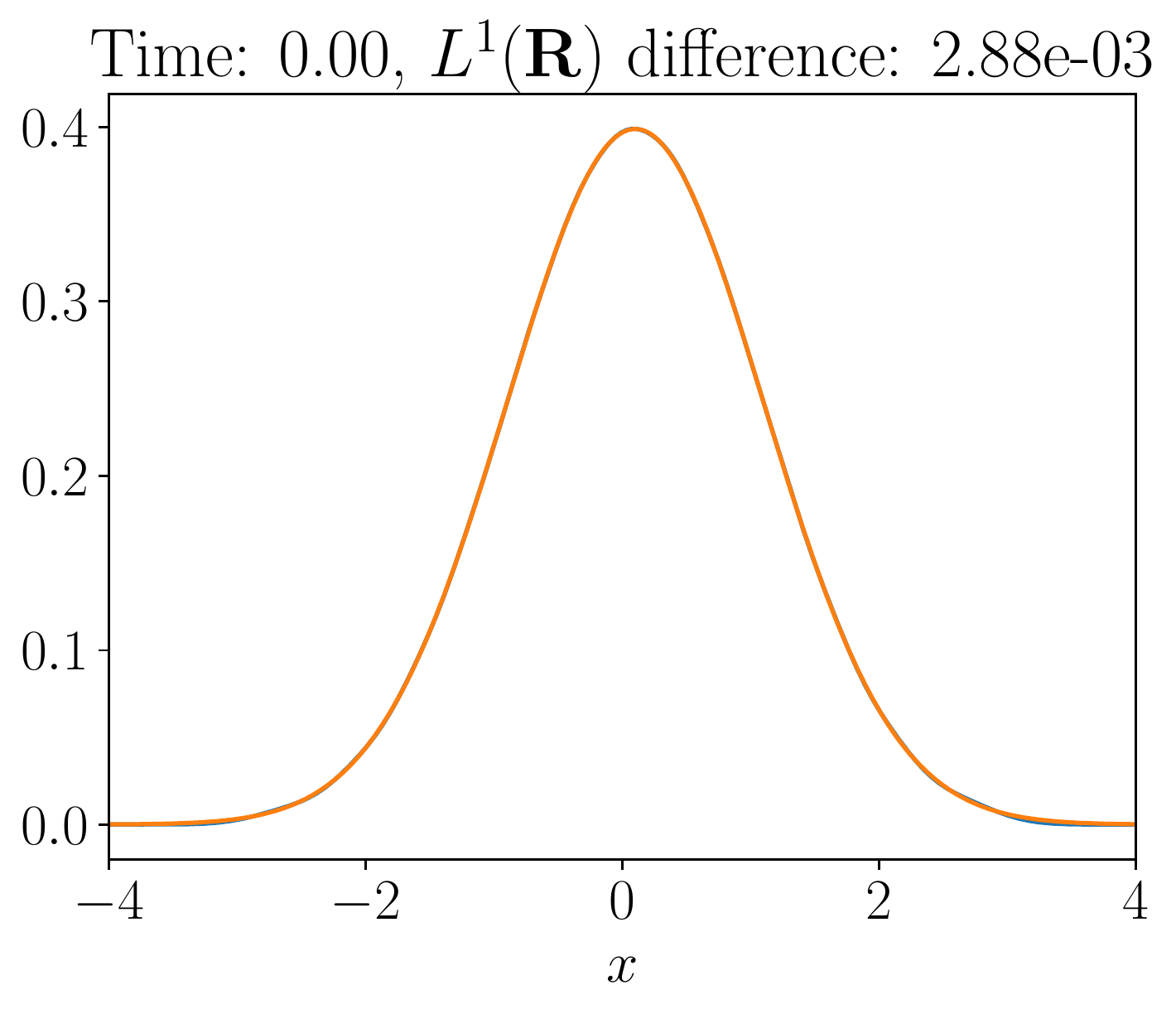}
    \includegraphics[width=0.32\linewidth]{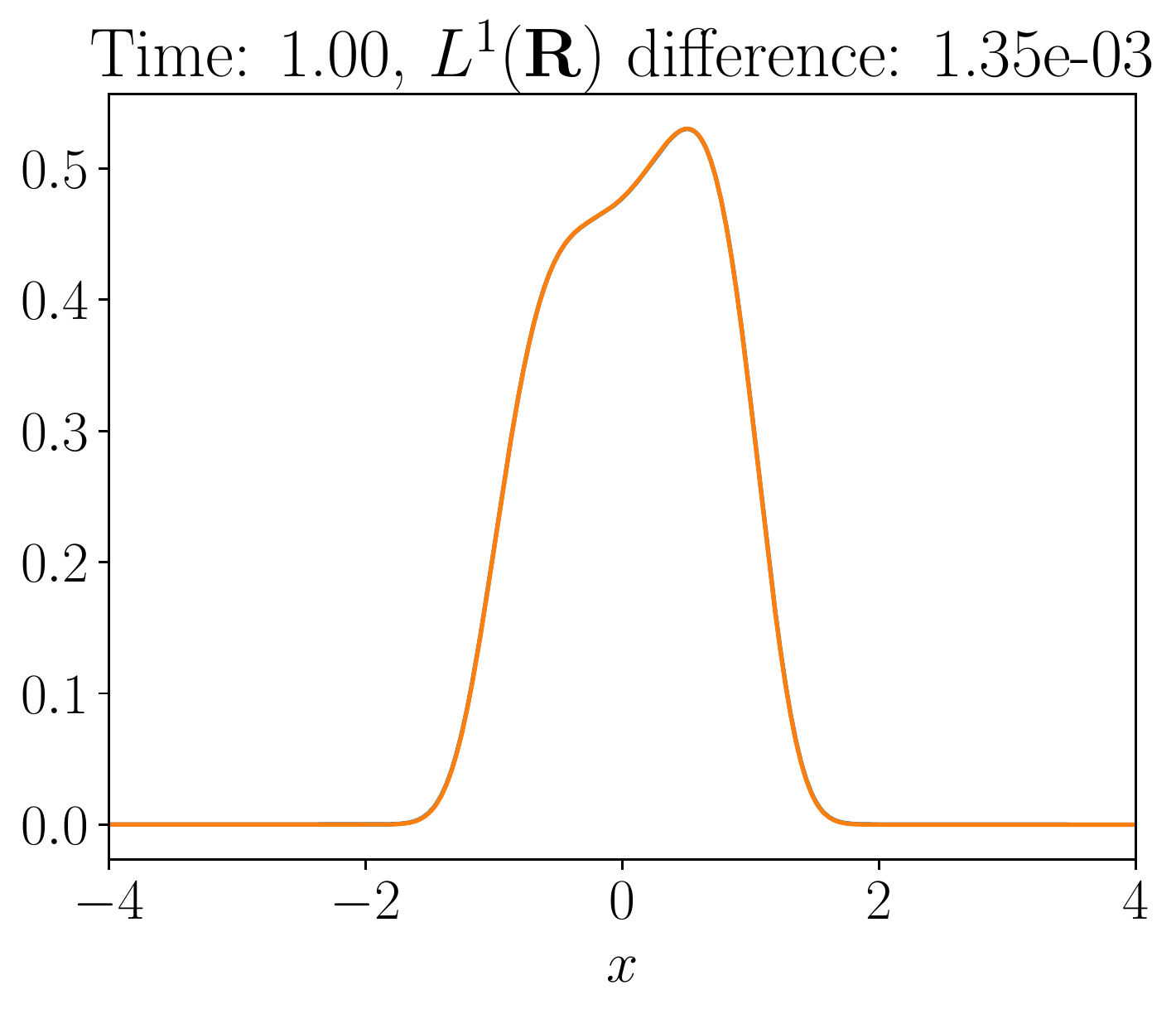}
    \includegraphics[width=0.32\linewidth]{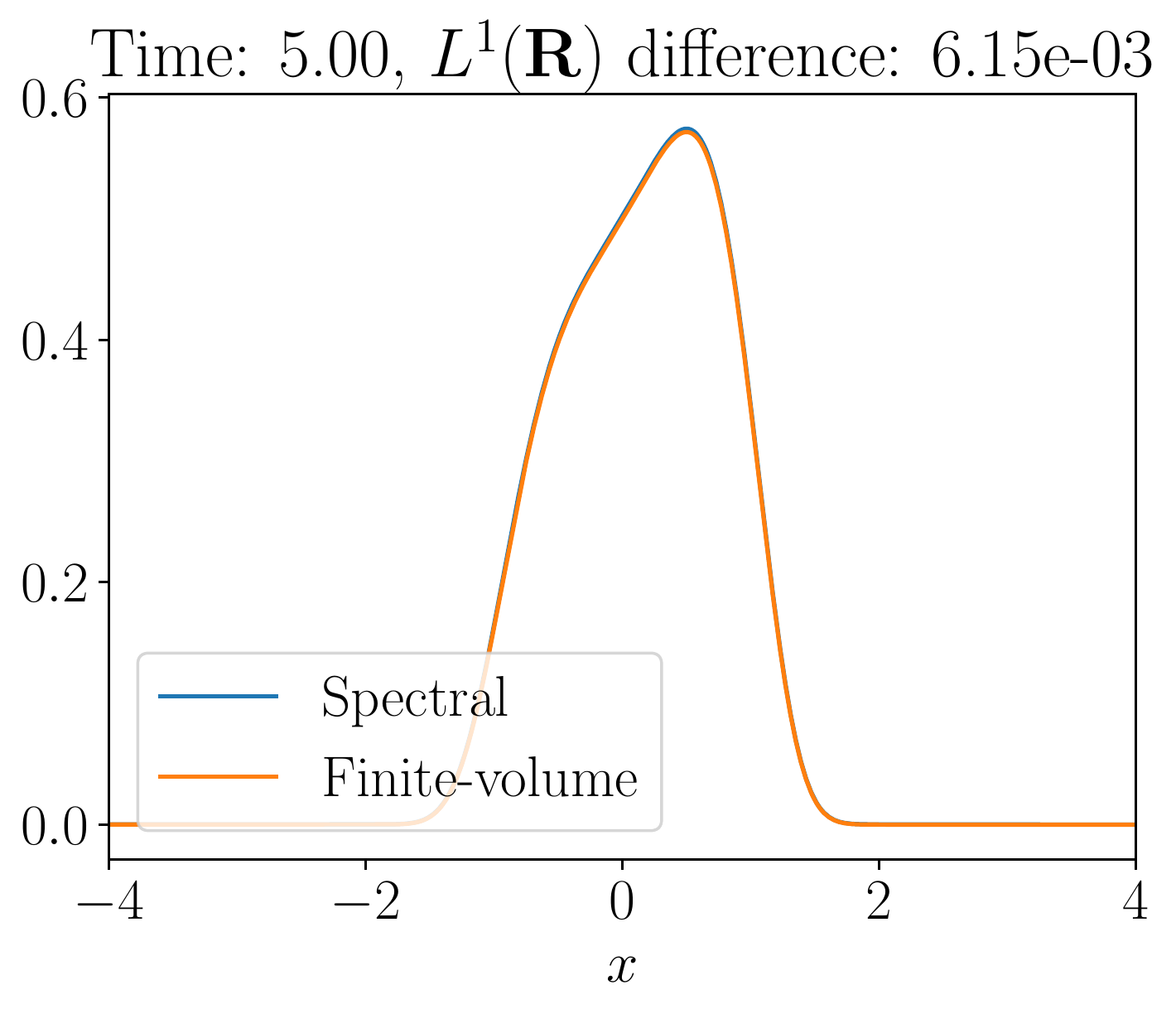}
    \caption{%
        Snapshots of the solution to~\eqref{eq:fokker_planck_white_noise_nonlinear} using either the finite volume method from~\cite{carrillo2015finite}
        or the Galerkin approximation~\eqref{eq:galerkin_approximation_white_noise_nonlinear} with the RK45 method and 25 Hermite functions.
    }
    \label{fig:comparison_carrillo_solutions_in_time}
\end{figure}

In the simulations presented in this section,
the scaling factor was set to $\sigma^2 = \frac{1}{10}$.
As discussed in
\iflong
\cref{sec:hermite_polynomials},
\else
\cite[Appendix A]{2019arXiv190405973G}
\fi
choosing this factor appropriately can significantly improve the accuracy of the method.
In particular,
given that the solution to~\cref{eq:fokker_planck_white_noise_nonlinear} decreases rapidly as $\abs{x} \to \infty$,
$\sigma$ should decrease with $d$,
with the optimal scaling being $\sigma \propto \sqrt{d}$,
as demonstrated in~\cite{tang1993hermite}.
For convergence studies, however,
it is convenient to use a fixed $\sigma$,
first because
this is assumed by most convergence results (such as \cref{thm:convergence_as_d_goes_to_infinity})
and, second,
because this simplifies the calculation of the matrices involved in the Galerkin formulation
(only the last row and the last column have to be calculated upon incrementing $d$).
\fi

\begin{remark}
    [Computational considerations]
Discretizing the operators appearing in the Galerkin approximations~\eqref{eq:galerkin_approximation_white_noise} and~\eqref{eq:galerkin_approximation_white_noise_nonlinear}
requires the calculation of multiple matrices corresponding to operators of the type $\Pi_d \, \left(f \, \partial_x \right) \Pi_d$,
where $f$ is a polynomial and $\Pi_d$ is the $\lp{2}{\real}[\e^{-V_q}]$ projection operator onto $\poly(d)$.
These calculations can be carried out by noticing that
\begin{equation*}
    \Pi_d \, \left(f \, \derivative*{m}{x^m} \right) \Pi_d = \left(\Pi_d \, f \, \Pi_d \right) \,  \left( \Pi_d \derivative*{m}{x^m} \Pi_d\right).
\end{equation*}
The matrix representation of the first operator on the right-hand side,
in a basis of Hermite polynomials,
can be obtained from the Hermite transform of $f$.
The matrix representation of the second operator,
on the other hand,
is a matrix with zero entries everywhere except on the $m$-th superdiagonal,
in view of the recursion
\iflong
relation~\eqref{eq:hermite_polynomials_1d_recursion_derivative}.
\else
relation~(A.1) in~\cite{2019arXiv190405973G}.
\fi
\end{remark}
\subsection{Linear Fokker--Planck equation with colored noise}%
\label{sub:fokker_planck_equation_with_colored_noise_case}

In this section,
we turn our attention to the case of Gaussian or non-Gaussian colored noise given in terms of overdamped Langevin dynamics.
The case of harmonic noise can be treated in a similar fashion,
and for conciseness we do not present the associated Galerkin formulation explicitly here.
We start by considering the linear (without the interaction term)
Fokker--Planck equation with colored noise:
\begin{equation}
    \label{eq:fokker_planck_colored_noise}
    \derivative{1}[\rho]{t} = \derivative{1}{x} \left( \derivative{1}[V]{x}\,\rho - \frac{\zeta}{\varepsilon} \, \sqrt{2 \, \beta^{-1}} \eta \, \rho \right)
    + \frac{1}{\varepsilon^{2}}\,\derivative{1}{\eta} \, \left( V_{\eta}'  \, \rho  + \derivative{1}[\rho]{\eta} \right) =: \mathcal L_{\varepsilon}^* \, \rho.
\end{equation}
We recall that $\varepsilon^2$ controls the correlation time of the colored noise and
$\zeta$ is a parameter such that the white noise limit is recovered (with inverse temperature $\beta$) when $\varepsilon \to 0$.
We include $\varepsilon$ in \cref{eq:fokker_planck_colored_noise} because,
although we do not consider the white noise limit in this section,
large values of $\varepsilon$ are in general more difficult to tackle numerically,
and it will be therefore convenient to use smaller correlation times in the numerical experiments below.
The problem is now two-dimensional and
the operator on the right-hand side of \cref{eq:fokker_planck_colored_noise} is no longer elliptic.
In contrast with the white noise case,
there does not exist an explicit formula for the steady-state solution for \cref{eq:fokker_planck_colored_noise}.

The procedure for obtaining a Galerkin formulation is the same as in \cref{sub:white_noise_case},
except that we now use tensorized Hermite polynomials/functions.
To retain some generality,
we will consider that the Galerkin approximation space is of the form $S_d = \e^{- U(x, \eta)/2} \, \e^{-V_q(x, \eta)/2} \, \poly(\mathcal I_d)$
for some function $U: \real^2 \mapsto \real$, a nondegenerate quadratic potential $V_q$ to be determined,
and where $\poly(\mathcal I_d) := \Span \left\{ x^{\alpha_x} \, \eta^{\alpha_{\eta}}: (\alpha_x, \alpha_{\eta}) \in \mathcal I_d \right\}$
for some index set $\mathcal I_d \subset \nat^2$ that grows with $d \in \nat$.
Compared to the one-dimensional case,
there are now two scaling parameters, $V_q := x^2/2\sigma_x^2 + \eta^2/2\sigma_\eta^2$.
The Galerkin approximation we propose consists in finding $\rho_d \in S_d$ such that
\begin{align}
    \label{eq:galerkin_approximation_colored_noise}
    \ip{\derivative{1}[\rho_d]{t}}{w_d}[\e^{U}] =& \ip{\mathcal L_{\varepsilon}^* \, \rho_d}{w_d}[\e^{U}] \qquad \forall w_d \in S_d, \quad \forall t > 0,
\end{align}
with appropriate initial conditions.
The choice of the weight $\e^U$ in the inner products of~\cref{eq:galerkin_approximation_colored_noise} is motivated by the fact that
differential operators admit sparse representations in the Hermite-type basis naturally associated with $S_d$,
and we note that $\e^{- U(x, \eta)/2} \, \e^{-V_q(x, \eta)/2} \, \poly(\nat^2)$,
where $\poly(\nat^2)$ is the space of polynomials in two dimensions,
is dense in $\lp{2}{\real^2}[\e^U]$.
In practice, we obtain $\rho_d$ as $\e^{- U(x, \eta)/2} \, \e^{-V_q(x, \eta)/2} \, v_d$,
where $v_d$ is obtained by solving
\begin{align}
    \label{eq:galerkin_approximation_colored_noise_after_mapping}
    \ip{\derivative{1}[v_d]{t}}{w_d}[\e^{-V_q}] =& \ip{\mathcal H_{\varepsilon} \, v_d}{w_d}[\e^{-V_q}] \qquad \forall w_d \in \poly(\mathcal I_d), \quad \forall t > 0,
\end{align}
where, \revision{for a test function $\varphi$,
$\mathcal H_{\varepsilon} \varphi := (\e^{U/2} \, \e^{V_q/2})\mathcal L_{\varepsilon}^* \, (\e^{- U/2} \, \e^{-V_q/2} \, \varphi)$},
and the basis functions used for \cref{eq:galerkin_approximation_colored_noise_after_mapping} are
Hermite polynomials orthonormal with respect to the Gaussian weight $\e^{-V_q}$.
Regarding the index set,
several choices are possible,
with the simplest ones being the triangle $\{\alpha \in \nat^2: \abs{\alpha}_1 \leq d \}$
and the square $\{\alpha \in \nat^2: \abs{\alpha}_{\infty} \leq d \}$,
see \cref{fig:solution_density_gaussian_case,fig:solution_density_bistable_case} below.
\iflong
We demonstrate in \cref{ssub:asymptotic_analysis_for_the_galerkin_formulation} that,
\else
It was demonstrated in~\cite{thesis_urbain} that,
\fi
in order to study the limit $\varepsilon \to 0$,
a rectangle-shaped index set is usually the only suitable choice.
When studying the behavior as $d$ increases,
however, we observed spectral convergence irrespectively of the index set utilized.

Clearly, it is necessary that $\rho \in \lp{2}{\real^2}[\e^U]$ for the Galerkin discretization~\eqref{eq:galerkin_approximation_colored_noise} to produce good results.
Since the $1/\varepsilon^2$ part of the operator on the right-hand side of \cref{eq:fokker_planck_colored_noise},
\revision{$\mathcal L_0^* \dummy = \derivative{1}{\eta} (V_{\eta}'(\eta) \dummy + \, \derivative{1}{\eta} \dummy)$},
is selfadjoint in $\lp{2}{\real}[\e^{-U_x(x)/2 - V_{\eta}(\eta)/2}]$
for any choice of $U_x$,
it is natural to choose $\e^{-U(x, \eta)/2} = \e^{-U_x(x)/2 - V_{\eta}(\eta)/2}$
for some one-dimensional potential $U_x$.
This guarantees that the matrix representation of $\mathcal L_0^*$ is symmetric and negative semi-definite,
but this is not a requirement.

\revision{%
    The performance of the Galerkin approximation~\eqref{eq:galerkin_approximation_colored_noise} is investigated through numerical experiments in \cref{sec:numerical_tests}.
    An asymptotic analysis of the numerical method in the limit as $\varepsilon \to 0$ is presented in~\cite{2019arXiv190405973G},
    which is a longer version of this paper.
}
\subsection{McKean--Vlasov equation with colored noise}%
We consider now the nonlinear McKean--Vlasov initial value problem with OU noise:
recalling that $\zeta = 1/\sqrt{2}$ in this case,
\begin{subequations}
\label{eq:mckean_vlasov_colored_noise}
\begin{align}
    &\derivative{1}[\rho]{t} = \derivative{1}{x} \left( \derivative{1}[V]{x} \, \rho + \theta \, (x - m(t)) \, \rho - \frac{1}{\varepsilon} \, \sqrt{\beta^{-1}} \, \eta \, \rho \right)
    + \frac{1}{\varepsilon^{2}}\,\derivative{1}{\eta} \, \left( \eta  \, \rho  + \derivative{1}[\rho]{\eta} \right), \\
    &m(t) = \int_{\real} \int_{\real} x \, \rho(x, \eta, t) \, \d \eta \, \d x, \\
    &\rho(x, \eta, t=0) = \rho_0 (x, \eta),
\end{align}
\end{subequations}
for some initial distribution $\rho_0(x, \eta)$ such that the noise is not necessarily started at stationarity.
The method that we use in this case,
which applies \emph{mutatis mutandis} to the other noise models,
is the same as in~\cref{eq:galerkin_approximation_colored_noise},
with the addition of the interaction term,
and we use the same time-stepping schemes as in~\cref{sub:white_noise_case}.

\revision{%
    Numerical experiments, testing the convergence of the method for the two time-stepping schemes,
    are presented in \cref{sec:numerical_tests}.
}

\subsection{Monte Carlo simulations}
We will compare the bifurcation diagrams obtained using the spectral method described above to those obtained by direct MC simulations of the system of interacting particles~\eqref{eq:main-sde}.
We use the Euler--Maruyama method:
\[
X_{k+1}^i = X_{k}^i - V'(X_k^i) \, \Delta t - \theta\left(X_k^i - \frac{1}{N}\sum_{j=1}^NX_k^j\right) \, \Delta t + \frac{\zeta}{\varepsilon} \, \sqrt{2\beta^{-1}} \, \eta_k^i \, \Delta t,
\]
where $\eta_k^i$ is the appropriate projection of the stochastic process $\vect{Y}_t$.
In the case of Gaussian noise, this is discretized as follows
\[
\vect{Y}_{k+1}^i =  \vect{Y}_{k}^i + \frac{1}{\varepsilon^2} A \, \vect{Y}_k^i \, \Delta t + \frac{1}{\varepsilon} \, \sqrt{2 \ \Delta t} \, D \xi,
\]
where $\xi\sim N(0,1)$, and $X_k$, $\vect{Y}_k$ and $\eta_k$ are the approximations to $X(k\Delta t), \, \vect{Y}(k\Delta t)$ and $\eta(k\Delta t)$, respectively.
The time step used was always $\mathcal O(\varepsilon^2)$,
to ensure the accurate solution of the equation.
This scheme has \revision{weak} order of convergence one, see~\cite{MR1872387,MR1949404},
and we find that we capture the correct behavior as long as the time step is sufficiently small.

\iflong
\section{Asymptotic analysis for the Galerkin formulation}%
\label{ssub:asymptotic_analysis_for_the_galerkin_formulation}
In \cref{sec:results},
we will construct bifurcation diagrams of $m$ as a function of $\beta$ for different values of $\varepsilon$,
and we will verify that the bifurcation diagram of the white noise case is recovered when $\varepsilon \to 0$.
Since the spectral method presented in \cref{sec:numerics} will be used to that purpose,
it is useful to study the behavior of the solution to the Galerkin formulation~\eqref{eq:galerkin_approximation_colored_noise}
in the limit $\varepsilon \to 0$,
which is the purpose of this section.
We will then confirm numerically the rates of convergence to the white noise limit presented in \cref{sub:the_white_noise_limit},
i.e.\ $\mathcal O(\varepsilon^2)$ for Ornstein--Uhlenbeck noise
and $\mathcal O(\varepsilon^4)$ for harmonic noise.

For simplicity,
we confine ourselves for the analysis to the case where the noise process is one-dimensional
and the weight function $\e^{-U(x, \eta)/2}$ can be decomposed as $\e^{-U(x, \eta)/2} = \e^{-U_x(x)/2} \, \e^{-V_{\eta}(\eta)/2}$.
As before,
$\hat \Pi_d$ denotes the $\lp{2}{\real^2}[\e^U]$ projection operator on the space of Hermite functions (with appropriate scalings).
Decomposing the operator $(\hat \Pi_d \, \mathcal L_{\varepsilon}^* \, \hat \Pi_d)$ in \cref{eq:galerkin_approximation_colored_noise}
in powers of $\varepsilon$,
we obtain the equation
\begin{subequations}
\label{eq:finite_dimensional_problem}
\begin{align*}
    \derivative{1}[\rho_d]{t} =  (\hat \Pi_d \, \mathcal L_{\varepsilon}^* \, \hat \Pi_d) \, \rho_d
                              &= \hat \Pi_d \, \left( \frac{1}{\varepsilon^2}  \mathcal L_0^* + \frac{1}{\varepsilon}  \mathcal L_1^* +  \mathcal L_2^* \right) \hat \Pi_d \, \rho_d, \\
                              &=: \left( \frac{1}{\varepsilon^2} \hat {\mathcal L}_0 + \frac{1}{\varepsilon} \hat {\mathcal L}_1 + \hat {\mathcal L}_2 \right) \rho_d.
\end{align*}
\end{subequations}
As a consequence of the choice of $\e^{U}$,
the largest (sign included) eigenvalue of $\hat {\mathcal L}_0$ is a nonpositive, nondecreasing function of $d$.
Since we cannot expect the leading order term of the discrete generator to have an eigenvalue exactly equal to 0,
we look for a solution of the form $\rho_d =  \e^{- \abs{\lambda_{0,d}} \, t / \varepsilon^2} (\varrho_0 + \varepsilon \, \varrho_1 + \varepsilon^2 \, \varrho_2 + \dotsb)$,
where $\lambda_{0,d}$ is the largest eigenvalue of $\hat {\mathcal L}_0$.
Denoting by $\id$ the identity operator and
gathering equal powers of $\varepsilon$,
we obtain the equations:
\begin{subequations}
\begin{align}
    \label{eq:galerkin_asymptotics:first_equation}
    & 0 = \left( \hat {\mathcal L_0} + \abs{\lambda_{0,d}} \, \id \right) \, \varrho_0, \\
    \label{eq:galerkin_asymptotics:second_equation}
    & 0 = \left( \hat {\mathcal L_0} + \abs{\lambda_{0,d}} \, \id \right) \, \varrho_1 + \hat {\mathcal L_1} \, \varrho_0, \\
    \label{eq:galerkin_asymptotics:third_equation}
    & \derivative{1}[\varrho_i]{t} = \left( \hat {\mathcal L_0} + \abs{\lambda_{0,d}} \, \id \right) \, \varrho_{i+2} + \hat {\mathcal L_1} \, \varrho_{i+1} + \hat {\mathcal L_2} \, \varrho_{i},
    \qquad i = 0, 1, \dotsc
\end{align}
\end{subequations}

\subsection{Suitable index sets}%
\label{sub:suitable_index_sets}
Let $H^x_i$ and $H^{\eta}_j$ denote the
(possibly rescaled) Hermite functions in the $x$ and $\eta$ directions, respectively.
Let also $\mathcal I_{i,\eta}$ be a slice of the index set, $\mathcal I_{i,\eta} := \{j: (i,j) \in \mathcal I_d\}$,
and $\Pi_x \mathcal I_d$ be the projected index set given by $\{i: (\exists j \in \nat) [(i, j) \in \mathcal I_d] \}$.
Expanding $\varrho_0$ in terms of the basis functions used for the Galerkin discretization in the first equation,
we obtain
\begin{align*}
    \label{eq:symmetric_non_gaussian_noise_centering_condition_first_equation}
    0 &= (\hat {\mathcal L}_0 + \abs{\lambda_{0,d}} \id) \, \sum_{(i,j) \in \mathcal I_d} c_{ij} \, \left( \e^{-U_x/2} \, H^{x}_{i} \otimes e^{-V_{\eta}/2} \, H^{\eta}_{j}  \right) \\
      &=  \sum_{i \in \Pi_x \mathcal I_d} \e^{-U_x/2} \, H^{x}_{i} \otimes
      \left(\sum_{j \in \mathcal I_{i,\eta}} c_{ij} \sum_{k \in \mathcal I_{i,\eta}} (L_{jk} + \abs{\lambda_{0,d}} \, \delta_{jk}) \, \e^{-V_{\eta}/2} \, H^{\eta}_{k} \right),
\end{align*}
where $L_{jk} := \ip{\mathcal L_0^* (\e^{-V_{\eta}/2} \,  H^{\eta}_j)}{\e^{-V_{\eta}/2} \,  H^{\eta}_k}[\e^{V_{\eta}}]$.
Since $\{\e^{-U_x/2} \, H^{x}_{i}\}_{i \in \Pi_x \mathcal I_d}$ are linearly independent,
this implies that
\begin{equation}
    \label{eq:symmetric_non_gaussian_noise_constraint_index_set}
    \sum_{j \in \mathcal I_{i,\eta}} c_{ij} \, (L_{jk} + \abs{\lambda_{0,d}} \, \delta_{jk}) = 0, \qquad \forall i \in \Pi_x \mathcal I_d,
\end{equation}
implying that,
for all $\forall i \in \Pi_x \mathcal I_d$,
the vector of coefficients ${(c_{ij})}_{j \in \mathcal I_{i,\eta}}$ is
in the kernel of the matrix ${(L_{jk} + \abs{\lambda_{0,d}})}_{j,k \in \mathcal I_{i,\eta}}$.
Therefore, if for some $i$ this matrix has full rank,
then the corresponding Hermite coefficients must be $0$.
This is of particular relevance when
the eigenfunction in the kernel of $\mathcal L_0^*$ cannot be exactly represented in terms of a finite number of the approximating basis functions,
as is the case with the noise processes {\bf B} and {\bf NS} considered in \cref{sec:model}.

\revision{
Consider for example the case of  a triangular index set.
In this case the matrix $M_i := {(L_{jk} + \abs{\lambda_{0,d}})}_{j,k \in \mathcal I_{i,\eta}}$ is of shape $d + 1-i \times d + 1-i$:
incrementing $i$ by one corresponds to removing the last line last and last column of the matrix.
Consequently, the maximal eigenvalue of $M_i$ (sign included),
which we denote by $\lambda_{\max}(i)$, is a nonincreasing function of $i$
such that $\lambda_{\max}(0) = 0$.
In the case of noise model {\bf B},
it holds that $\lambda_{\max}(i) < 0$ for $i > 0$ or $i > 1$,
depending on whether $d$ is even or odd.
Consequently, in this case~\cref{eq:symmetric_non_gaussian_noise_constraint_index_set} implies that $c_{ij} = 0$ for these values of $i$ and all $j \in \mathcal I_{i,\eta}$.
}

From this we conclude that,
in order to capture the correct solution as $\varepsilon \to 0$,
it is necessary to choose a rectangularly shaped index set,
which is consistent with the fact that,
in the limit $\varepsilon \to 0$,
the solution can be expressed as a tensor product $\rho_{\infty}(x, \eta) = \rho^x(x) \, \rho^{\eta}(\eta)$.

To illustrate the point made in the previous paragraph,
we present side by side in \cref{fig:comparison_index_sets} the results of numerical experiments performed
using either a triangular index set or a square index set,
for the parameters $\varepsilon = 0.01$, $d = 20$, $\beta = 15$, $\theta = 0$, $\sigma_x^2 = \sigma_{\eta}^2 = 1/15$, $\e^{-U_x(x)} = 1$.
While the probability density obtained using a square index set is close to the exact solution and clearly exhibits four local maxima,
the solution obtained using a triangle index set is concentrated around $x = 0$,
and all the associated Hermite coefficients $c_{ij}$ with $i > 0$ are very close to zero.

\begin{figure}[ht]%
    \centering {%
        \begin{minipage}[b]{.8\linewidth}
            \centering%
            \label{fig:symmetric_non_gaussian_index_set_square}
            \includegraphics[width=.47\linewidth]{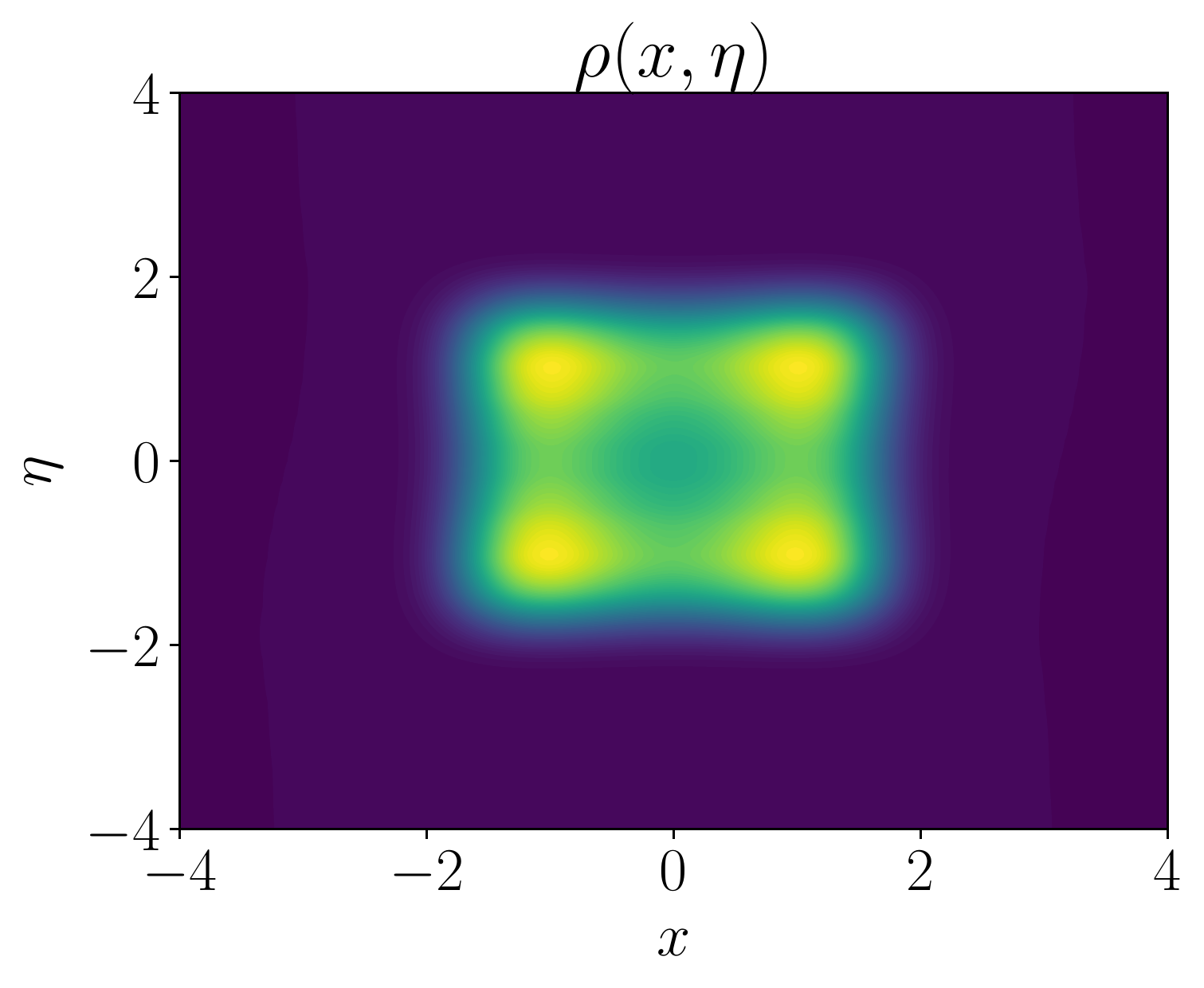}
            \includegraphics[width=.47\linewidth]{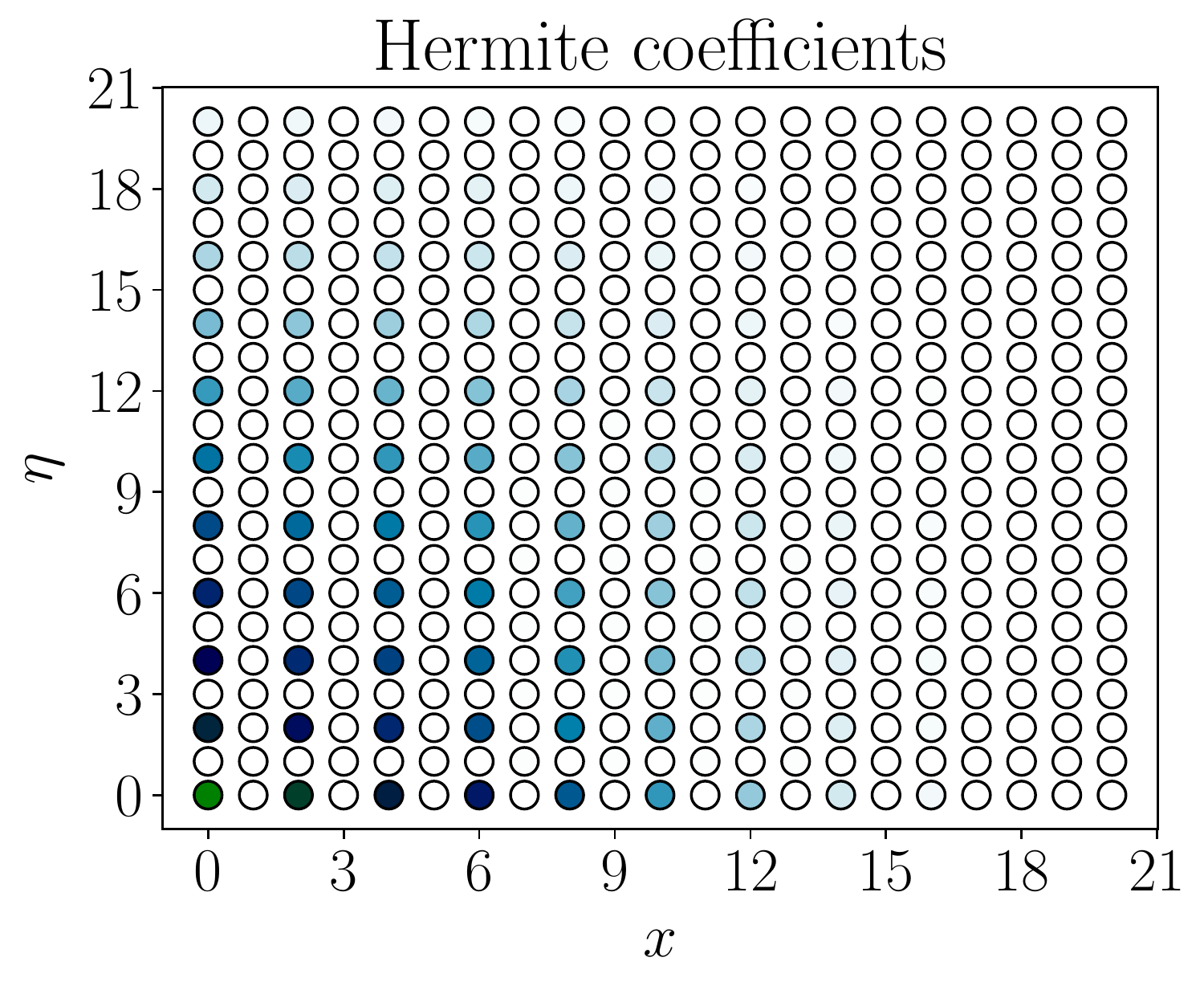}
            \subcaption{%
                Square index set.
            }
        \end{minipage}%

        \begin{minipage}[b]{.8\linewidth}
            \centering%
            \label{fig:symmetric_non_gaussian_index_set_triangle}
            \includegraphics[width=.47\linewidth]{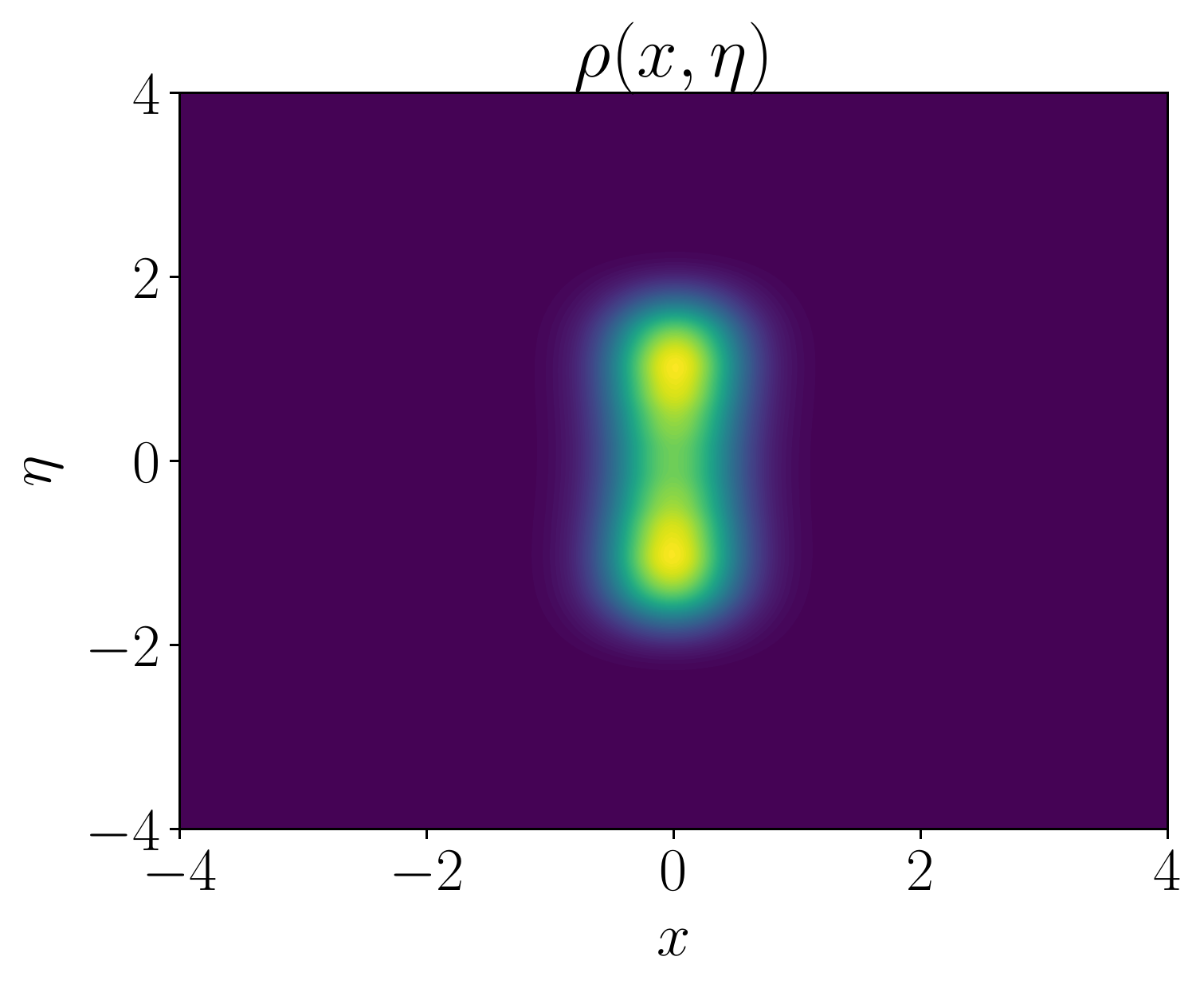}
            \includegraphics[width=.47\linewidth]{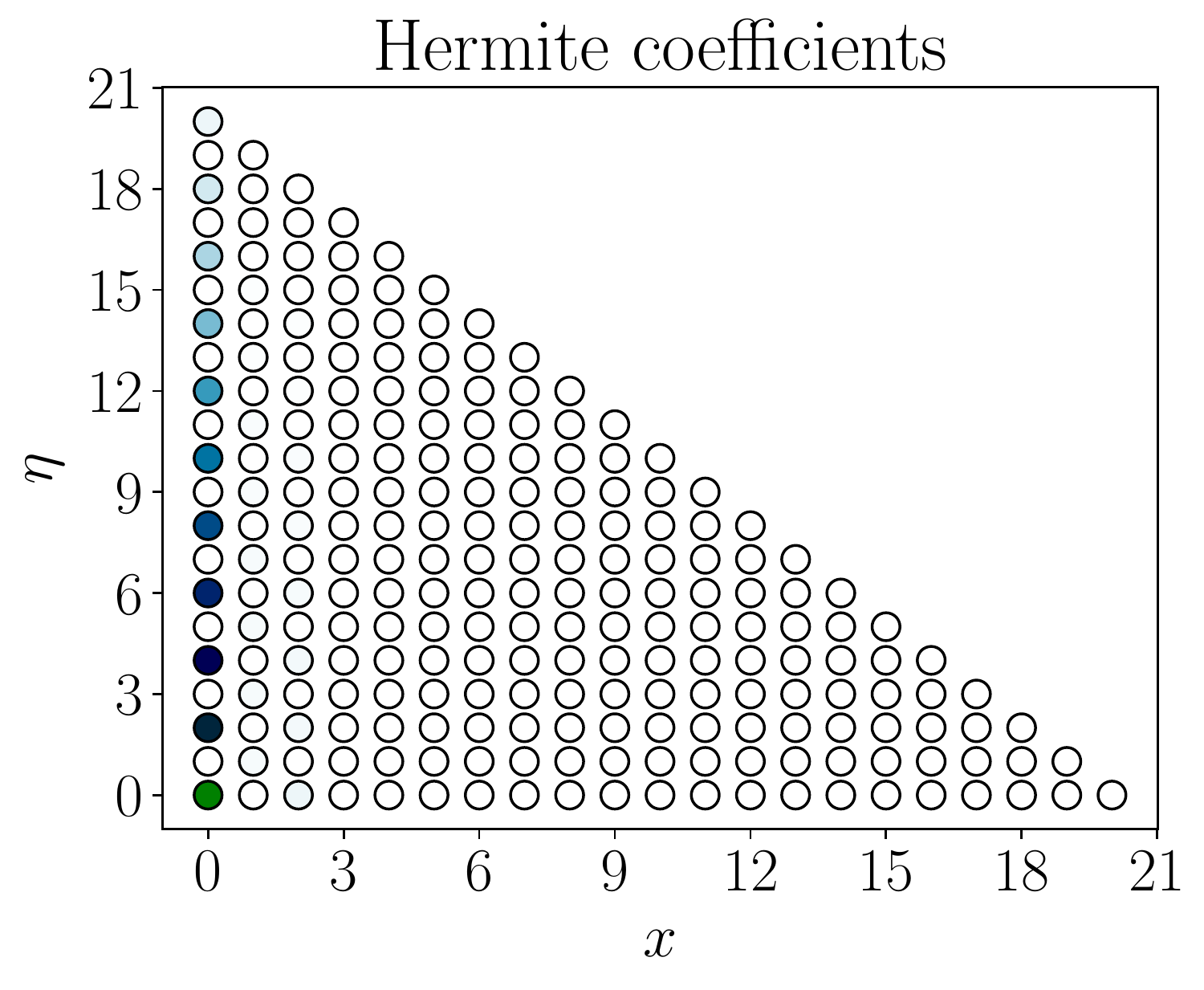}
            \subcaption{%
                Triangle index set.
            }
        \end{minipage}%
    }
    \caption{%
        Comparison of the numerical solutions for the steady-state Fokker--Planck equation with the bistable noise model (model {\bf B} in \cref{sec:model}),
        using either a square index set or a triangle index set.
        While not obvious from the figures, in the former case
        it follows from the fact that $\varepsilon = 0.01 \ll 1$ that the columns of Hermite coefficients in the $\eta$ direction are essentially colinear.
    }%
    \label{fig:comparison_index_sets}
\end{figure}

\subsection{Effective drift and diffusion coefficients}%
\label{sub:parasitic_drift}
We assume from now on that the index-set has a rectangular shape, $\mathcal I_d = \{0, 1, \dotsc, d_x \} \times \{0, 1, \dotsc, d_{\eta} \}$.
At the continuous, infinite dimensional level,
the absence of an effective drift term
for the models of the noise {\bf B} and {\bf NS}
as $\varepsilon \to 0$ is ensured by the centering condition~\eqref{eq:centering_condition_potential}.
Any deviation from zero would lead to an effective drift term,
scaling as $1/\varepsilon$ and proportional to
\begin{equation*}
    \sqrt{2 \, \beta^{-1}} \, \zeta \, \int_{\real} \eta \, \exp \left( -\beta  V_{\eta}(\eta) \right) \, \d \eta.
\end{equation*}
At the finite-dimensional, numerical level,
a parasitic effective drift can arise even when $V_{\eta}$ satisfies~\eqref{eq:centering_condition_potential},
as we demonstrate below.
This can occur when $V_{\eta}$ is not an even function,
and it is especially critical when
the number of basis functions used to approximate the solution in the $\eta$ direction is relatively low,
leading to a nonzero first moment of the approximate equilibrium probability density of the noise process.
In these cases,
it is useful to introduce an artificial drift term $\mu_d/\varepsilon$ in Galerkin formulation,
for some constant $\mu_d$ to be determined.
To formulate \cref{result:convergence_epsilon_galerkin_formulation} below,
let $\varphi_{0,d}$ denote the (assumed unique up to a sign) normalized one-dimensional eigenfunction associated with $\lambda_{0,d}$,
\begin{equation*}
    \label{eq:definition_phi_0}
    \varphi_{0,d} = \underset{\revision{\ip{\e^{V_{\eta}/2} \, \varphi_d}{ \e^{-\eta^2/4 \sigma_{\eta}^2}} > 0}}{\argmax_{\varphi_d \in S_{d_{\eta}}^{\eta}, \, \norm{\varphi_d}_{\e^{V_{\eta}} = 1}}}  \ip{\mathcal L_0^* \varphi_d}{\varphi_d}[\e^{V_{\eta}}], \qquad \text{where } S_{d_{\eta}}^{\eta} = \e^{- V_{\eta}/2} \, \e^{-\eta^2/4 \sigma_{\eta}^2} \, \poly(d_{\eta}).
\end{equation*}

\begin{result}
    \label{result:convergence_epsilon_galerkin_formulation}
    Let $\mu_d$ be defined by
    \begin{equation}
        \label{eq:corrective_drift_term}
        \mu_d = - \sqrt{2 \, \beta^{-1}} \, \zeta \, \int_{\real} \eta \, \varphi_{0,d}^2 \, \e^{V(\eta)} \, \d \eta.
    \end{equation}
    Then, when $\varepsilon \ll 1$, the solution $\rho_{d}$ of
    \begin{align*}
        \derivative{1}[\rho_d]{t} = \hat \Pi_d \, \left( \mathcal L_{\varepsilon}^* + \abs{\lambda_{0,d}} \, \id - \left(\frac{\mu_d}{\varepsilon}\right) \derivative{1}{x}\right) \, \hat \Pi_d \,  \rho_d, \qquad \rho_d(x, \eta, t=0) = \rho_{d,0}(x) \, \varphi_{0,d}(\eta)
    \end{align*}
    with $\mathcal L_{\varepsilon}^*$ as in \cref{eq:fokker_planck_colored_noise},
    can be approximated by $\rho_{d}^x(x, t) \, \varphi_{0,d}(\eta)$,
    where $\rho_{d}^x$ satisfies
    \begin{align}
        \label{eq:effective_fokker_planck_equation}
        \derivative{1}[\rho_d^x]{t} = \hat \Pi_d \left[ \derivative{1}{x} \left( \derivative*{1}[V]{x} \, \rho_d^x + A_d \, \derivative{1}[\rho_d^x]{x} \right) \right], \qquad \rho_d^x(x, t=0) = \rho_{d,0}(x).
    \end{align}
    Here the effective diffusion $A_d$ is equal to
    \begin{equation}
        \label{eq:effective_diffusion_Galerkin}
        A_d := \int_{\real} {(-\hat \Pi_d \, \mathcal L_0 \, \hat \Pi_d + \abs{\lambda_{0,d}} \, \id)}^{-1} \left(b_{\eta}  \, \varphi_{0,d} \right) \, \left(b_{\eta} \, \varphi_{0,d} \right)\, \e^{V_{\eta}(\eta)} \, \d \eta,
    \end{equation}
    with $b_{\eta} := \left(\mu_d + \sqrt{2  \, \beta^{-1}} \, \zeta \, \eta\right)$ and where $\mathcal L_0$ is the formal $L^2$ adjoint of $\mathcal L_0^*$.
\end{result}
\begin{proof}
The argument below is formal,
but it can be turned into a rigorous proof using standard methods in multiscale analysis;
see e.g.\ \cite{pavliotis2011applied}.
Expanding the solution in powers of $\varepsilon$ and gathering terms multiplying equal powers of $\varepsilon$,
a system of equations similar to \cref{eq:galerkin_asymptotics:first_equation,eq:galerkin_asymptotics:second_equation,eq:galerkin_asymptotics:third_equation} can be obtained,
differing only by the presence of the corrective drift term next to $\hat {\mathcal L}_1$.
The solvability condition for the first equation implies that $\varrho_0 = \rho^x_d(x, t) \, \varphi_{0,d}(\eta)$.
For the second equation,
we see from the definition of $\mu_d$ and
using the symmetry of $\hat {\mathcal L}_0$ in $\lp{2}{\real^2}[\e^{V_{\eta}(\eta) + U_x(x)}]$,
that the Fredholm solvability condition is automatically satisfied,
which enables solving for $\varrho_1$:
\begin{align*}
    \label{eq:galerkin_asymptotics:rho_1}
    \varrho_1 = {(-\hat \Pi_d \, \mathcal L_0^* \, \hat \Pi_d + \abs{\lambda_{0,d}} \, \id)}^{-1} \left( -(\mu_d + \sqrt{2 \, \beta^{-1}} \, \zeta \, \eta) \, \varphi_{0,d} \right) \, \derivative{1}[\rho^x_d]{x} + \Phi_1(x, t) \, \varphi_{0,d}.
\end{align*}
Writing out the solvability condition for the third equation,
we obtain the effective equation for $\rho^x_d$:
\begin{equation*}
    \label{eq:galerkin_asymptotics:centering_condition_third_equation}
    \derivative{1}[\rho^x_d]{t} = \hat {\mathcal L_2} \, \rho^x_d(x, t) \, + \int_{\real} \left(\hat {\mathcal L_1} - \mu_d \, (\hat \Pi_d \,{\textstyle \derivative{1}{x}} \hat \Pi_d)\right) \, \varrho_{1} \, \varphi_{0,d} \, \e^{V_{\eta}} \, \d \eta,
\end{equation*}
which after expansion of the terms is \cref{eq:effective_fokker_planck_equation}.
\end{proof}
In the examples we considered in~\cref{sec:numerics},
both $\varphi_{0,d}$ and $V_{\eta}$ were even functions,
so the corrective drift term defined by \cref{eq:corrective_drift_term} was zero.
This is not the case for the noise model {\bf NS},
where the noise is confined by the uneven potential~\eqref{eq:potential_for_non_symmetric_case}.
To illustrate the importance of including the corrective drift term~\eqref{eq:corrective_drift_term},
we present simulation results with and without it in the case of the noise model {\bf NS}.
We consider the following parameters for the equation: $\varepsilon = 2^{-5}$, $\theta = 0$, $\beta = 1$, $V(x) = x^4/4 - x^2/2$,
and for the numerics we choose $d=20$, $\sigma_x^2 = \sigma_{\eta}^2 = 0.1$, $\e^{-U_x(x)} = 1$.
The solutions obtained are presented in~\cref{fig:nonsymmetric_potential_with_corrective,fig:nonsymmetric_potential_without_corrective}.
\begin{figure}[!ht]
    \centering {%
        \begin{minipage}[b]{.49\linewidth}
            \centering%
            \includegraphics[width=.7\linewidth]{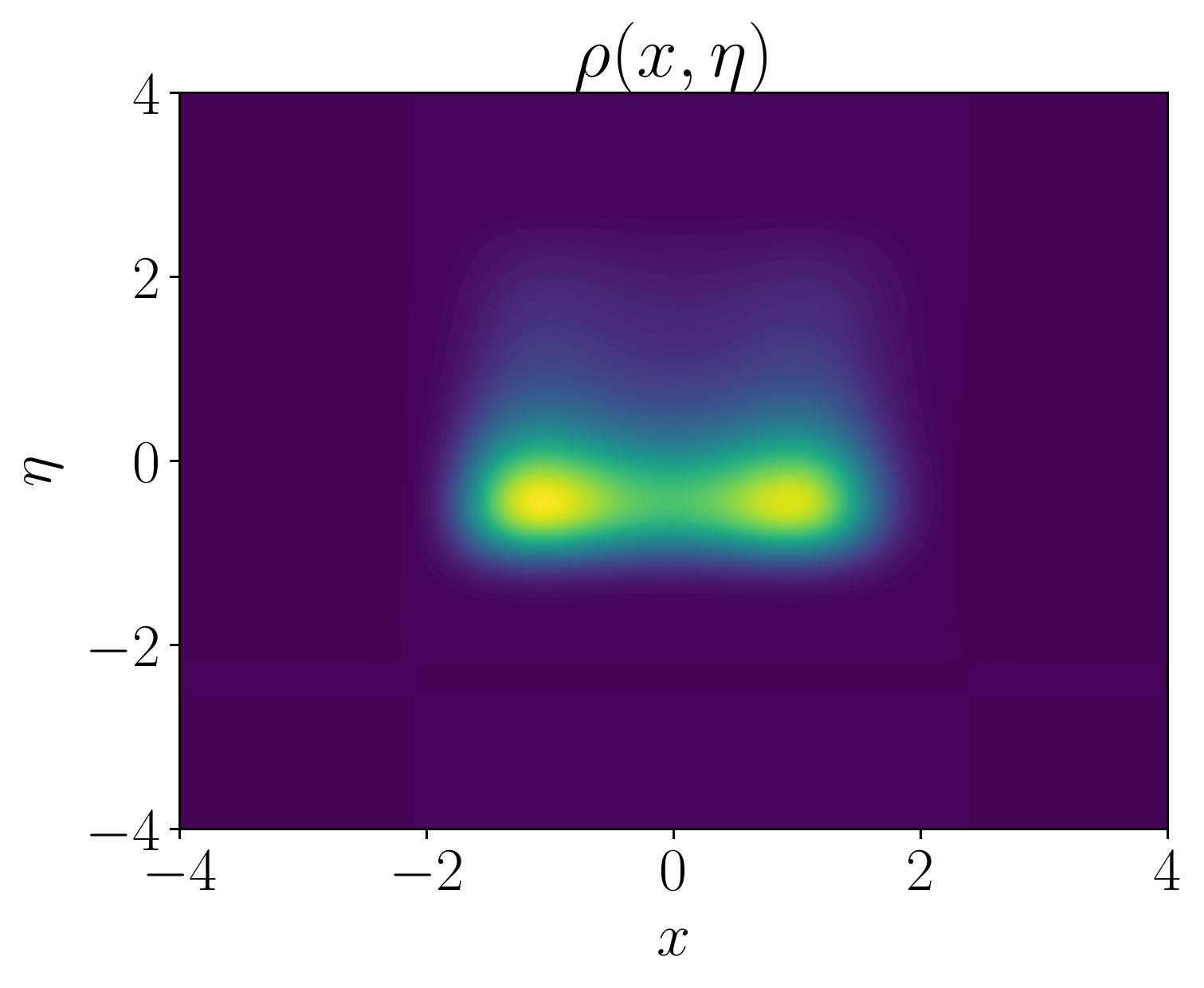}%
            \subcaption{With the corrective drift.}%
            \label{fig:nonsymmetric_potential_with_corrective}%
        \end{minipage}%
        \begin{minipage}[b]{.49\linewidth}
            \centering%
            \includegraphics[width=.7\linewidth]{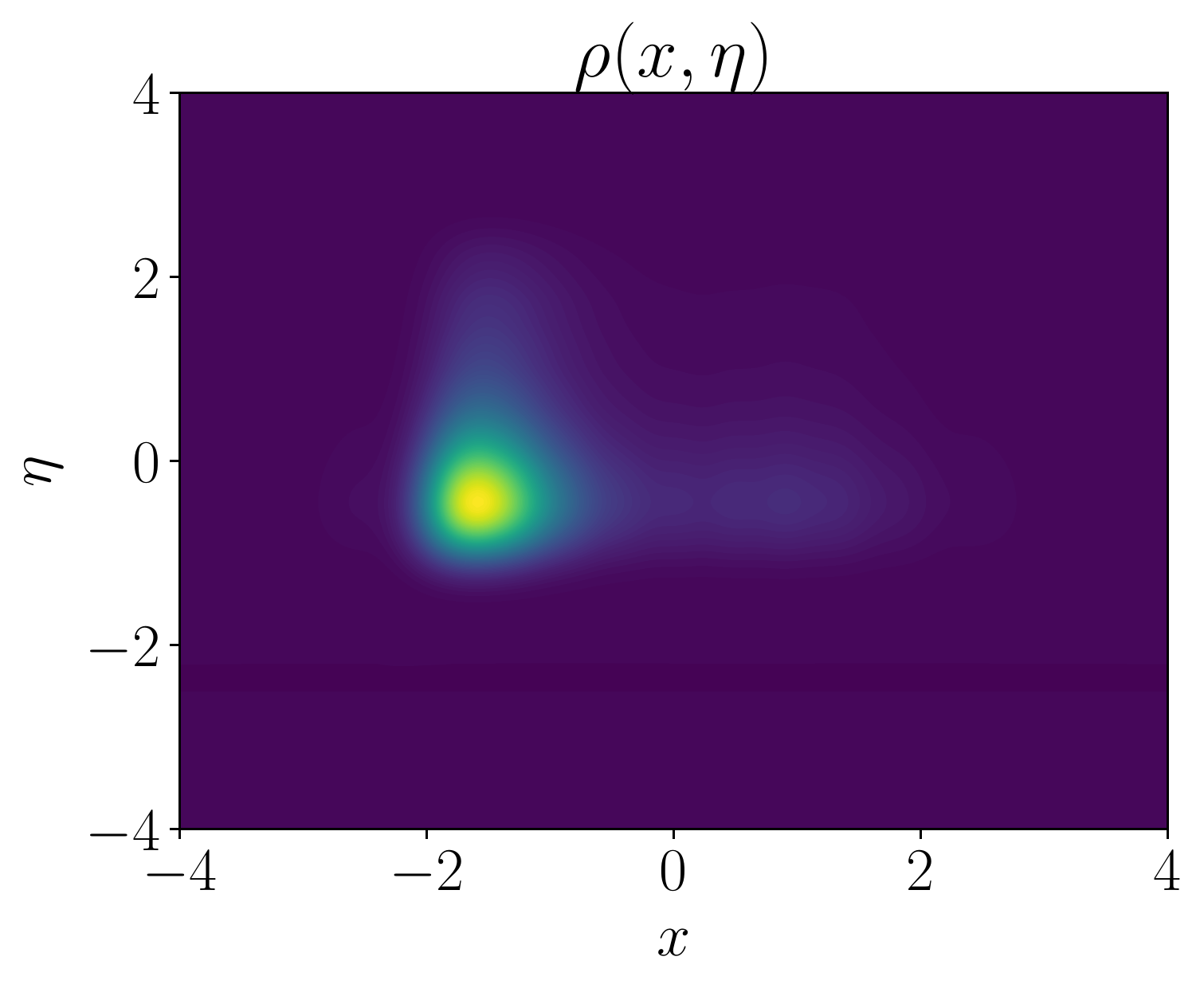}%
            \subcaption{Without the corrective drift.}%
            \label{fig:nonsymmetric_potential_without_corrective}%
        \end{minipage}%
    }
    \caption{%
        Comparison of the numerical solutions for the steady-state Fokker--Planck equation when the noise is confined by potential~\eqref{eq:potential_for_non_symmetric_case},
        with and without the corrective drift term.
        Without correction,
        the parasitic drift completely deteriorates the accuracy of the solution.
    }
    \label{fig:corrective_drift:visual_comparison}
\end{figure}

We note however that, for fixed $\varepsilon$, $\mu_d \to 0$ as $d \to \infty$,
by convergence of $\varphi_{0,d}$ to $C \, \e^{-V_{\eta}}$.
This is illustrated in \cref{fig:nonsymmetric_potential_bias},
where the same parameters as those used in~\cref{fig:corrective_drift:visual_comparison} were used.
While the value of the bias decreases exponentially with the degree of approximation,
we have found through numerical experiments that,
when $\varepsilon$ is of the order of $0.01$,
failure to account for the parasitic drift is very detrimental to the accuracy of the solution for $d$ as high as 30.
\begin{figure}[ht!]
    \centering%
    \includegraphics[width=.5\linewidth]{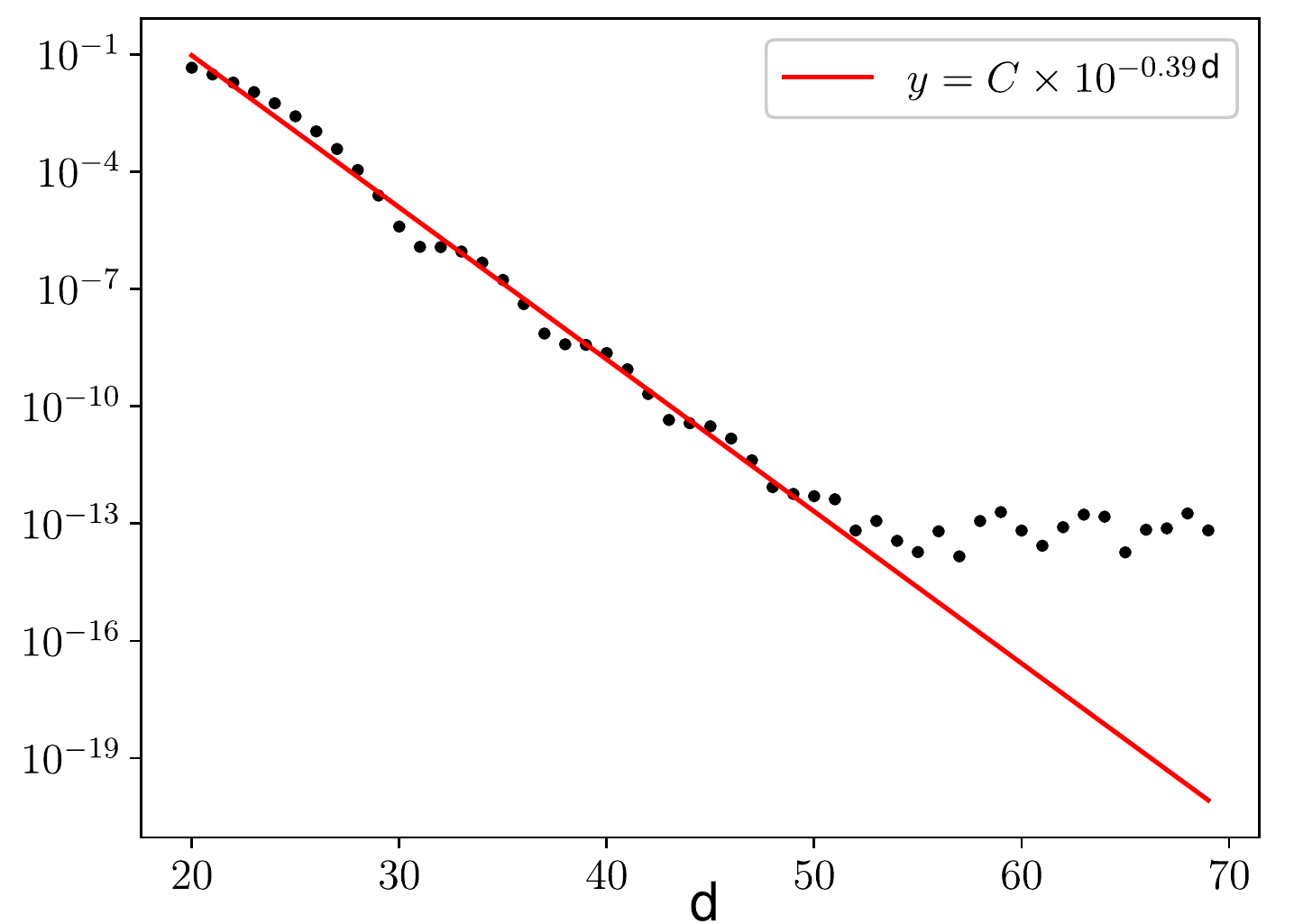}%
    \caption{%
        Normalized bias $\varepsilon \, \mu_d$ defined in~\eqref{eq:corrective_drift_term},
        against degree of approximation, for $\sigma_{\eta}^2 = 0.1$.
        (The values shown need to be divided by $\varepsilon$ to obtain the effective parasitic bias for a given value of $\varepsilon$.)
    }%
    \label{fig:nonsymmetric_potential_bias}%
\end{figure}

\subsection{Numerical verification of the rates of convergence}%
\label{sub:verification_of_the_rates_of_convergence}

In this section,
we verify numerically the rate of convergence of the stationary solution to the Galerkin formulation \cref{eq:galerkin_approximation_colored_noise} in the limit $\varepsilon \to 0$.
We recall that this Galerkin formulation is associated to the linear Fokker--Planck equation~\eqref{eq:galerkin_approximation_colored_noise}.
In addition to verifying the rates of convergence,
examining the limit $\varepsilon \to 0$ numerically will enable us to gain insight into the accuracy of the asymptotic expansions for moderate values of $\varepsilon$.

\paragraph{One-dimensional Ornstein--Uhlenbeck noise}%
\label{par:One-dimensional_Ornstein--Uhlenbeck_noise}
For this test we use the same parameters as in the convergence study for the bistable potential in \cref{sub:fokker_planck_equation_with_colored_noise_case}.
We verify the accuracy of the asymptotic expansion up to order $\varepsilon^2$ of the full solution in the $x$--$\eta$ plane,
given in~\cite[Chapter 3]{thesis_urbain},
by comparing it to numerical results obtained using the spectral method introduced in~\cref{sec:numerics}.
The convergence is presented in \cref{fig:asymptotics:convergence_in_bistable_case}.
We notice that, even for the smallest value of $\varepsilon$ considered ($2^{-6}$),
the norm of the difference between the asymptotic and spectral solutions appears to be roughly constant for $d \geq 20$.
This is because, beyond this point,
the spectral method is more accurate than the asymptotic expansion.
\begin{figure}[ht]
    \centering%
    \includegraphics[height=.3\linewidth]{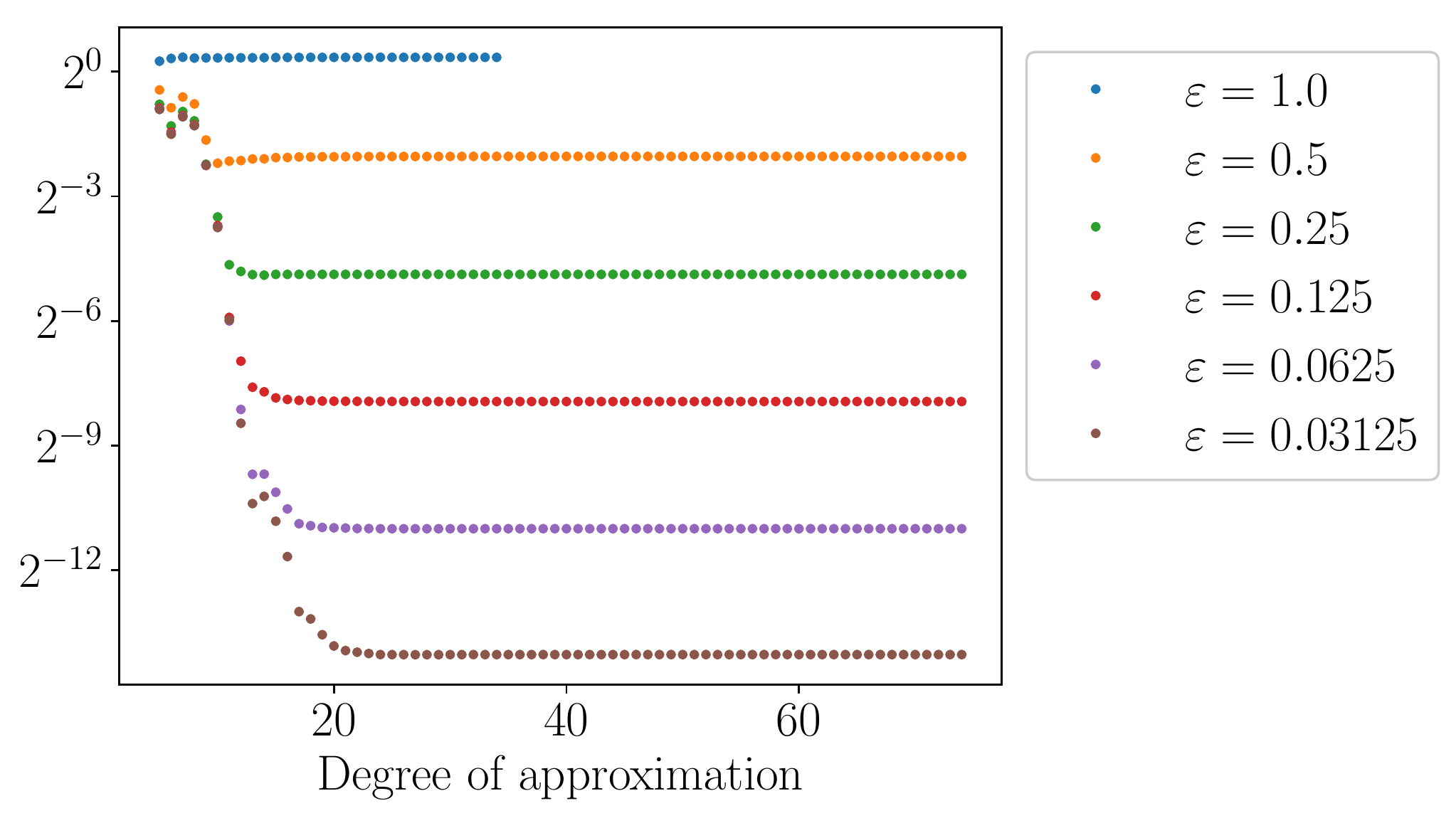}\hspace{1cm}
    \includegraphics[height=.3\linewidth]{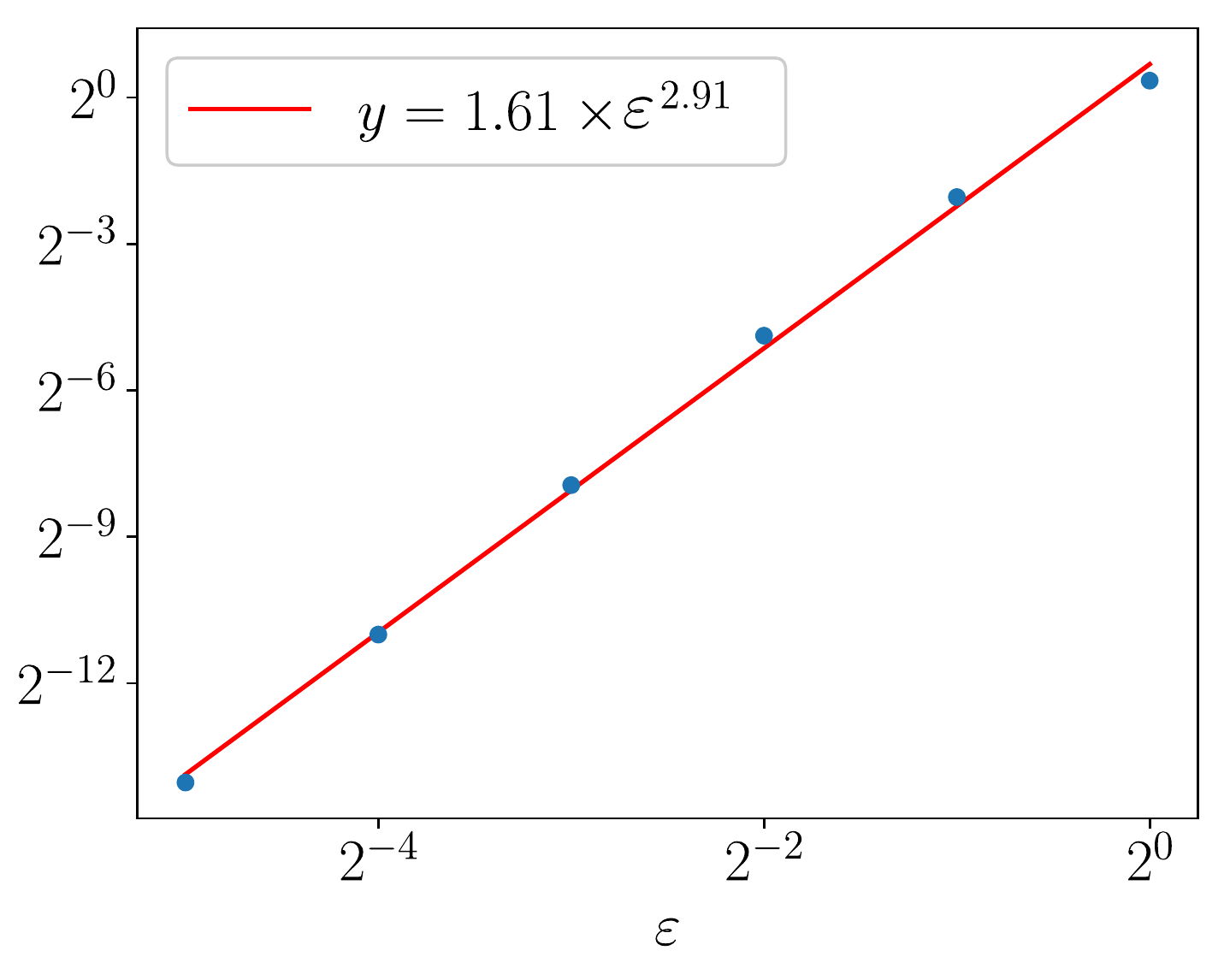}
    \caption{%
        $L^1$ norm of the difference between the solutions found using the Hermite spectral method and the asymptotic expansion up to order $\varepsilon^2$%
        , in both $x$ and $\eta$,
        in the case of {\bf OU} noise with $\varepsilon = 2^{-i}, i=0, \dotsc, 5$ (left),
        and the corresponding $L^1$ error of the asymptotic expansion (right).
        The scaling of the error (with respect to $\varepsilon$) is close to the expected value of 3.
    }%
    \label{fig:asymptotics:convergence_in_bistable_case}
\end{figure}

\paragraph{Harmonic noise}%
\label{par:hormonic_noise}
The second case we consider is that of harmonic noise,
for which the order of the first nontrival correction in the expansion of the solution is $\mathcal O(\varepsilon^4)$;
see~\cite{thesis_urbain}.
We confirmed this numerically for $V(x) = x^4/4 - x^2/2$ and $\beta = 5$
using 50 basis functions in each direction, with scaling factors $\sigma_x^2 = 1/30, \sigma_p^2 = \sigma_q^2 = 1$.
The results are illustrated in \cref{fig:harmonic_convergence_epsilon}.
\begin{figure}[ht!]
    \centering
    \includegraphics[width=.5\linewidth]{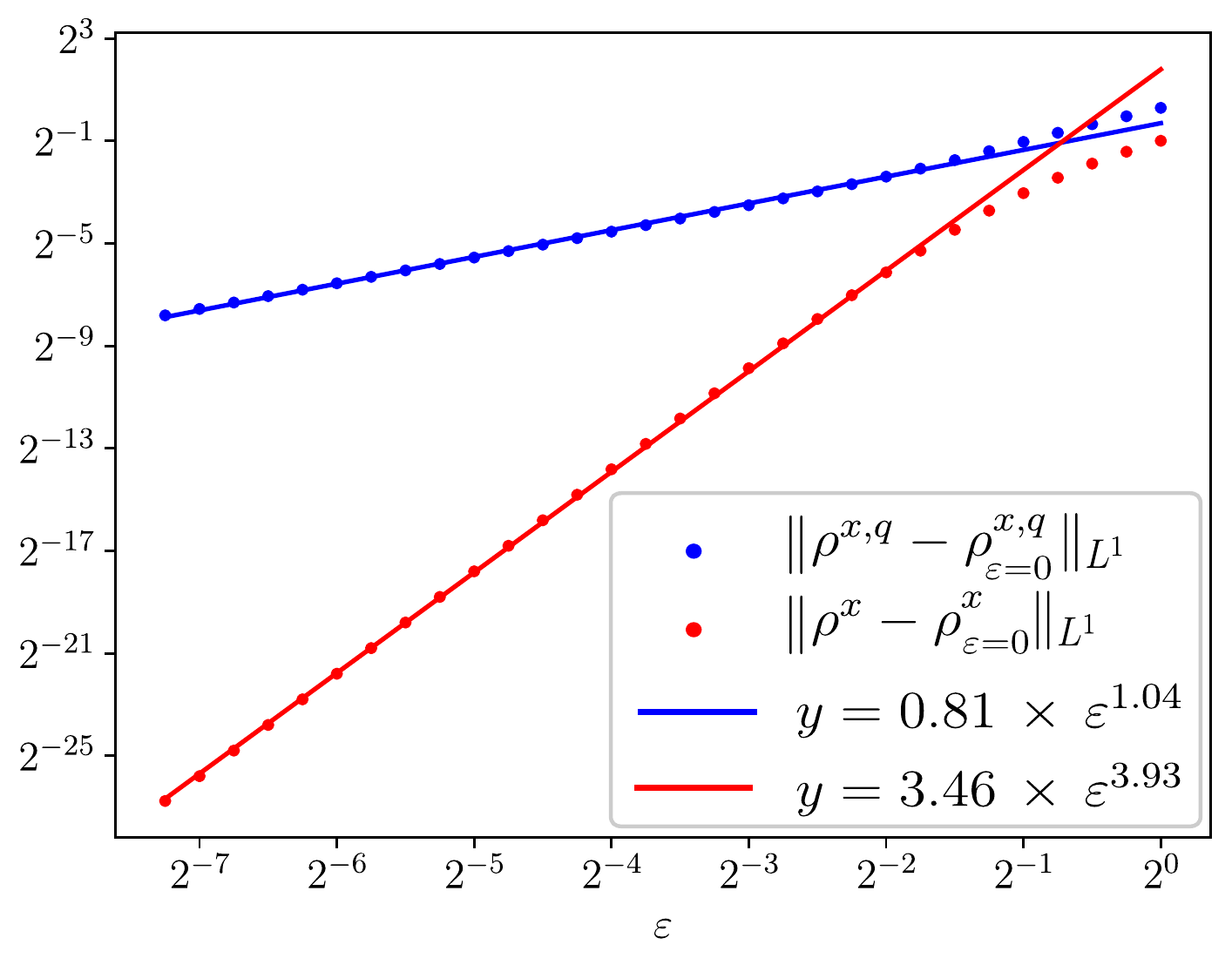}
    \caption{Convergence as $\varepsilon \to 0$ in the case of harmonic noise.
    The observed rate of convergence of the $x$-marginal to the white noise limit is $3.93$,
    which is close to the theoretical value of $4$.
    Here $\rho^{x,q}$ denotes the marginal of the solution on the $x,q$ plane.
    }%
    \label{fig:harmonic_convergence_epsilon}
\end{figure}
\fi

\section{Results: effect of colored noise on bifurcations}
\label{sec:results}

In this section we present the bifurcation diagrams corresponding to the four models of the noise introduced in \cref{sec:model}.
We begin with the case of Gaussian noise,
and later move to the case of non-Gaussian noise.

\subsection{Construction of the bifurcation diagrams for the mean field equation}
We constructed the bifurcation diagrams using three different approaches:

\paragraph{Monte Carlo simulations}
We solved the system of interacting particles~\eqref{eq:main-sde} with a sufficiently large number of particles,
and we approximated the first moment by ergodic average over an interval $(T,T+\Delta T)$,
where $T$ is sufficiently large to guarantee that the system has reached its stationary state and
$\Delta T$ is sufficiently large to ensure that the ergodic averages are accurate.
By applying this procedure for a range of inverse temperatures, $\beta = 0.1, 0.15, 0.2, \dotsc, 10$,
we obtained the desired bifurcation diagram.

\paragraph{Perturbation expansions}
This approach,
which we already outlined in \cref{sec:model},
relies on the fact that  the self-consistency map can be approximated as $R(m, \beta) \approx R_0(m, \beta) + \varepsilon^{\delta} R_{\delta}(m, \beta)$,
with good accuracy when $\varepsilon \ll 1$.
Here we used the same notation as in \cref{sec:model},
and in particular $\delta$ denotes the order of the first nontrivial correction in \cref{eq:x_marginal_asymptotic_expansion}.
Using arclength continuation\footnote{%
    We do this using the Moore--Penrose quasi-arclength continuation algorithm.
    The rigorous mathematical construction of the arclength continuation methodology can be found,
    e.g., in~\cite{Krauskopf} and~\cite{ContAllgower}.
    Some useful practical aspects of implementing arclength continuation are also given in~\cite{matcont}.
    See also~\cite{GomesEtAl2018}.
}
for the resulting approximate self-consistency equation, $m = R_0(m, \beta) + \varepsilon^{\delta} R_{\delta}(m, \beta)$,
we can plot the first moment $m$ as a function of $\beta$ for a fixed value of $\varepsilon$.
We note that, in view of the typical shape of the self-consistency map,
depicted in a particular case in \cref{fig:rm-vs-m},
a standard root finding algorithm can be employed to initiate the arclength continuation at some initial inverse temperature $\beta_0$.

\paragraph{The spectral method}
Finally, we employed the Galerkin method presented in \cref{sub:fokker_planck_equation_with_colored_noise_case}.
We considered two different methodologies:
on the one hand, by calculating numerically an approximation $\rho_{d,\infty}(x, \eta; \beta, m)$ of the steady-state solution of the linear Fokker--Planck equation~\eqref{eq:SS-general} with fixed $m$ and $\beta$,
we approximated the self-consistency map as $R(m, \beta) \approx \int_{\real} \int_{\real^n} x \, \rho_{d,\infty}(x, \eta; \beta, m) \, \d x \, \d \vect y$,
after which a bifurcation diagram can be constructed by using the same method as in the previous paragraph.
Each evaluation of the self-consistency map requires the computation of the eigenvector
associated with the eigenvalue of smallest magnitude of the discretized operator,
which can be performed efficiently for sufficiently small systems using the \emph{SciPy} toolbox.
On the other hand,
the time-dependent (nonlinear) McKean--Vlasov equation can be integrated directly using our spectral method.
Since only the final solution is of interest to us,
the semi-implicit time-stepping scheme~\eqref{eq:galerkin_approximation_white_noise_nonlinear_time_stepping}
can be used with a large time step,
which enables a quick and accurate approximation of the steady-state solutions.
While both methodologies work well in the two-dimensional case,
in three dimensions (harmonic noise) solving the McKean--Vlasov equation directly proved more efficient,
so this is the approach we employed for all the tests presented in this section.

\subsection{Gaussian case}%
\label{sub:gaussian_case}

The one-dimensional Ornstein--Uhlenbeck noise provides an ideal testbed for the three methods we use to construct bifurcation diagrams.
\Cref{fig:results:bifurcations_ou} below plots the bifurcation diagram of the first moment $m$ as a function of $\beta$ for $\varepsilon = 0.1, 0.2, \dotsc, 0.5$.
Three different initial conditions ($\vect{X}_0\sim N(0,0.1), \, \vect{X}_0\sim N(0.1,0.1)$, and $\vect{X}_0\sim N(-0.1,0.1)$) were used for the MC simulations.
Although we observe that the results of MC simulations tend to be less precise around the bifurcation point,
the agreement between the three methods overall is excellent for $\varepsilon = 0.1, 0.2$.
For the other values of $\varepsilon$,
while the results of MC simulations and of our spectral method continue to agree,
those obtained from the asymptotic expansion are significantly less accurate,
which is consistent with the observations presented in
\iflong
\cref{fig:asymptotics:convergence_in_bistable_case}.
\else
\cite{thesis_urbain}.
\fi

\begin{figure}[ht!]
    \centering
    \begin{minipage}[b]{.9\linewidth}%
        \includegraphics[width=\linewidth]{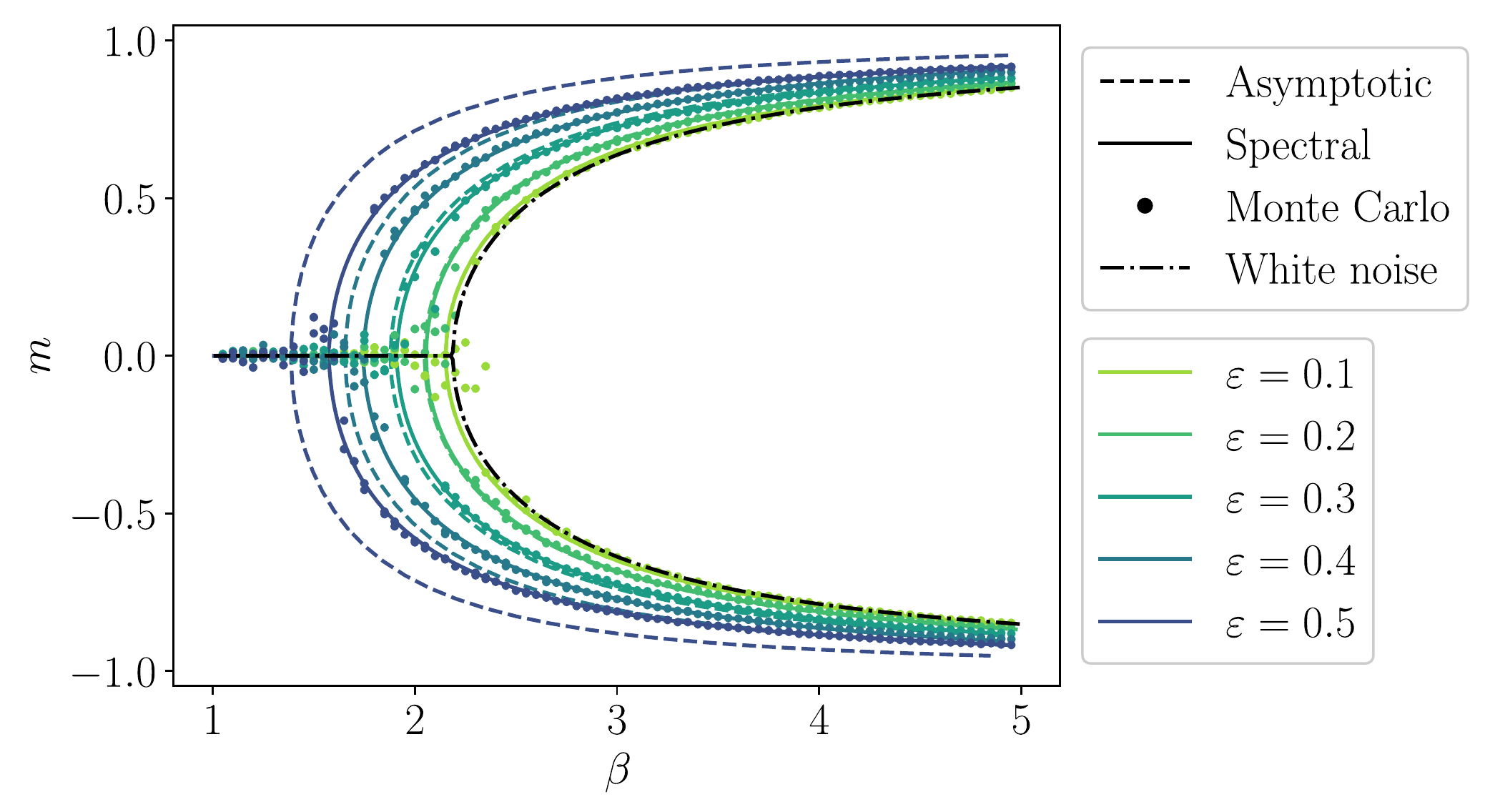}
    \end{minipage}%
    \caption{%
        Bifurcation diagram of $m$ against $\beta$ for Ornstein--Uhlenbeck noise,
        obtained via MC simulation, the spectral method, and the asymptotic expansion~\eqref{eq:p-asymptotic-OU}.
    }
    \label{fig:results:bifurcations_ou}
\end{figure}

The case of harmonic noise,
corresponding to a three-dimensional McKean--Vlasov equation,
is more challenging to tackle using our spectral method.
When using 40 basis functions in each direction,
the CPU time required to construct the full bifurcation diagram was of the order of a week.
As a consequence of the lower number of basis functions used in this case,
we observe a small discrepancy between the results of the spectral method and
those of MC simulations for large $\beta$ in the case $\varepsilon = 0.4$.
Nevertheless, as can be seen in~\cref{fig:results:bifurcations_harmonic},
for small $\varepsilon$ the overall agreement between the three methods is excellent.
We note in particular that,
as suggested by the asymptotic expansions,
the use of harmonic noise produces results much closer to the white noise limit than scalar OU noise.
\begin{figure}[ht!]
    \includegraphics[height=.31\linewidth]{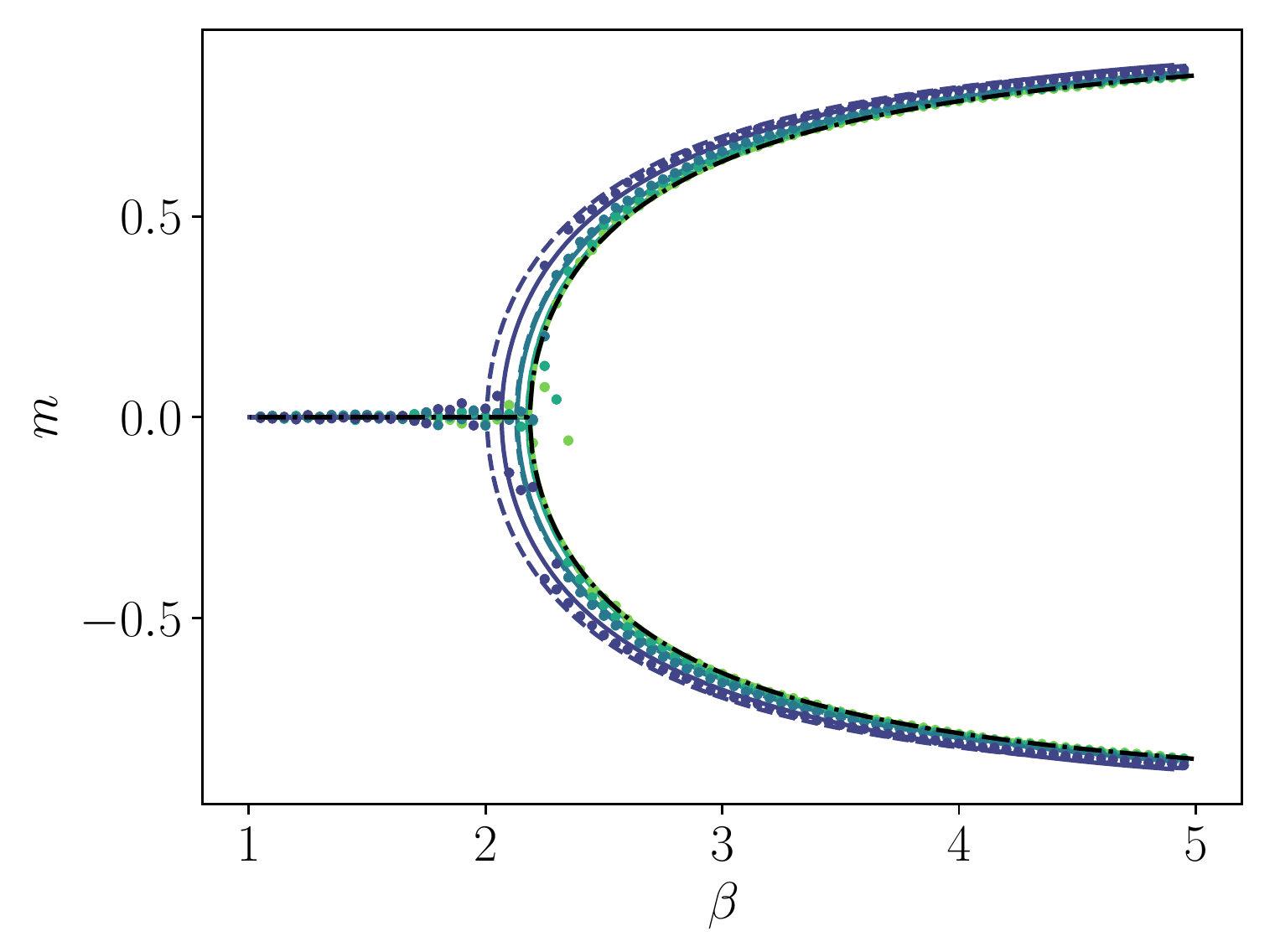}
    \includegraphics[height=.31\linewidth]{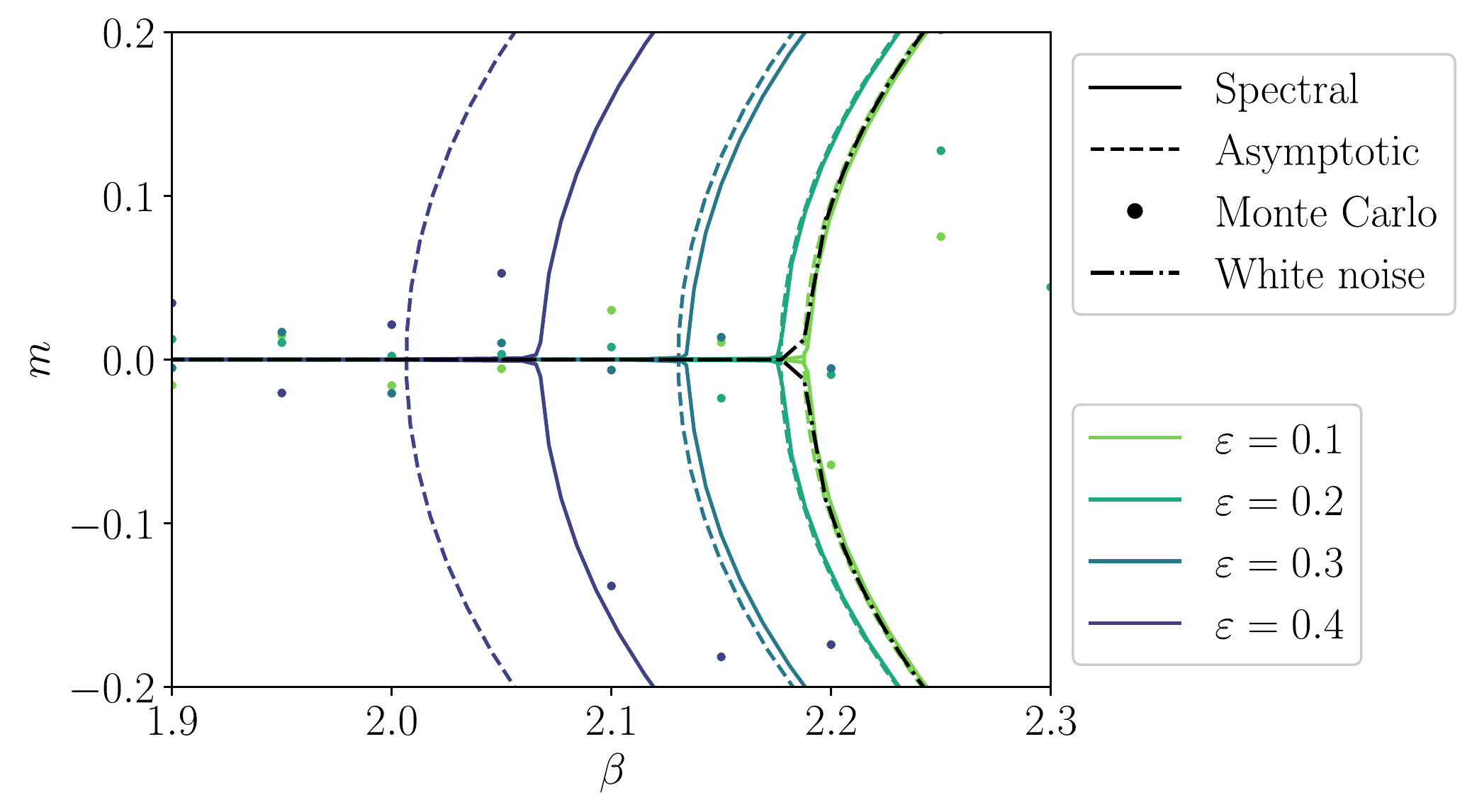}
    \caption{%
        Bifurcation diagram of $m$ against $\beta$ for harmonic noise (model {\bf H}),
        obtained via MC simulation, the Hermite spectral method, and the asymptotic expansion~\eqref{eq:p-asymptotic-OU}.
    }
    \label{fig:results:bifurcations_harmonic}
\end{figure}

\subsection{Non-Gaussian noise}%
\label{sub:non_gaussian_noise}

For the non-Gaussian noise processes we consider,
the $x^4$ asymptotic growth of the confining potentials in both directions
causes the McKean--Vlasov equation to be stiffer than in the cases of OU and harmonic noise,
especially for large values of $\varepsilon$.
Consequently, we were not able to consider as wide a range of $\varepsilon$ as in the previous subsection using the spectral method.
Since, on the other hand, MC simulations become overly computationally expensive for small $\varepsilon$,
the comparisons in this section comprise only results obtained using our spectral method and asymptotic expansions.
Results of simulations for the bistable noise (model {\bf B}) are presented in \cref{fig:results:bifurcations_bistable},
in which a very good agreement can be observed.
\begin{figure}[ht!]
    \centering
    \includegraphics[height=0.31\linewidth]{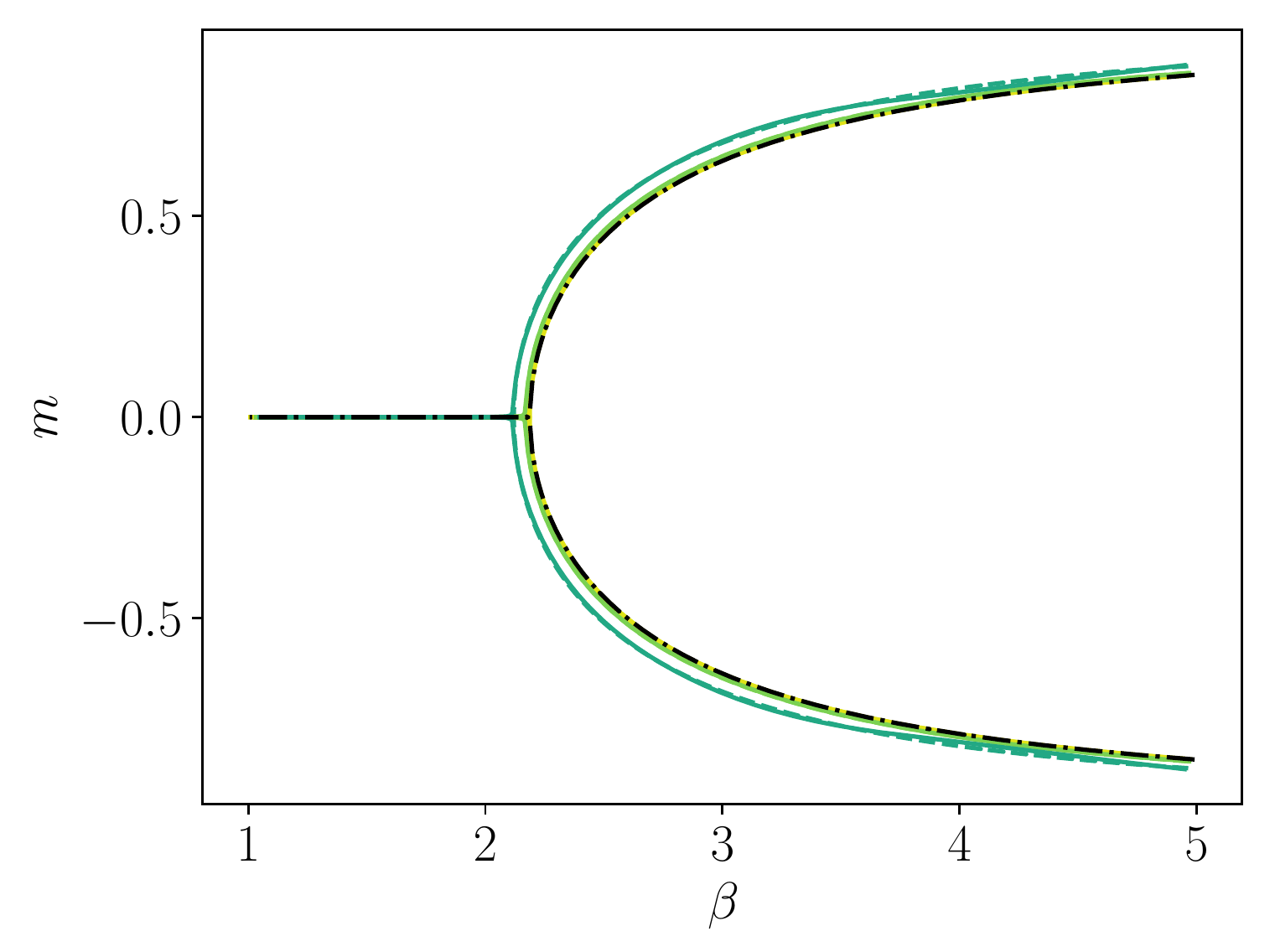}
    \includegraphics[height=0.31\linewidth]{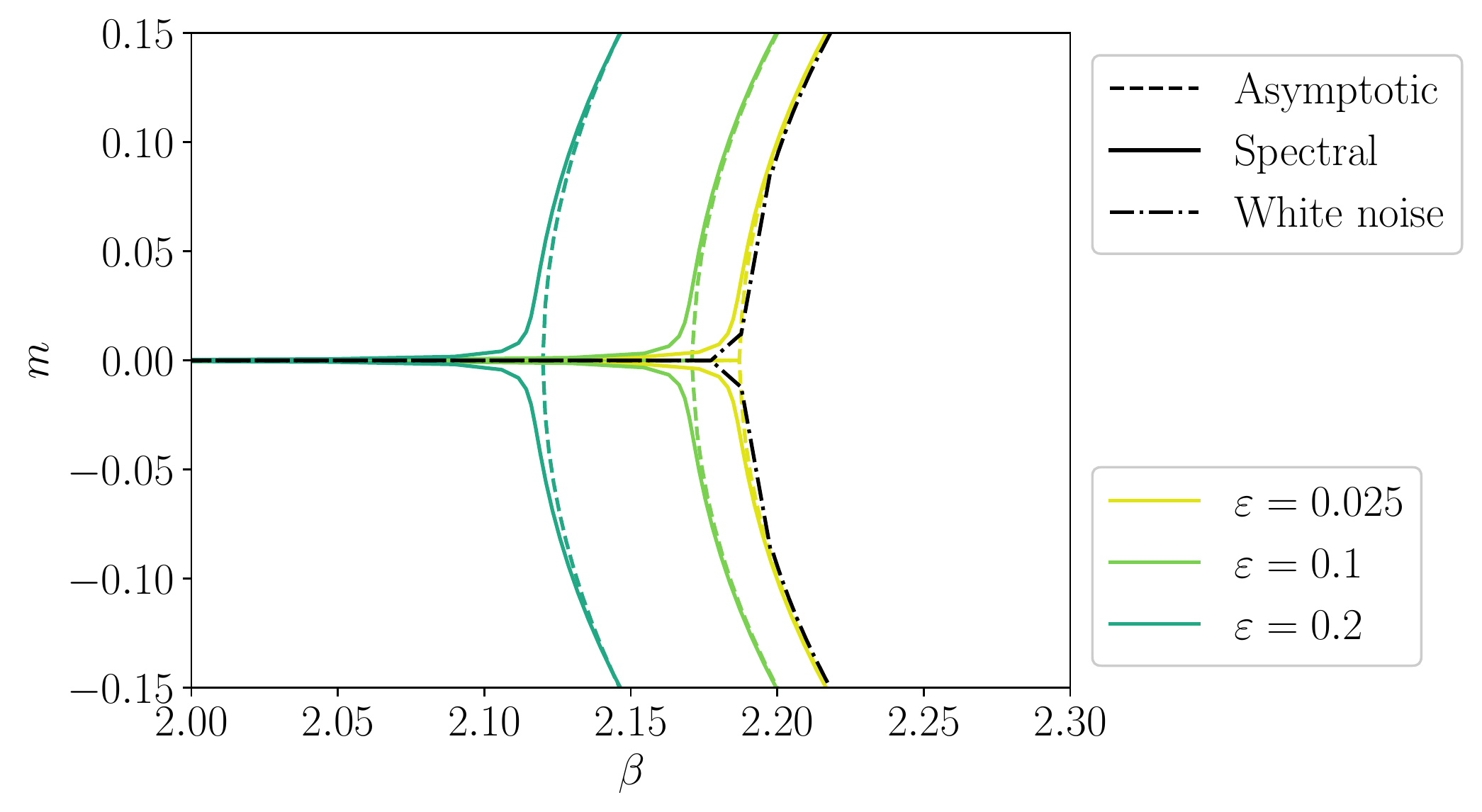}
    \caption{%
        Bifurcation diagram of $m$ against $\beta$ for the bistable noise (model {\bf B}),
        using the spectral method and a truncated asymptotic expansion including the first nonzero correction.
        We see that, overall, the agreement between the two methods is excellent.
    }
    \label{fig:results:bifurcations_bistable}
\end{figure}

For nonsymmetric noise (model {\bf NS}),
the two branches in the bifurcation diagram are separate,
as illustrated in~\cref{fig:results:bifurcations_ns}.
Here too, the agreement between the spectral method and the asymptotic expansion is excellent.
In contrast with the other models considered,
the first nonzero term in the asymptotic expansion is of order $\varepsilon$,
which is reflected by the manifestly higher sensitivity to the correlation time of the noise.
In the right panel of~\cref{fig:results:bifurcations_ns},
we present the graph of $R_0(m; \beta) + \varepsilon R_1(m; \beta)$ for a value of $\beta$ close to the point at which new branches (one stable and one unstable) emerge.
\begin{figure}[ht!]
    \centering
    \includegraphics[height=.31\linewidth]{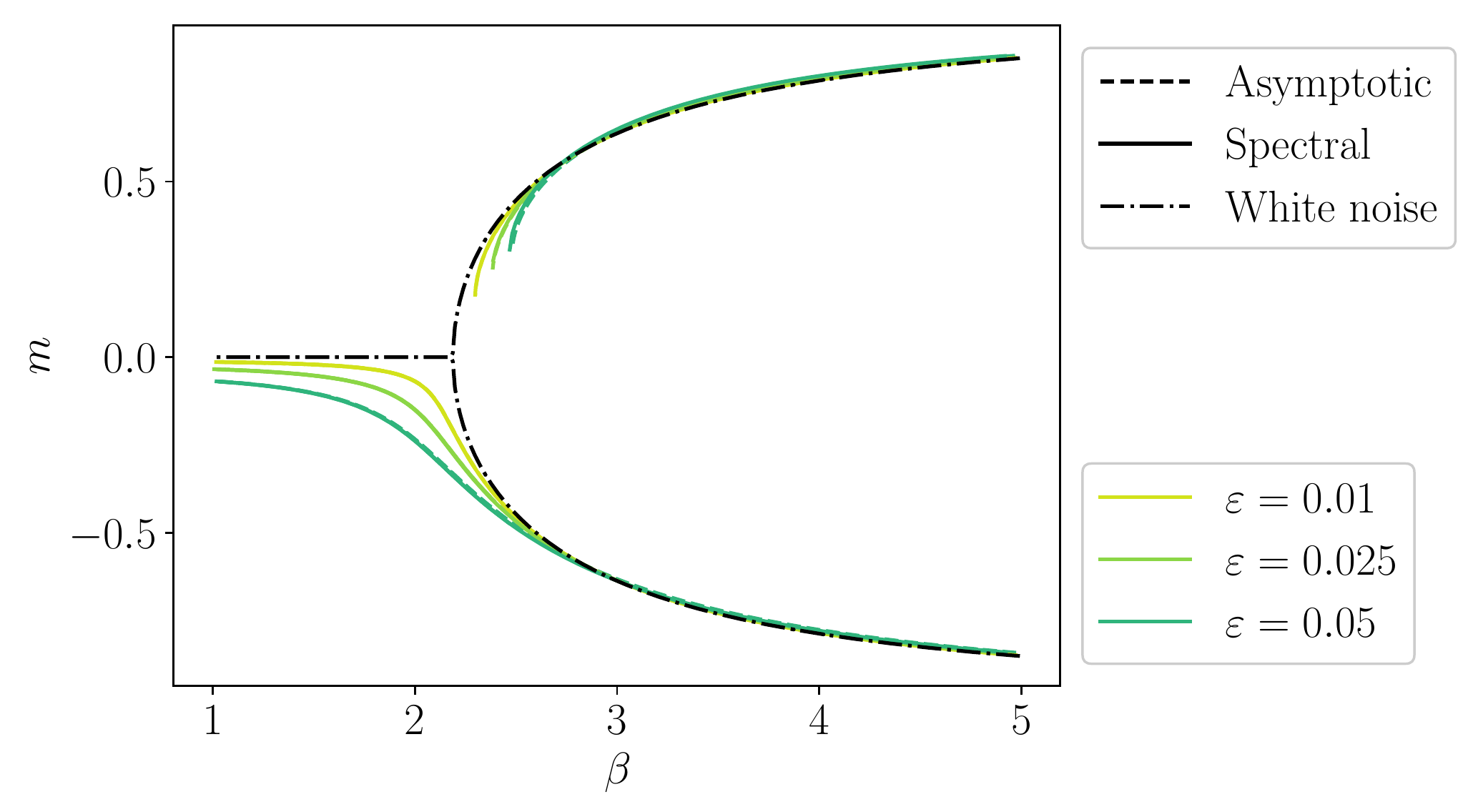}
    \includegraphics[height=.31\linewidth]{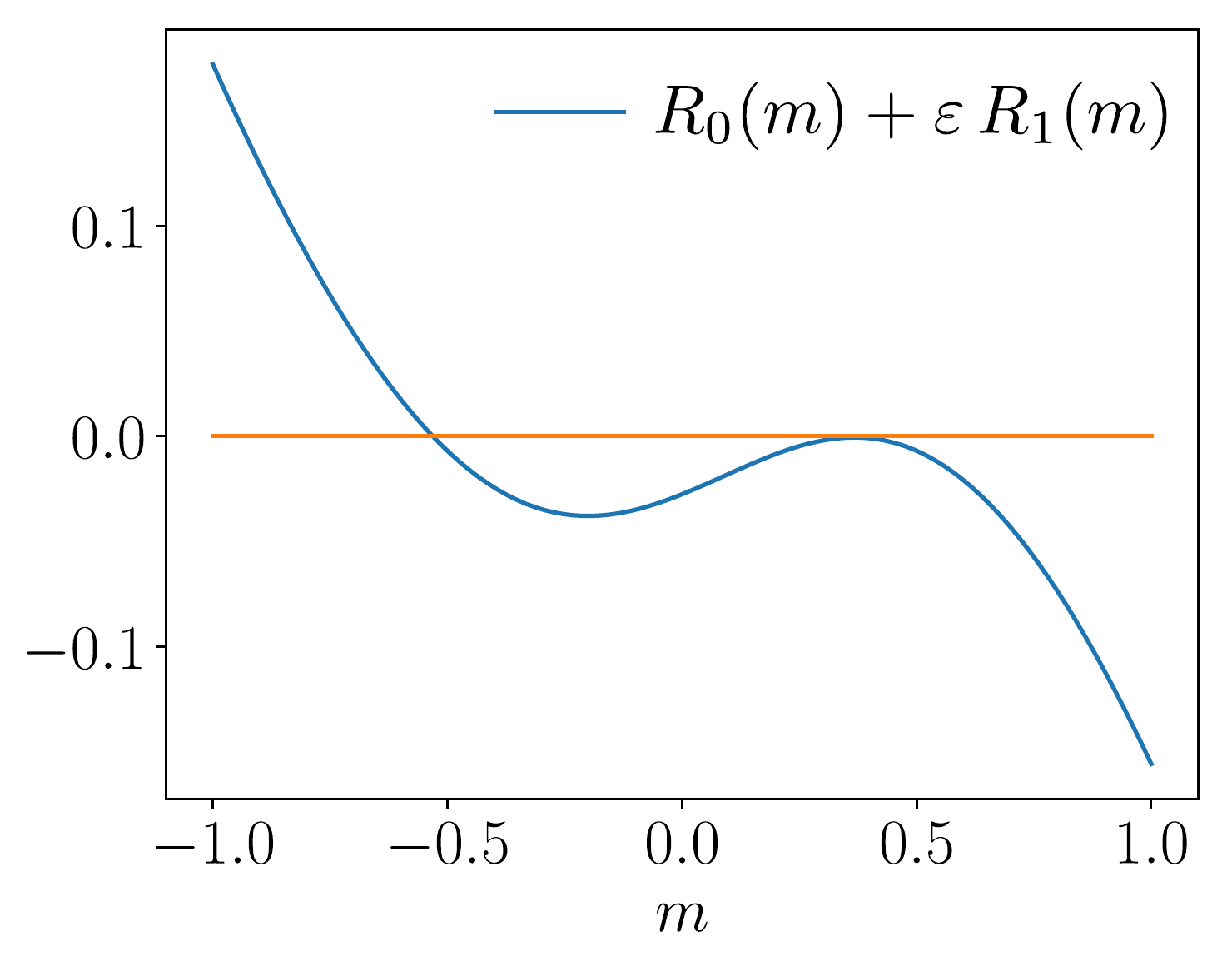}
    \caption{%
        Left: bifurcation diagram of $m$ against $\beta$ for the nonsymmetric noise (model {\bf NS}),
        using the spectral method and a truncated asymptotic expansion including the first nonzero correction.
        Right: $R_0 + \varepsilon R_1(m) - m$ against $m$ for $\varepsilon = 0.1$ and $\beta = 2.6$.
    }
    \label{fig:results:bifurcations_ns}
\end{figure}

\subsection{Dependence of the critical temperature on \texorpdfstring{$\varepsilon$}{e}}%
\label{sub:critical_temperature}




For the noise models {\bf OU}, {\bf H} and {\bf B},
the effect of colored noise on the dynamics is a shift of the critical temperature:
the pitchfork bifurcation occurs for \emph{smaller} values of $\beta$ (i.e., larger temperatures) as the correlation time increases.
In order to further investigate the effect of the correlation time on
the long time behavior of the system of interacting particles,
we will compute the critical temperature as a function of $\varepsilon$ based on the asymptotic expansions
and compare with the results of spectral and MC simulations,
see \cref{fig:critical_temperature_epsilon}.
Rather than finding the critical inverse temperature $\beta_C$ for a range of values of $\varepsilon$ (and for a fixed $\theta$),
it is convenient to fix $\beta_{C}$ and find the corresponding $\varepsilon$, satisfying
\begin{equation}
    \label{eq:critical_epsilon}
    \derivative*{1}{m} \left( \int_{\real} x \, p_0(x; \beta_C, m) \, \d x \right)_{m=0} +  \varepsilon^{\delta} \derivative*{1}{m} \left(\int_{\real} p_{\delta}(x; \beta_C, m) \, \d x \right)_{m=0} = 1,
\end{equation}
which is merely a polynomial equation in $\varepsilon$,
the coefficient of which can be calculated by numerical differentiation.
With this procedure,
the dependence of the critical $\beta$ upon $\varepsilon$ can be calculated on a fine mesh.
In the case of {\bf OU} noise, for example,
both coefficients on the left-hand side of \cref{eq:critical_epsilon} are positive,
implying that the equation has a solution
(in fact, two, but one of them negative)
only if $\beta_C$ is lower than
the inverse critical temperature in the white noise case.

Of the three methods employed in \cref{fig:critical_temperature_epsilon},
the approach based on the asymptotic expansions has the lowest computational cost:
calculating all the solid curves took only about a couple of minutes on a personal computer with an Intel i7-3770 processor.
The data points associated with the spectral method and the MC simulations were obtained from the bifurcation diagrams presented above.
\begin{figure}[ht!]
    \centering
    \includegraphics[width=.8\linewidth]{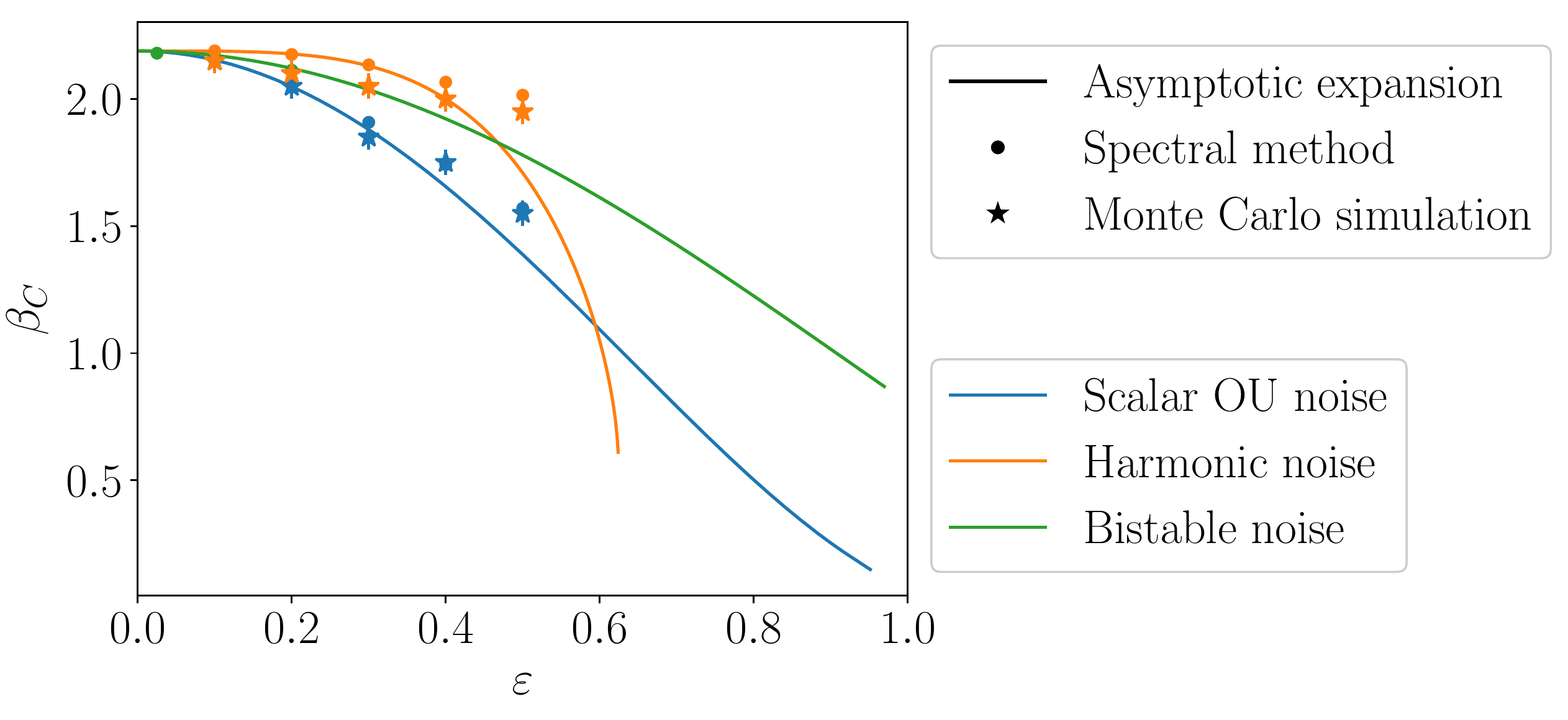}
    \caption{Critical $\beta$ against $\varepsilon$.}
    \label{fig:critical_temperature_epsilon}
\end{figure}

\section{Conclusions}\label{sec:conclusions}
In this paper, we introduced a robust spectral method
for the numerical solution of linear and nonlinear, local and nonlocal Fokker--Planck-type PDEs
that does not require that the PDE is a gradient flow.
We then used our method to construct the bifurcation diagram for the stationary solutions of the mean field limit of a system of weakly interacting particles driven by colored noise.

To verify our results,
we also constructed the bifurcation diagrams by using two other independent approaches,
namely by MC simulation of the $N$-particle system and
by using explicit asymptotic expansions with respect to correlation time of the noise.
In the small correlation time regime, we observed a very good agreement between all three methods.
For larger values of the correlation time,
the asymptotic expansions become inaccurate,
but the results obtained via the spectral method and MC simulations continue to be in good agreement.

It appeared from our study that,
unless the potential in which the noise process is confined is asymmetric,
the correlation structure of the noise does not influence the topology of the bifurcation diagram:
the mean-zero steady-state solution,
which is stable for sufficiently large temperatures,
becomes unstable as the temperature decreases below a critical value,
at which point two new stable branches emerge,
in the same manner as reported in~\cite{Dawson1983,shiino1987}.
The correlation structure does, however, influence the temperature at which bifurcation occurs,
and in general this temperature increases as the correlation time of the noise increases.
In the presence of an asymmetry in the confining potential of the noise,
on the other hand,
the two stable branches in the bifurcation diagram are separate,
indicating that the system always reaches the same equilibrium upon slowly decreasing the temperature.
This behavior is similar to what has been observed previously in the white noise case
when a tilt is introduced in the confining potential $V(\cdot)$, see~\cite{GomesEtAl2018,GomesPavliotis2017}.

Several problems remain open for future work.
On the theoretical front,
we believe that the analysis we presented in~\cref{sub:white_noise_case,sec:proof_convergence_theorem} for the linear Fokker--Planck equation
can be extended to both the linear Fokker--Planck equation with colored noise and the nonlinear McKean--Vlasov equation.
Another direction for future research could be the rigorous study of bifurcations and,
more specifically, of fluctuations and critical slowing down near the bifurcation point.
On the modeling front, it would be interesting
to consider more general evolution equations for the interacting particles,
such as the generalized Langevin equation,
and also
to study systems of interacting particles subject to colored noise that is multiplicative.

\appendix
\section{Hermite polynomials and Hermite functions}%
\label{sec:hermite_polynomials}

In one dimension, the orthonormal Hermite polynomials can be defined by:
\begin{equation*}
    \label{eq:hermite_polynomials_1d}
    H_n(x) = \frac{(-1)^n}{\sqrt{n!}} \exp\left({\frac{x^2}{2}}\right) \derivative*{n}{x^n} \, \left(\exp{\left(-\frac{x^2}{2}\right)}\right), \qquad n = 0, 1, \dotsc
\end{equation*}
They form a complete orthonormal basis of the weighted space $\lp{2}{\real}[g]$,
where $g := \e^{-x^2/2}/\sqrt{2\, \pi}$,
and they satisfy the following recursion relations
(\cref{eq:hermite_polynomials_1d_recursion_three_terms} can be obtained by combining \cref{eq:hermite_polynomials_1d_recursion_adjoint_derivative,eq:hermite_polynomials_1d_recursion_derivative}):
\begin{align}
    \label{eq:hermite_polynomials_1d_recursion_adjoint_derivative}
    &\left(x - \derivative*{1}{x}\right) \, H_n = \sqrt{n+1} \, H_{n+1}; \\
    \label{eq:hermite_polynomials_1d_recursion_derivative}
    &\derivative*{1}[H_{n+1}]{x} = \sum^{n}_{i=0} \, H_i \, \int_{\real} \derivative*{1}[H_{n+1}]{x} \, H_i(x) \, g(x) \, \d x = \sqrt{n+1} \, H_{n}; \\
    \label{eq:hermite_polynomials_1d_recursion_three_terms}
    &H_{n+1} = \sqrt{\frac{1}{n+1}} \, x \, H_{n} - \sqrt{\frac{n}{n+1}} H_{n-1}.
\end{align}
From \cref{eq:hermite_polynomials_1d_recursion_adjoint_derivative,eq:hermite_polynomials_1d_recursion_derivative},
we see that Hermite polynomials are the eigenfunctions of a second-order operator:
\begin{equation}
    \label{eq:hermite_polynomials_1d_as_eigenfunctions}
    \mathcal L \, H_i := \partial^*_x \partial_x \, H_i := \left(x - \derivative*{1}{x}\right) \derivative*{1}{x} H_i = i \, H_i,
\end{equation}
which is essential to proving approximation results.
We note that the eigenvalues grow linearly,
which explains the square root in the rate of convergence
in the results presented below.
Since Hermite polynomials constitute an orthonormal basis of $\lp{2}{\real}[g]$,
any function $u$ in that space can be expanded as a series of Hermite polynomials, and
\begin{equation*}
    \sum_{i=0}^{d} H_i \, \ip{u}{H_i}[g] \to u \quad \text{in}~\lp{2}{\real}[g]~\text{as}~d \to \infty.
\end{equation*}
We will call \emph{Hermite transform} the operator:
\begin{align*}
    \mathcal T: \, &\lp{2}{\real}[g] \to \ell^2 \\
                   &u \mapsto \big( \ip{u}{H_0}[g], \ip{u}{H_1}[g], \dotsc \big).
\end{align*}
Denoting by $\Pi_d$ the $\lp{2}{\real}[g]$ projection operator on $\Span \{ H_0, H_1, \dotsc, H_d \}$,
the following theorem follows from~\eqref{eq:hermite_polynomials_1d_recursion_derivative}; see~\cite{shen2011spectral} for details.
\begin{theorem}
    [Approximation by Hermite polynomials]
    \label{thm:hermite_polynomials_approximation}
    For any $u \in \sobolev{\m}{\real}[g]$ with $0 \leq \m \leq d + 1$,
    the following inequality holds:
    \begin{equation*}
        \label{eq:hermite_convergence_polynomials}
        \left\|\derivative*{\l}{x^\l} (\Pi_d u - u)\right\|_g \leq \sqrt{\frac{(d-\m+1)!}{(d-\l+1)!}} \, \norm{\derivative*{\m}[u]{x^\m}}_g, \qquad \l = 0, \dotsc, \m.
    \end{equation*}
\end{theorem}
The Hermite polynomials, as defined above,
are suitable for the approximation of functions with respect to the norm of $\lp{2}{\real}[g]$,
which assigns a significant weight only to the region around $x = 0$.
For the approximation with respect to the flat $\lp{2}{\real}$ norm,
or with respect to other norms that penalize growth as $x \to \infty$, such as the weighted $\lp{2}{\real}[e^{x^2/2}]$ norm,
one can use the basis functions $(\e^{-U/2} \, H_i)_{i = 0}^{\infty}$ for some function $U$,
which constitute an orthonormal basis of $\lp{2}{\real^n}[\e^{U}g]$.
For $u \in \lp{2}{\real}[\e^{U}g]$,
we define the generalized Hermite transform associated with the factor $\e^{-U/2}$,
which we denote by $\mathcal T_{U}$:
\begin{equation}
\label{eq:generalized_hermite_transform}
\begin{aligned}
    \mathcal T_{U}: \, &\lp{2}{\real}[\e^{U}g] \to \ell^2 \\
                       &u \mapsto \left( \ip{u}{\e^{-U/2} H_0}[\e^{U}g], \ip{u}{\e^{-U/2} H_1}[\e^{U}g], \dotsc \right).
\end{aligned}
\end{equation}
Note that $\mathcal T_{U} (u) = \mathcal T (\e^{U/2} u)$,
i.e.\ $T_U(u)$ is the usual Hermite transform of $\e^{U/2} u$.
This formalism enables us to treat in a unified manner
the case of Hermite polynomials ($U = 0$),
of Hermite functions ($\e^{-U/2} = \sqrt{g}$),
as well as other useful cases.
As an example of why such generality can be useful,
it has been shown in~\cite{MR1933042} that the choice $\e^{-U/2} = g$
leads to basis functions,
referred to as generalized Hermite functions in that paper,
that can be used to design an efficient numerical method for the solution of the Kramers Fokker--Planck equation.
By the property~\eqref{eq:hermite_polynomials_1d_as_eigenfunctions},
we see that $(\e^{-U/2} H_i)_{i=0}^{\infty}$ are the eigenfunctions of the operator:
\begin{equation*}
    u \mapsto (\e^{-U/2} \mathcal L \e^{U/2}) u,
\end{equation*}
which is of Schrödinger type when $\e^{-U} = g$, see~\cite{pavliotis2011applied}.
Introducing the notations $\Pi_d^U := \e^{-U/2} \Pi_d \e^{U/2}$ and $\partial_x^U := \partial_x + \derivative*{1}[U]{x}/2 = \e^{-U/2} \partial_x (\e^{U/2} \cdot )$,
we have the following immediate corollary of \cref{thm:hermite_polynomials_approximation}.
\begin{corollary}
    [Approximation by generalized Hermite functions]
    \label{thm:hermite_functions_approximation}
    For any $u$ such that $(\partial_x^U)^\m u \in \lp{2}{\real}[\e^{U}g]$ with $0 \leq \m \leq d + 1$,
    \begin{equation*}
        \label{eq:hermite_convergence_functions}
        \norm{(\partial_x^U)^\l(\Pi_d^U u - u)}_{\e^{U} g} \leq \sqrt{\frac{(d-\m+1)!}{(d-\l+1)!}} \, \norm{(\partial_x^U)^\m u}_{\e^{U} g}, \qquad \l = 0, \dotsc, \m.
    \end{equation*}
\end{corollary}
In the case $\e^{-U} = g$ (orthonormal Hermite functions in $\lp{2}{\real}$),
one can prove a similar statement with the usual derivative instead of $\partial_x^U$ in the left-hand side,
see~\cite[Theorem 7.14]{shen2011spectral}.
We note that \cref{thm:hermite_polynomials_approximation,thm:hermite_functions_approximation} can be extended to the multi-dimensional case,
see e.g.~\cite{abdulle2017spectral}.

In addition to the function $\e^{-U/2}$ multiplying the Hermite polynomials in the definition of basis functions,
it is usual in numerical simulations to introduce a scaling factor,
which can be chosen appropriately depending on how localized the function to be approximated is.
We define $H^{\sigma}_i (x) := H_i(x/\sigma)$,
and note that these polynomials form an orthonormal basis in $\lp{2}{\real}[g_{\sigma}]$,
where $g_{\sigma}$ is the normal distribution with mean 0 and variance $\sigma^2$.
Although we do not present them explicitly,
approximation results similar to~\cref{thm:hermite_polynomials_approximation,thm:hermite_functions_approximation} can be proved
in the presence of this scaling factor,
with the only difference being the presence of additional constant factors on the right-hand side; see, for example,~\cite{abdulle2017spectral}.

\iflong
The importance of choosing this scaling factor in a suitable manner was demonstrated rigorously in~\cite{tang1993hermite},
where the author shows that,
for Gaussian-type functions,
the optimal value of $\sigma$ should depend on the number of basis function used with a dependence of the form $1/\sqrt{d}$.
This can be justified intuitively by taking into account that,
on the one hand,
the behavior of Hermite functions as $x \to \infty$ is well-understood,
with the final inflection point occurring at $x \propto \sqrt{d} \, \sigma$ and a rapid decrease to 0 beyond that point~\cite{MR2474337},
and that, on the other hand,
scaled Hermite functions are eigenfunctions of the Fourier transform operator,
up to dilations/contractions.
Since a contraction in real space results in a dilation in Fourier space,
by choosing a scaling factor that decreases with $d$
one effectively favors exploration in Fourier space.
In particular,
choosing $\sigma \propto 1/\sqrt{d}$ leads to a situation where the position of the final inflection point
of the Hermite function of highest order remains approximately constant
(while the position of the final inflection point of its Fourier transform grows linearly),
so it is crucial in that case to ensure that $\sqrt{d} \, \sigma$ is large enough
to ensure that the basis functions cover the support of the target function.
\fi

In practice,
calculating the Hermite transform numerically requires the introduction of a quadrature.
To bound the associated error,
results similar to \cref{thm:hermite_polynomials_approximation,thm:hermite_functions_approximation},
with the projection operators replaced by interpolation operators,
can be proved; see~\cite[Theorems 7.17, 7.18]{shen2011spectral}.

\section{Proof of \texorpdfstring{\cref{thm:convergence_as_d_goes_to_infinity}}{the convergence theorem}}%
\label{sec:proof_convergence_theorem}

Using the same notation as in
\iflong
\cref{sec:hermite_polynomials},
\else
\cite[Appendix A]{2019arXiv190405973G},
\fi
we let $\Pi_d$ be the $\lp{2}{\real}[e^{-V_q}]$ projection operator on $\poly(d)$
and $\hat \Pi_d := e^{-V_q/2} \, \Pi_d \, e^{V_q/2}$.
The solution $u_d$ of \cref{eq:galerkin_approximation_white_noise} satisfies
$\partial_t u_d = \hat \Pi_d \, \mathcal H_x \, \hat \Pi_d \, u_d =: \mathcal H_d \, u_d$.
Clearly,
the operator $\mathcal H_d$ is selfadjoint on $\e^{-V_q/2} \poly(d)$ with the $\lp{2}{\real}$ inner product,
and it is also negative, by negativity of $\mathcal H_x$:
\begin{equation}
    \label{eq:negativity_of_schrodinger_operator}
    \ip{\mathcal H_d w_d}{w_d} = \ip{\mathcal H_x w_d}{w_d} \leq 0 \quad \forall w_d \in \e^{-V_q/2} \poly(d).
\end{equation}
To prove the convergence of $u_d$ when $d \to \infty$,
we will rely on the following lemma.
\begin{lemma}
    \label{lemma:sobolev-like_inequality}
    Let $\hat \partial_x := \partial_x + x/2$,
    and assume that $\hat \partial_x^n u \in \lp{2}{\real}$ for $n = 0, \dots, \m$.
    Then for all natural numbers $\m_1, \m_2$ such that $\m_1 + \m_2 \leq \m$, it holds that $x^{\m_1} \, \derivative*{\m_2}[u]{x^\m_2} \in \lp{2}{\real}$ and
    \begin{align}
        \label{eq:sobolev-like_inequality_equation}
        K_1(\m) \, \max_{\m_1 + \m_2 \leq \m} \, \norm{x^{\m_1} \, \derivative*{\m_2}[u]{x^\m_2}} \leq \max_{0 \leq i \leq \m} \norm{\hat \partial_x^i u} \leq K_2(\m) \, \max_{\m_1 + \m_2 \leq \m} \, \norm{x^{\m_1} \, \derivative*{\m_2}[u]{x^\m_2}},
    \end{align}
    where $K_1(\m), K_2(\m)$ are positive constants depending only on $\m$ and $\norm{\cdot}$ is the usual $\lp{2}{\real}$ norm.
\end{lemma}
\begin{proof}
    We denote by $\sobolev{\m}{\real}[\e^{-x^2/2}]$ the Sobolev space weighted by $\e^{-x^2/2}$,
    \[
        \sobolev{\m}{\real}[\e^{-x^2/2}] = \{v : \derivative*{i}[v]{x} \in \lp{2}{\real}[\e^{-x^2/2}] \text{ for } i = 0, \dotsc, \m \},
    \]
    and by $\norm{\dummy}[\m,\e^{-x^2/2}]$ the associated norm: $\norm{v}[\m,\e^{-x^2/2}]^2 = \sum_{i=0}^{\m}\norm{\derivative*{i}[v]{x}}[\e^{-x^2/2}]^2$.
    For the first inequality, we know from~\cite[Lemma B.6]{shen2011spectral} that
    \begin{equation*}
        \norm{x v}[\e^{-x^2/2}] \leq 4 \, \norm{v}[1,\e^{-x^2/2}] \qquad  \forall v \in \sobolev{1}{\real}[\e^{-x^2/2}].
    \end{equation*}
    Applying this inequality repeatedly,
    we obtain
    \begin{align}
        \label{eq:lemma_auxiliary_inequality}
        \norm{x^{\m_1} v}[\e^{-x^2/2}] \leq C(\m) \, \norm{v}[\m,\e^{-x^2/2}], \qquad  \m_1 = 0, \dotsc,  \m, \qquad \forall v \in \sobolev{\m}{\real}[\e^{-x^2/2}],
    \end{align}
    for a constant $C(\m)$ depending only on $\m$.
    By definition, $\hat \partial_x u = \e^{-x^2/4}\partial_x (\e^{x^2/4} u)$,
    so the assumption implies that $\e^{x^2/4} u \in \sobolev{\m}{\real}[\e^{-x^2/2}]$,
    from which we obtain using \cref{eq:lemma_auxiliary_inequality} that, for $0 \leq \m_1 \leq \m$,
    \begin{equation*}
        \norm{x^{\m_1} u} = \norm{x^{\m_1} u \, \e^{x^2/4}}[\e^{-x^2/2}]\leq C(\m) \, \norm{u \, \e^{x^2/4}}[\m][\e^{-x^2/2}] = C(\m) \sqrt{\sum_{i=0}^{\m} \norm{\hat \partial_x^\m u}^2}.
    \end{equation*}
    This proves the first inequality of \cref{eq:sobolev-like_inequality_equation} in the case $\m_2 = 0$.
    We assume now that the statement is proved up to $\m_2 - 1$,
    and we show that it is valid for $\m_2$.
    Using the triangle inequality we obtain
    \begin{align*}
        \label{eq:Sobolev-like_inequality}
        \norm{x^{\m_1} \, \derivative*{\m_2}[u]{x^\m_2}} \leq \norm{x^{\m_1} \, (\derivative*{\m_2}[u]{x^\m_2} - \hat \partial_x^{\m_2}u) } + \norm{x^{\m_1} \, \hat \partial_x^{\m_2}u}.
    \end{align*}
    The derivatives in the first term are of order strictly lower than $\m_2$,
    and therefore this term can be bounded by the induction assumption.
    The second term is bounded by applying the base case to $\hat \partial_x^{\m_2} u$:
    introducing $v := \hat \partial_x^{\m_2} u$, we notice that $\hat \partial_x^{m - m_2} v = \hat \partial_x^{m} u\in \lp{2}{\real}$ by assumption,
    so we can apply the first inequality in \cref{eq:sobolev-like_inequality_equation},
    without any derivative of $v$ in the left-hand side, to deduce
    \[
        \norm{x^{m-m_2} v} \leq \max_{0 \leq i \leq m - m_2} \norm{\hat \partial_x^i v} \leq \max_{0 \leq i \leq m} \norm{\hat \partial_x^i u}.
    \]
    The second inequality in~\eqref{eq:sobolev-like_inequality_equation} then holds trivially by expanding $\hat \partial_x$ and applying a triangle inequality.
\end{proof}

With \cref{assumption:potential},
we can show that the two norms in \cref{lemma:sobolev-like_inequality} can be bounded from above
by the norm $\sqrt{\ip{{(- \mathcal H_x + 1)}^\m u}{u}}$ for appropriate $\m$.
\begin{lemma}
    [Bound by alternative norm]
    \label{lemma:bound_alternative_norm}
    If \cref{assumption:potential} holds, then
    \begin{equation*}
        \label{eq:bound_alternative_norm}
        \sum_{i=0}^{\m} {\norm{\hat \partial_x^\m u}}^2 \leq C \, \ip{(-\mathcal H_x + 1)^\m u}{u}
    \end{equation*}
    for any smooth $u$ for which the right-hand side is well-defined.
    Here $C$ is a positive constant that depends on $\beta$, $\m$,
    and on the particular expression of the potential $W$ defined in \cref{assumption:potential}.
\end{lemma}
\begin{proof}
Below $C_1$ and $C_2$ denote the same constants as in \cref{assumption:potential}.
First we notice that, for any constant $K > 1$,
\begin{equation}
    \label{eq:proof_auxiliary}
    \ip{(-\mathcal H_x + K)^\m u}{u} \leq K^\m \ip{(-\mathcal H_x + 1)^\m u}{u},
\end{equation}
because $\mathcal H_x$ is a negative operator.
Since $W$ is a polynomial,
its derivatives grow asymptotically more slowly that $W$  itself,
and so it is possible for any $\varepsilon > 0$ to find $K \geq C_2$ large enough that
\begin{equation}
    \label{eq:proof:bound_derivative_of_W}
    \abs{\derivative*{i}[W]{x^i}(x)} \leq \varepsilon \, (W(x) + K) \qquad \forall x \in \real, \qquad i = 1, 2, \dotsc.
\end{equation}
\revision{%
    For this proof to go through, it is in fact sufficient that this inequality be satisfied for $i = 1, \dotsc, m$.
}
We decompose $-\mathcal H_x + K$ as $(- \beta^{-1} \derivative{2}{x^2}) + (W(x) + K)$.
The two operators in this sum are positive because $K \geq C_2$ and by assumption $W(x) + C_2 \geq C_1 (1 + \abs{x}^2)$.
Expanding the inner product in the left-hand side of \cref{eq:proof_auxiliary} and using integration by parts,
\begin{align}
    \label{eq:proof:expansion_sum_of_operators}
    \ip{(-\mathcal H_x + K)^\m u}{u} = \left(\sum_{i=0}^{\m} \beta^{-i} \binom{m}{i} \, \int_{\real}{(W(x) + K)}^{i} \, {(\derivative*{\m-i}[u]{x}(x))}^2 \, \d x\right) + \dotsb,
\end{align}
where the remainder terms originate from the fact that
the operators $\derivative{1}{x}$ and $(W(x) + K)$ do not commute.
By \cref{eq:proof:bound_derivative_of_W},
these terms can be bounded for sufficiently large $K$ by half the leading term in \cref{eq:proof:expansion_sum_of_operators}.
To conclude,
we further expand this leading term:
\begin{align*}
    \ip{(-\mathcal H_x + K)^\m u}{u} &\geq \frac{1}{2} \sum_{i=0}^{\m} \binom{m}{i} \beta^{-i}\,\int_{\real}{(W(x) + K)}^{i} \, {(\derivative*{\m-i}[u]{x}(x))}^2 \, \d x \\
                                     &\geq \frac{1}{2} \sum_{i=0}^{\m} \binom{m}{i} C_1^i \, \beta^{-i}\,\int_{\real} \, {(1 +  x^2)}^{i} \, {(\derivative*{\m-i}[u]{x}(x))}^2 \, \d x \\
                                     &\geq \frac{1}{2} \sum_{i=0}^{\m}\sum_{j=0}^{i} \binom{m}{i} \binom{i}{j} \, C_1^i \, \beta^{-i} \, \, \int_{\real}  x^{2j} \, {(\derivative*{\m-i}[u]{x}(x))}^2 \,\d x \\
                                      &\geq C(\m, \beta, C_1) \,\sum_{\m_1 + \m_2 \leq \m} \norm{x^{\m_1} \, \derivative*{\m_2}[u]{x}}^2,
\end{align*}
from which \cref{lemma:sobolev-like_inequality} allows us to conclude.
\end{proof}
\begin{proof}[Proof of \cref{thm:convergence_as_d_goes_to_infinity}]
We assume for simplicity that $\sigma = 1$,
and we begin by splitting the error as $u_d - u = (u_d - \hat \Pi_d u) + (\hat \Pi_d u - u) =: e_d + \delta_d$.
The first term is related to the so-called consistency error,
and the second to the approximation error.
We obtain from~\cref{eq:schrodinger_equation,eq:galerkin_approximation_white_noise}
\begin{align*}
    \textstyle \derivative{1}[e_d]{t} =  \hat \Pi_d \mathcal H_x \hat \Pi_d e_d + (\hat \Pi_d \mathcal H_x \hat \Pi_d - \hat \Pi_d \mathcal H_x) u.
\end{align*}
Taking the inner product with $e_d$ and using~\eqref{eq:negativity_of_schrodinger_operator},
this implies
\begin{align*}
    \textstyle \ip{\derivative{1}[e_d]{t}}{e_d}
    &\leq \ip{\mathcal H_x (\hat \Pi_d \,u - u)}{e_d} \\
    &\leq \frac{\alpha}{2} \ip{e_d}{e_d} + \frac{1}{2 \alpha} \ip{\mathcal H_x^2 (u - \hat \Pi_d \, u)}{(u - \hat \Pi_d \, u)} \qquad \forall \alpha > 0,
\end{align*}
where we used Young's inequality.
We see from this equation that $e_d$ can be controlled if one can bound the second inner product on the right-hand side.
For this we use arguments similar to the ones employed in~\cite{abdulle2017spectral,gagelman2012spectral}.
Since $\mathcal H_x$ is negative and selfadjoint,
we notice
\begin{align}
    \ip{(-\mathcal H_x)^i u(t)}{u(t)}
    &= \ip{(-\mathcal H_x)^i u_0}{u_0} + \int_{0}^t \derivative*{1}{s} \ip{(-\mathcal H_x)^i u(s)}{u(s)} \, \d s \nonumber \\
    &= \ip{(-\mathcal H_x)^i u_0}{u_0} - 2\int_{0}^t \, \ip{(-\mathcal H_x)^{i+1}u(s)}{u(s)} \, \d s \nonumber \\
    \label{eq:proof:energy_estimate_linear}
    &\leq \ip{(-\mathcal H_x)^i u_0}{u_0},
\end{align}
for $i = 1, 2, \dots$,
which implies that the inner products $\ip{(-\mathcal H_x)^i u}{u}$ remain bounded for all positive times.
We can now apply
\iflong
\cref{thm:hermite_functions_approximation}
\else
\cite[Corollary A.2]{2019arXiv190405973G}
\fi
to obtain,
using \cref{lemma:sobolev-like_inequality,lemma:bound_alternative_norm} and \cref{assumption:potential},
\begin{align*}
\ip{\mathcal H_x^2 (u - \hat \Pi_d \, u)}{(u - \hat \Pi_d \, u)}
&\leq C \, \sum_{i=0}^{2\,k} \, \norm{\hat \partial_x^i (u - \hat \Pi_d \, u)}^2 \\
&\leq C \, {\frac{(d-\m+1)!}{(d-2k+1)!}} \, \norm{\hat \partial_x^\m u}^2 \\
&\leq C \, {\frac{(d-\m+1)!}{(d-2k+1)!}} \, \ip{(-\mathcal H_x + 1)^\m u}{u} \\
&\leq C \, {\frac{(d-\m+1)!}{(d-2k+1)!}} \, \ip{(-\mathcal H_x + 1)^\m u_0}{u_0}.
\end{align*}
We note that when $V$ is quadratic, $k = 1$ is a valid choice in \cref{assumption:potential},
and the bound above can be obtained by simply expanding $u$ in terms of the eigenfunctions of $\mathcal H_x$,
which in that case are just rescaled Hermite functions.
Using Grönwall's inequality,
we finally obtain
\begin{align}
    \nonumber%
    \norm{e_d(t)}^2
    &\leq \e^{\alpha t} \, \norm{e_d(0)}^2 + \int_{0}^{t}
    \e^{\alpha(t - s)} \, \ip{\mathcal H_x^2 (u - \hat \Pi_d \, u)}{(u - \hat \Pi_d \, u)} \d s, \\
    \label{eq:proof:bound_ed}%
    &\leq \e^{\alpha t} \, \left( \norm{e_d(0)}^2 + C_{\alpha} \, {\frac{(d-\m+1)!}{(d-2k+1)!}} \right).
\end{align}
The first term, proportional to $\norm{e_d(0)}^2$, depends only on
the interpolation error of the initial condition,
which is nonzero when using a Gauss--Hermite quadrature.
It was proved that this error term also decreases spectrally,
see e.g.~\cite[Theorems 7.17, 7.18]{shen2011spectral},
and in our case faster than the second error term.
For the approximation error $\delta_d$,
similar inequalities to the ones used above can be used to obtain a bound of the type~\eqref{eq:proof:bound_ed},
which leads to the conclusion.
\end{proof}

\begin{remark}
\label{remark:stationary_error}
\revision{
    As mentioned in \cref{remark:not_optimal},
    \cref{thm:convergence_as_d_goes_to_infinity} is not optimal.
    It leaves open, in particular, the question of precisely how the error behaves as $t \to \infty$.
    In this remark, we give a partial answer to the question:
    we show how a bound on the stationary error can be obtained,
    under the assumption that the solution to the discretized equation is rescaled in time in such a way that
    the integral of $\rho_d := \sqrt{\rho_s} \, u_d$ remains equal to $1$.
    With this rescaling and with the notation $S_d := \hat \Pi_d (L^2(\real))$,
    the stationary solution of the discretized (in space) equation is
    \[
        \hat u_s :=  \argmax_{\psi_d \in S_d, \int_{\real} \sqrt{\rho_s} \, \psi_d  = 1}  \ip{\mathcal H_x \psi_d}{\psi_d}.
    \]
    Taking
    \[
        \psi_d = \frac{\hat \Pi \sqrt{\rho_s}}{\int_{\real} \sqrt{\rho_s} \hat \Pi \sqrt{\rho_s}},
    \]
    we deduce
    \begin{align*}
        - \ip{\mathcal H_x \hat u_s}{\hat u_s} \leq \,
        & - \left| \frac{1}{\int_{\real} \sqrt{\rho_s} \hat \Pi \sqrt{\rho_s}} \right|^2 \ip{\mathcal H_x (\hat \Pi \sqrt{\rho_s})}{\hat \Pi \sqrt{\rho_s}}.
    \end{align*}
    Since $\mathcal H_x$ is self adjoint in $\lp{2}{\real}$ and $\mathcal H_x \sqrt{\rho_s} = 0$,
    it follows that
    \begin{align*}
        - \ip{\mathcal H_x (\hat u_s - \sqrt{\rho_s})}{\hat u_s - \sqrt{\rho_s}} \leq \,
        & - \left| \frac{1}{\int_{\real} \sqrt{\rho_s} \hat \Pi \sqrt{\rho_s}} \right|^2 \ip{\mathcal H_x (\hat \Pi \sqrt{\rho_s}  - \sqrt{\rho_s})}{\hat \Pi \sqrt{\rho_s}  - \sqrt{\rho_s}}.
    \end{align*}
    An approximation argument similar to the one above can be employed to show that the right-hand side,
    and therefore also the left-hand side,
    decrease to zero as $d \to \infty$ faster than $d^{-n}$ for any $n > 0$.
    To conclude, a Poincaré-type inequality can be invoked,
    which is justified because
    \[
        \int_{\real} (\hat u_s - \sqrt{\rho_s}) \, \sqrt{\rho_s} = 1 - 1 = 0,
    \]
    to obtain a bound of the type
    \[
        \norm{\hat u_s - \sqrt{\rho_s}} \leq  - C \ip{\mathcal H_x (\hat u_s - \sqrt{\rho_s})}{\hat u_s - \sqrt{\rho_s}}.
    \]
}
\end{remark}

\section{Benchmark tests for the spectral numerical method}%
\label{sec:numerical_tests}
In this section,
we investigate the performance of the spectral method through numerical experiments.

\subsection{Linear Fokker--Planck equation with colored noise}%
We focus first on the Galerkin approximation~\eqref{eq:galerkin_approximation_colored_noise}.
Here we consider only the cases
where $V(\cdot)$ is a quadratic or a bistable potential
and where the noise is described by an OU process,
but results of additional numerical experiments,
corresponding to harmonic noise and non-Gaussian noise,
are presented in~\cite{thesis_urbain}.
\begin{figure}[ht]
    \centering%
    \begin{minipage}[b]{\linewidth}
        \centering%
        \includegraphics[width=.43\linewidth]{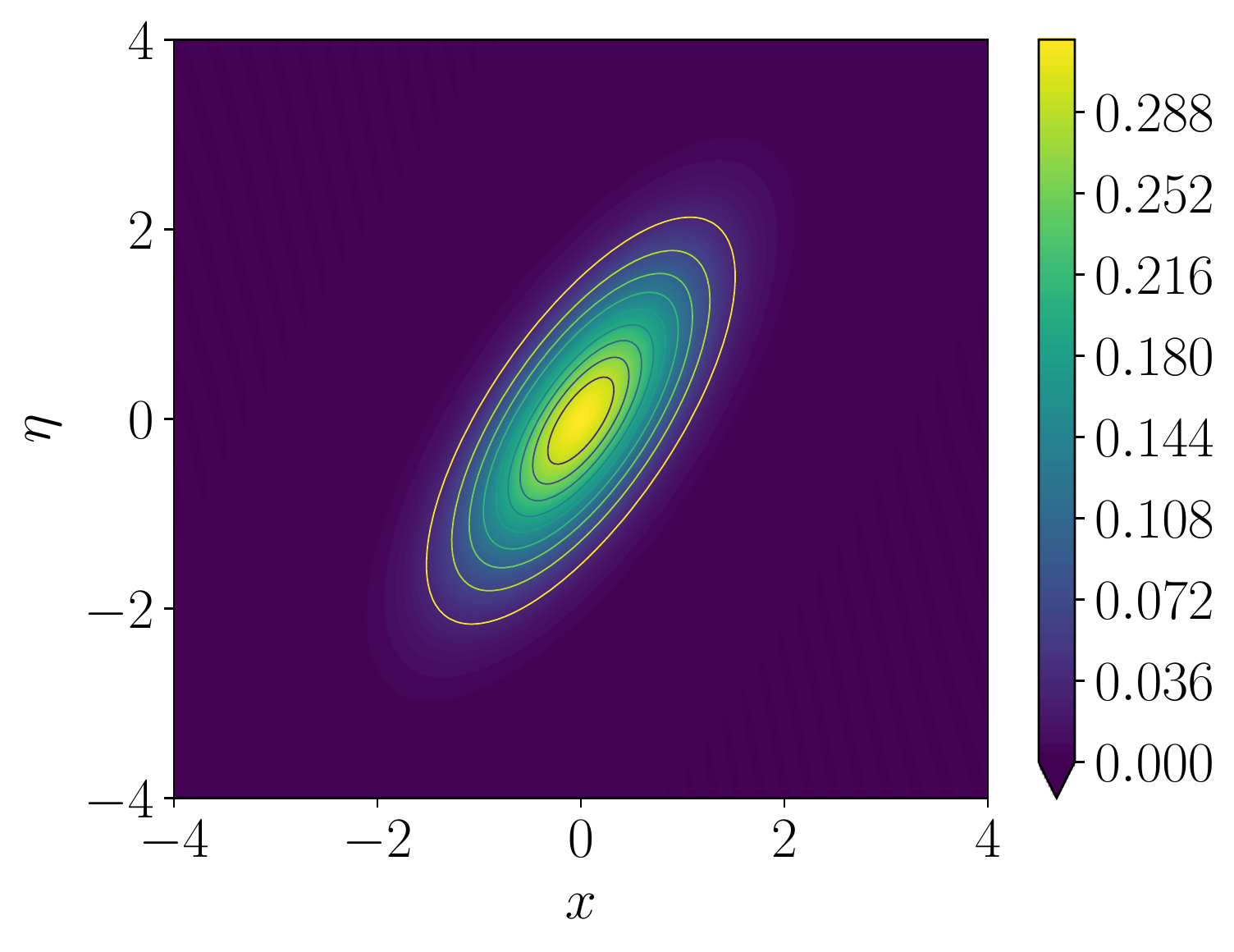}%
        \includegraphics[width=.42\linewidth]{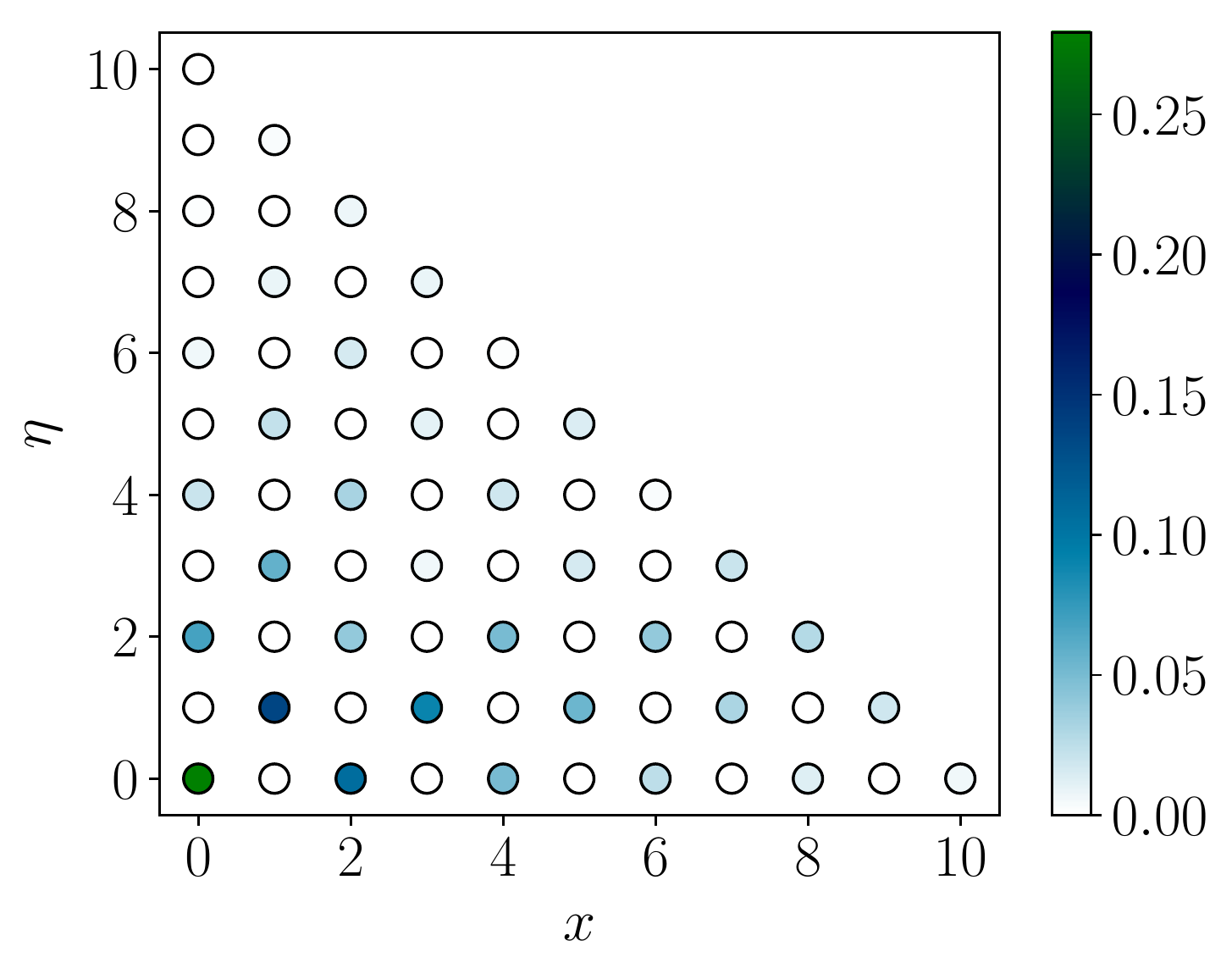}%
        \subcaption{%
            Steady-state solution of the Fokker--Planck equation \eqref{eq:galerkin_approximation_colored_noise} with the associated field lines of the probability flux (left)
            and absolute value of the coefficients of degree less than equal to $10$ in the Hermite expansion (right).
        }%
        \label{fig:solution_density_gaussian_case}
    \end{minipage}
    \begin{minipage}[b]{\linewidth}
        \centering%
        \includegraphics[width=.42\linewidth]{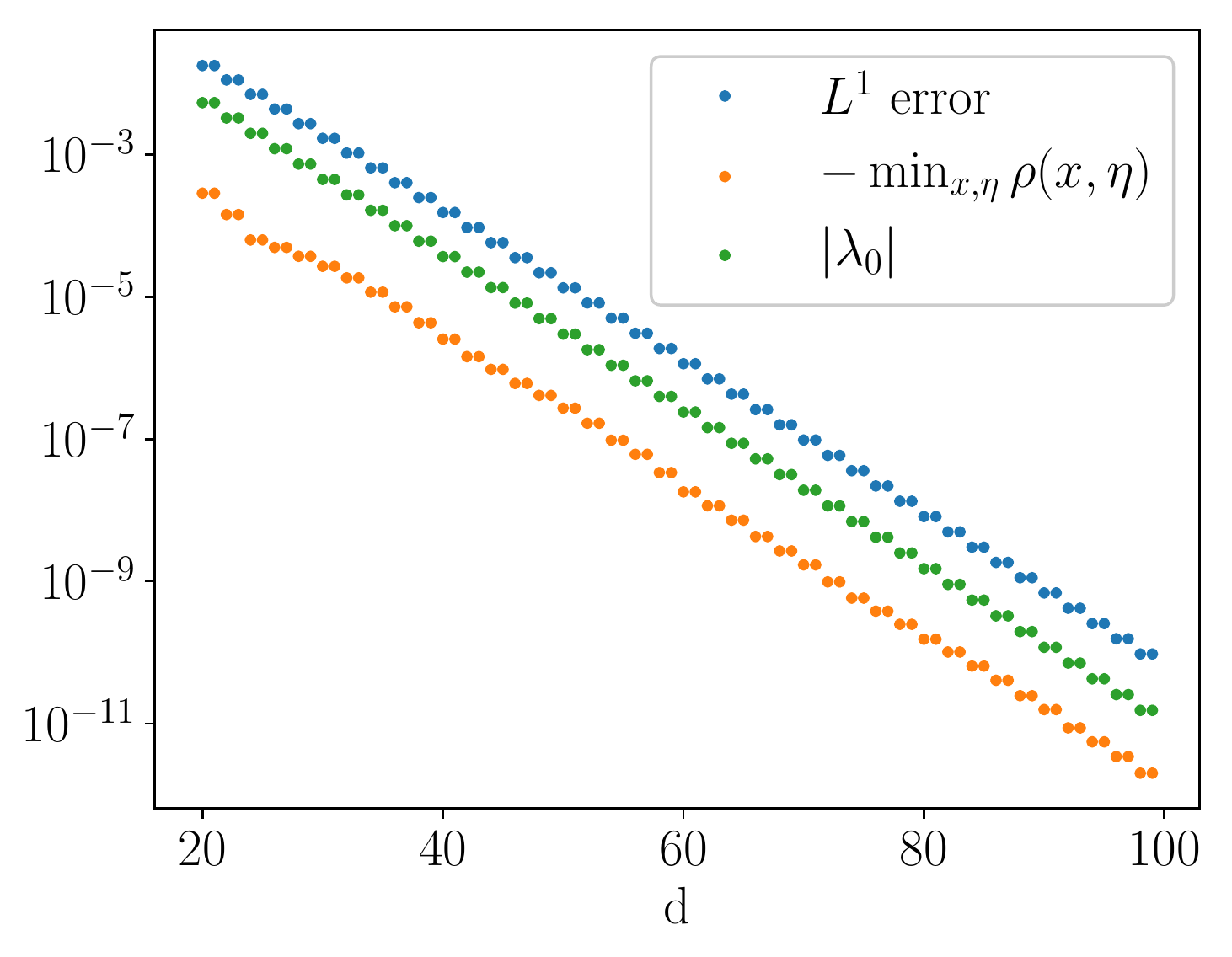}%
        \subcaption{%
            Convergence of the method, using three different metrics for the error:
            the $L^1$ norm of the error between the numerical and exact solutions,
            the negative of the minimum of the numerical solution,
            and the absolute value of the eigenvalue with smallest real part.
        }%
        \label{fig:convergence_gaussian_case}
    \end{minipage}%

    \caption{%
        Simulation data when $V(x) = x^2/2$.
    }%
\end{figure}

We start with the case $V(x) = x^2/2$,
for which the exact solution to the Fokker--Planck equation~\eqref{eq:fokker_planck_colored_noise} can be calculated explicitly by substitution of a Gaussian ansatz,
see~\cite[Section 3.7]{pavliotis2011applied}.
We study the convergence of the steady-state solution,
obtained by calculating the eigenfunction associated with the eigenvalue of lowest magnitude of $\hat \Pi_d \, \mathcal L_{\varepsilon}^* \, \hat \Pi_d$,
where $\hat \Pi_d$ is the $\lp{2}{\real^2}[\e^U]$ projection operator on $S_d$,
directly using the method \verb?eigs? from the \emph{SciPy} toolbox.
The parameters used for this simulation are the following:
$\beta = \varepsilon = 1$, $\sigma_x^2 = \frac{1}{10}$, $\sigma_{\eta}^2 = 1$, $\e^{-U(x, \eta)/2} = \e^{-V_{\eta}(\eta)/2} = \e^{-\eta^2/4}$.
With these parameters,
the steady-state solution to \cref{eq:fokker_planck_colored_noise} is equal to $\rho_{\infty}(x, \eta) = \e^{-2x^2 + 2x\eta - \eta^2} / \pi$,
and clearly $\rho_{\infty} \in \lp{2}{\real^2}[\e^U]$.
\Cref{fig:solution_density_gaussian_case}
presents the steady-state solution,
obtained using the spectral method with Hermite polynomials up to degree 100 ($d = 100$) and a triangular index set,
and \cref{fig:convergence_gaussian_case} presents the convergence of the method.
Since the solution satisfies $\rho_{\infty}(x, \eta) = \rho_{\infty}(- x, -\eta)$,
the Hermite coefficients corresponding to even values of $i + j$ are zero,
where $i$ and $j$ are the indices in the $x$ and $\eta$ directions, respectively.

Now we consider that $V$ is the bistable potential $x^4/4 - x^2/2$,
which was solved numerically in~\cite{hanggi1994colored} using generalized Hermite functions and a variation of the matrix continued fraction technique.
For this case an explicit analytical solution is not available.
The parameters we use are the following:
$\beta = 1$, $\varepsilon = \frac{1}{2}$, $\sigma_x^2 = \frac{1}{20}$, $\sigma_{\eta}^2 = 1$.
Through numerical exploration,
we noticed that a good convergence could be obtained by using the
multiplier function $\e^{-U(x,\eta)/2} = \e^{- \beta V(x)/2 - \eta^2/4}$,
rather than just $\e^{-\eta^2/4}$ in the previous paragraph.
We note that this would have been the natural choice
if the noise in the $x$ direction had been white noise.
The solution obtained using a square-shaped index set and $d = 100$,
as well as the corresponding Hermite coefficients up to degree 10,
is illustrated in~\cref{fig:solution_density_bistable_case}.
We observe that the Hermite coefficients corresponding to the degree 0 in the $\eta$ direction
(i.e.\ to the basis function $\e^{-\eta^2/2}$)
are significantly larger than the other coefficients,
which is consistent with the fact that, as $\varepsilon \to 0$,
the steady-state solution approaches $\e^{-\beta V(x)} \, \e^{-\eta^2/2}$ (up to a constant factor).
The associated convergence curves are presented in \cref{fig:bistable_potential_convergence_curves}.
\revision{
    We observe that the convergence is exponential,
    which is better than the rate of convergence predicted in~\cref{thm:convergence_as_d_goes_to_infinity}.
}
\begin{figure}[ht]
    \begin{minipage}[b]{\linewidth}
        \centering%
        \includegraphics[width=.42\linewidth]{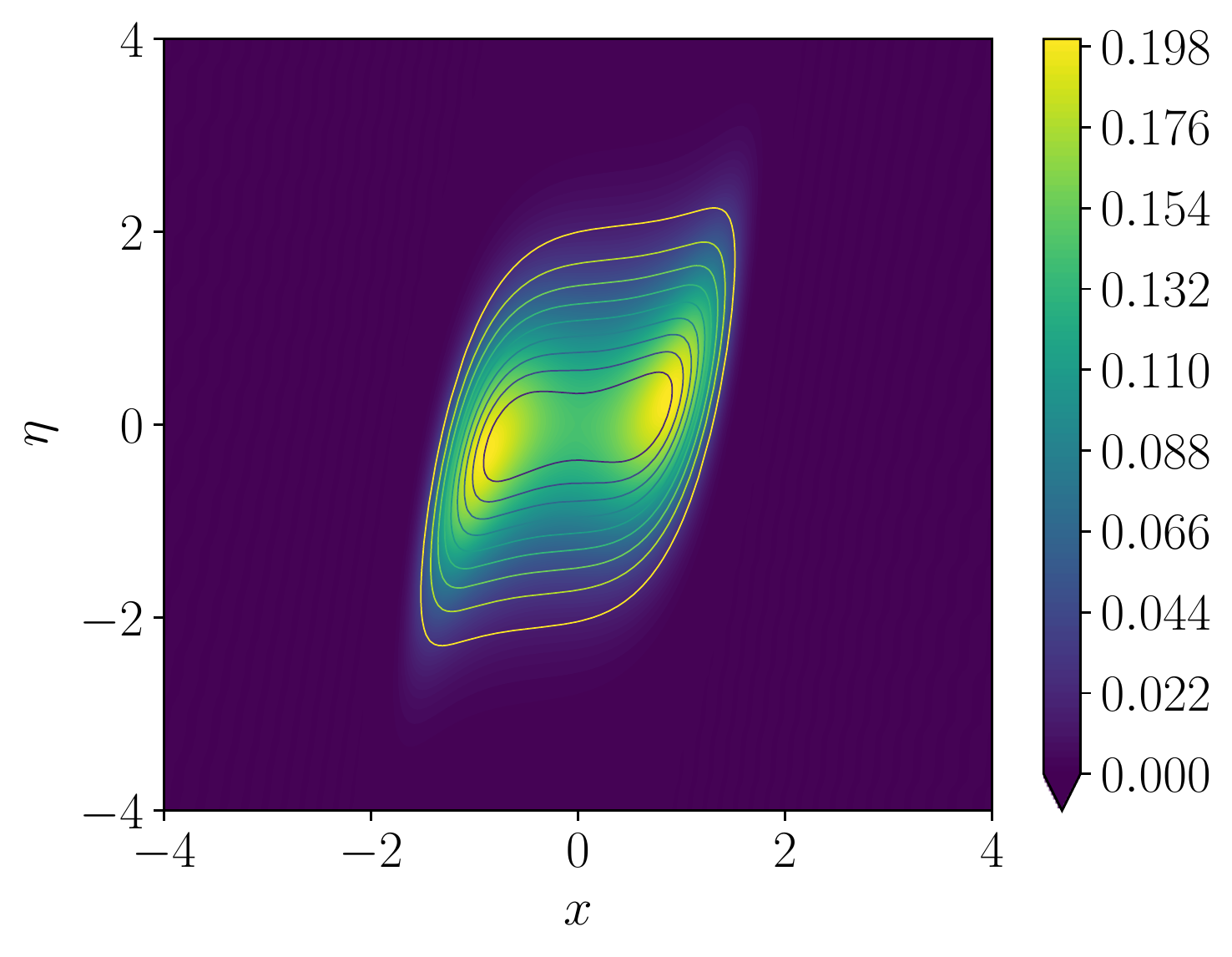}
        \includegraphics[width=.42\linewidth]{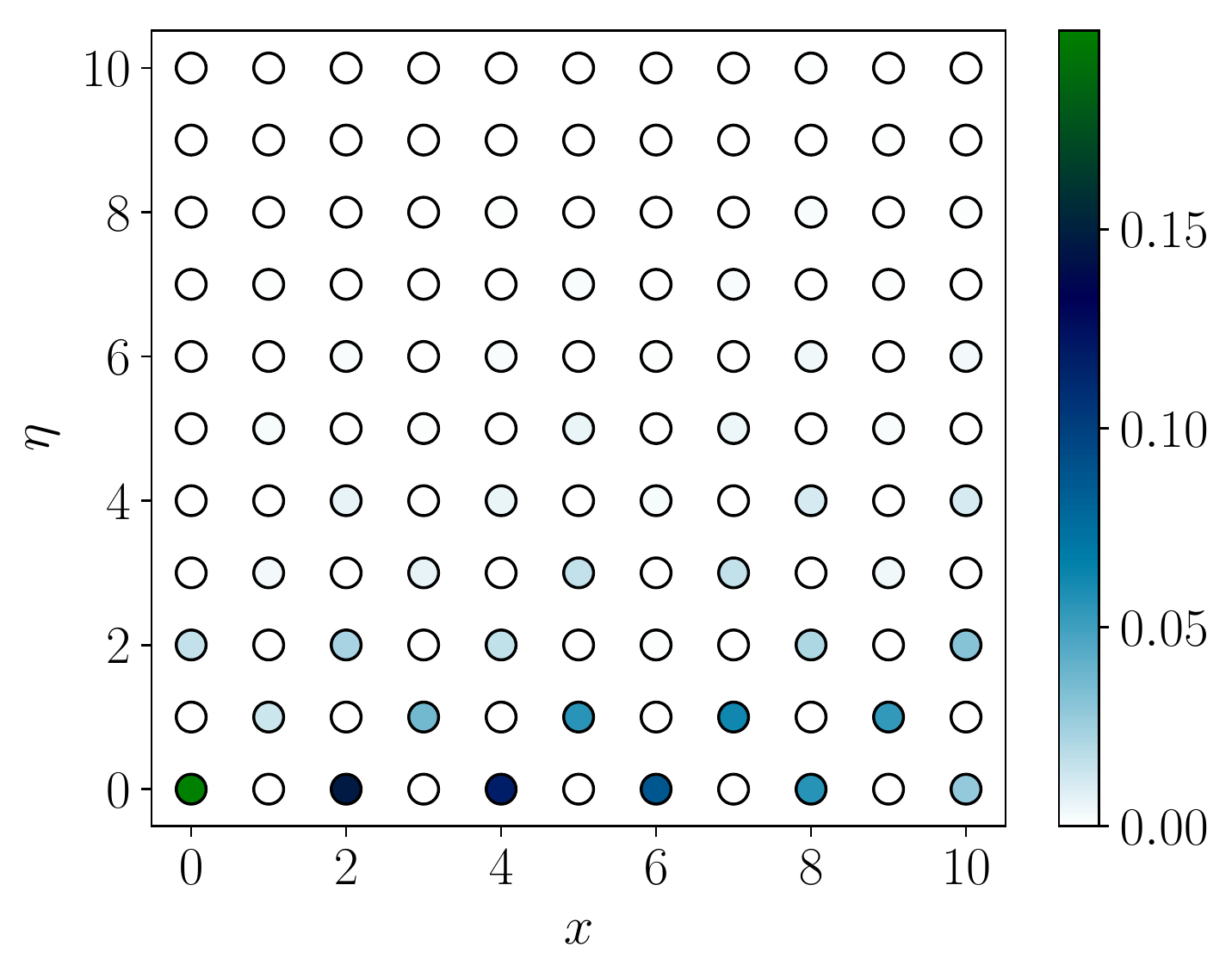}
        \subcaption{%
            Steady-state solution to \cref{eq:galerkin_approximation_colored_noise}
            and associated field lines of the probability flux (left),
            and absolute value of the coefficients of degree less than or equal to $10$ in the Hermite expansion (right).
        }%
        \label{fig:solution_density_bistable_case}
    \end{minipage}
    \begin{minipage}[b]{\linewidth}
        \centering%
        \includegraphics[width=.42\linewidth]{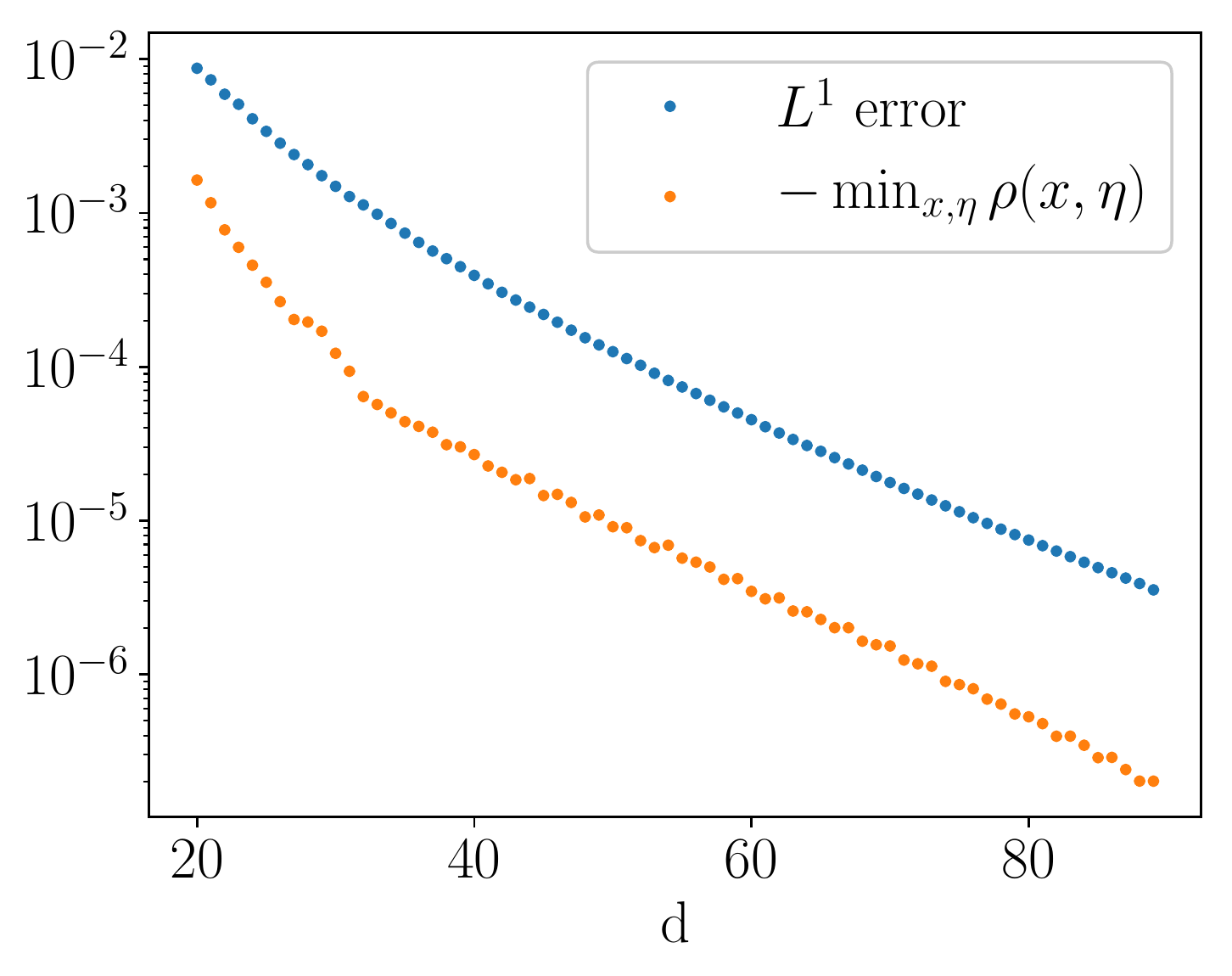}
        \includegraphics[width=.42\linewidth]{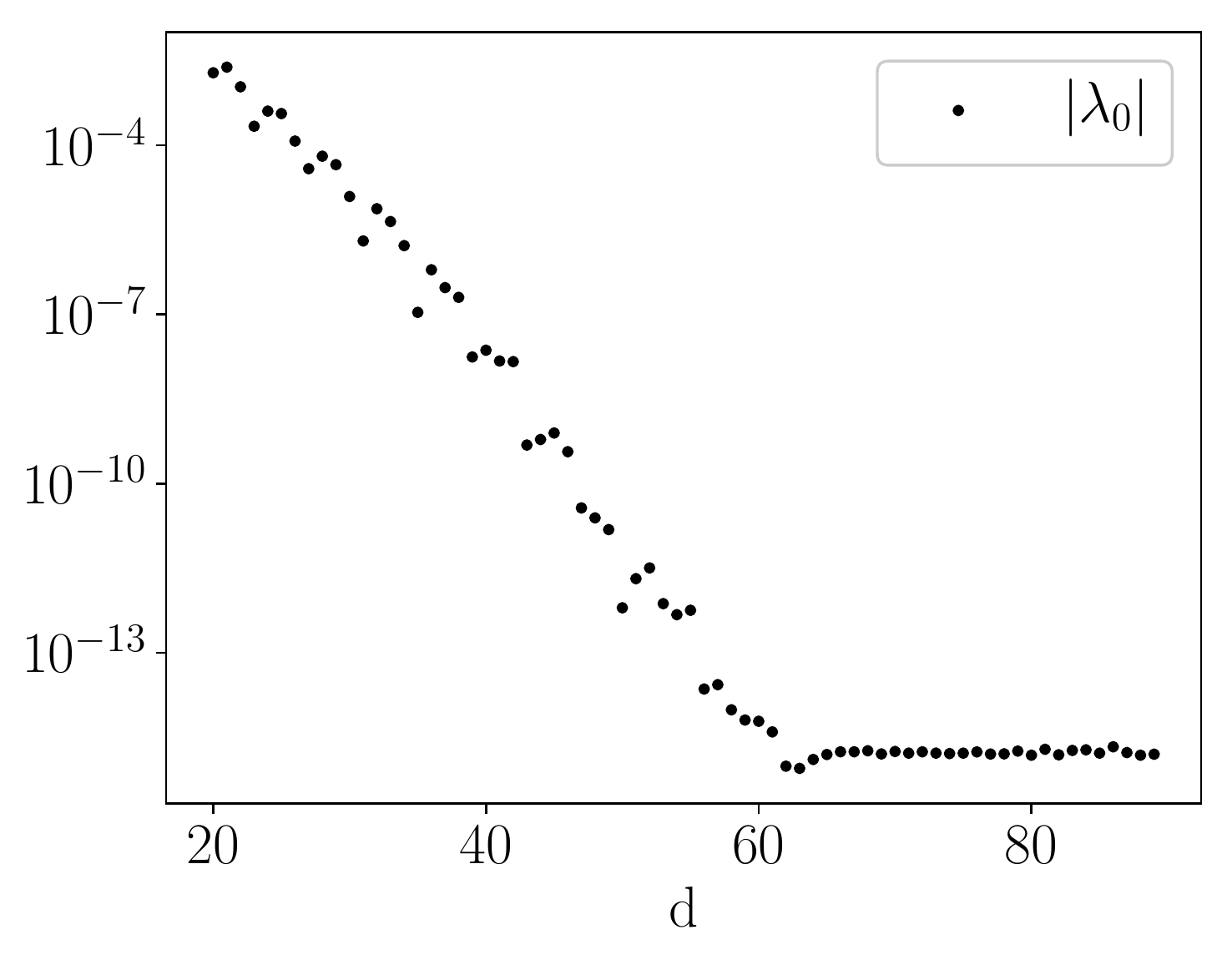}
        \subcaption{%
            Convergence of the method using the same measures of the error as in \cref{fig:solution_density_gaussian_case},
            except that the $L^1$ error is calculated by comparison with the numerical solution obtained when $d = 100$.
        }%
        \label{fig:bistable_potential_convergence_curves}
    \end{minipage}%
    \caption{%
        Simulation data when $V(x) = x^4/4 - x^2/2$.
    }%
\end{figure}

\subsection{McKean--Vlasov equation with colored noise}%
We focus now on the Galerkin approximation to~\eqref{eq:mckean_vlasov_colored_noise}.
When the potential $V(\cdot)$ is quadratic and the initial condition is Gaussian,
it is well-known that the McKean--Vlasov equation has an explicit solution and that this solution is Gaussian.
We assume that $V(x) = x^2/2$ and
we rewrite \cref{eq:mckean_vlasov_colored_noise} in the formalism of~\cite{duongmean}, as
\begin{equation*}
    \label{eq:mckean_vlasov_colored_noise_matrix_form}
    \derivative{1}[\rho]{t} = - \grad \cdot \left(
        B \, \vect x \, \rho
        + \int_{\real^{2}} K(\vect x - \vect x') \, \rho(\vect x', t ) \, \d \vect x' \, \rho
        - D \, \grad \rho \right),
\end{equation*}
where $\vect x = (x, \eta)^T$ and
\begin{equation*}
    \label{eq:mckean_vlasov_colored_noise_matrix_form:parameters}
    B = \begin{pmatrix} -1 & \varepsilon^{-1} \beta^{-1/2} \\ 0 & -\varepsilon^{-2} \end{pmatrix},
    \quad K = \begin{pmatrix} -\theta & 0 \\ 0 & 0 \end{pmatrix},
    \quad D = \begin{pmatrix} 0 & 0 \\ 0 & \varepsilon^{-2} \end{pmatrix}.
\end{equation*}
Adapting \cite[Proposition 2.3]{duongmean} to our case,
we deduce that the solution is of the type
\begin{equation*}
    \label{eq:exact_solution_mckean_vlasov_gaussian_case}
    \rho(\vect x, t) = \frac{1}{(2\pi)\, \det{\Sigma(t)}} \, \exp \left( - \frac{1}{2} \, (\vect x - \mu(t))^T \Sigma^{-1}(t)(\vect x - \mu(t)) \right),
\end{equation*}
where $\mu(t)$ and $\Sigma(t)$ are given by
\begin{equation}
    \label{eq:exact_solution_mckean_vlasov_gaussian_case_mean_cov}
    \mu(t) = \e^{Bt} \, \mu(0),
    \quad \Sigma(t) = \e^{t(B+K)} \, \Sigma(0) \, \e^{t{(B+K)}^T} +  2 \, \int_{0}^{t} \e^{s(B + K)} \, D \,\e^{s{(B + K)}^T} \, \d s.
\end{equation}
This solution can be obtained by introducing $g = - \ln \rho$,
rewriting \cref{eq:mckean_vlasov_colored_noise} as an equation for $g$,
and using a quadratic ansatz for $g$.
The eigenvalue decomposition of $B + K$ is
\begin{equation*}
    \label{eq:mckean_vlasov_colored_noise_eigenvectors_and_eigenvalues}
    (B + K) \begin{pmatrix} 1 & - \varepsilon \\ 0 & \sqrt{\beta} \, (1 - \varepsilon^2(1+\theta)) \end{pmatrix}
    = \begin{pmatrix} 1 & - \varepsilon \\ 0 & \sqrt{\beta} \, (1 - \varepsilon^2(1+\theta)) \end{pmatrix} \,
    \begin{pmatrix} - 1 - \theta & 0 \\ 0 & - \varepsilon^{-2} \end{pmatrix},
\end{equation*}
which enables the explicit calculation of the integral in the expression of $\Sigma(t)$.
From \cref{eq:exact_solution_mckean_vlasov_gaussian_case_mean_cov} and the structure of $B$ and $K$,
we notice that, as $t \to \infty$, $\mu \to 0$ and
\begin{equation*}
    \label{eq:mckean_vlasov_colored_noise_covariance_matrix_equilibrium}
    \Sigma(t) \to \Sigma_{\infty} = 2 \int_{0}^{\infty}\e^{s(B + K)} \, D \, \e^{s{(B + K)}^T} \, \d s,
\end{equation*}
which coincides with the solution of the steady state linear Fokker--Planck equation corresponding to the McKean--Vlasov equation when $m$ is a parameter equal to 0.
For this test case,
we use the following parameters: $\beta = \theta = 1$, $\varepsilon = 1/2$, $\sigma_x^2 = \sigma_{\eta}^2 = 1/5$, $\e^{-U(x,\eta)/2} = \e^{-V(x)/2} \, \e^{- V_{\eta}(\eta)/2}$.
The initial condition is taken to be the Gaussian density $\mathcal N\left((1, 1)^T, I_{2\times2}\right)$.
The evolution of the probability density is illustrated in \cref{fig:mckean_vlasov_colored:snapshots_of_the_solution},
and the convergence of the method, in the $\lp{\infty}{0, T; \lp{1}{\real^2}}$ norm,
is illustrated in \cref{fig:mckean_vlasov_colored:convergence}.

\begin{figure}[ht]
    \centering
    \includegraphics[width=0.24\linewidth]{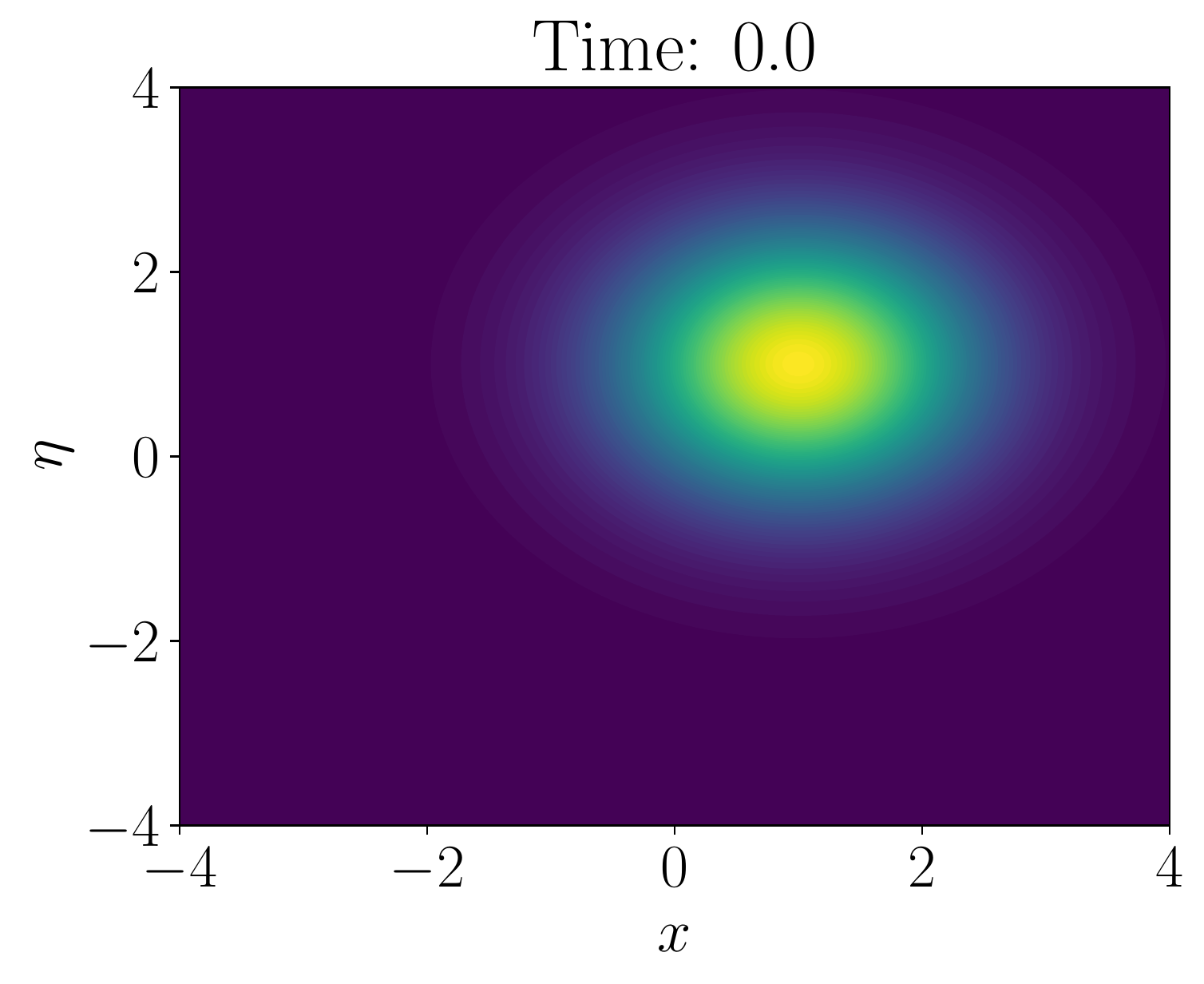}
    \includegraphics[width=0.24\linewidth]{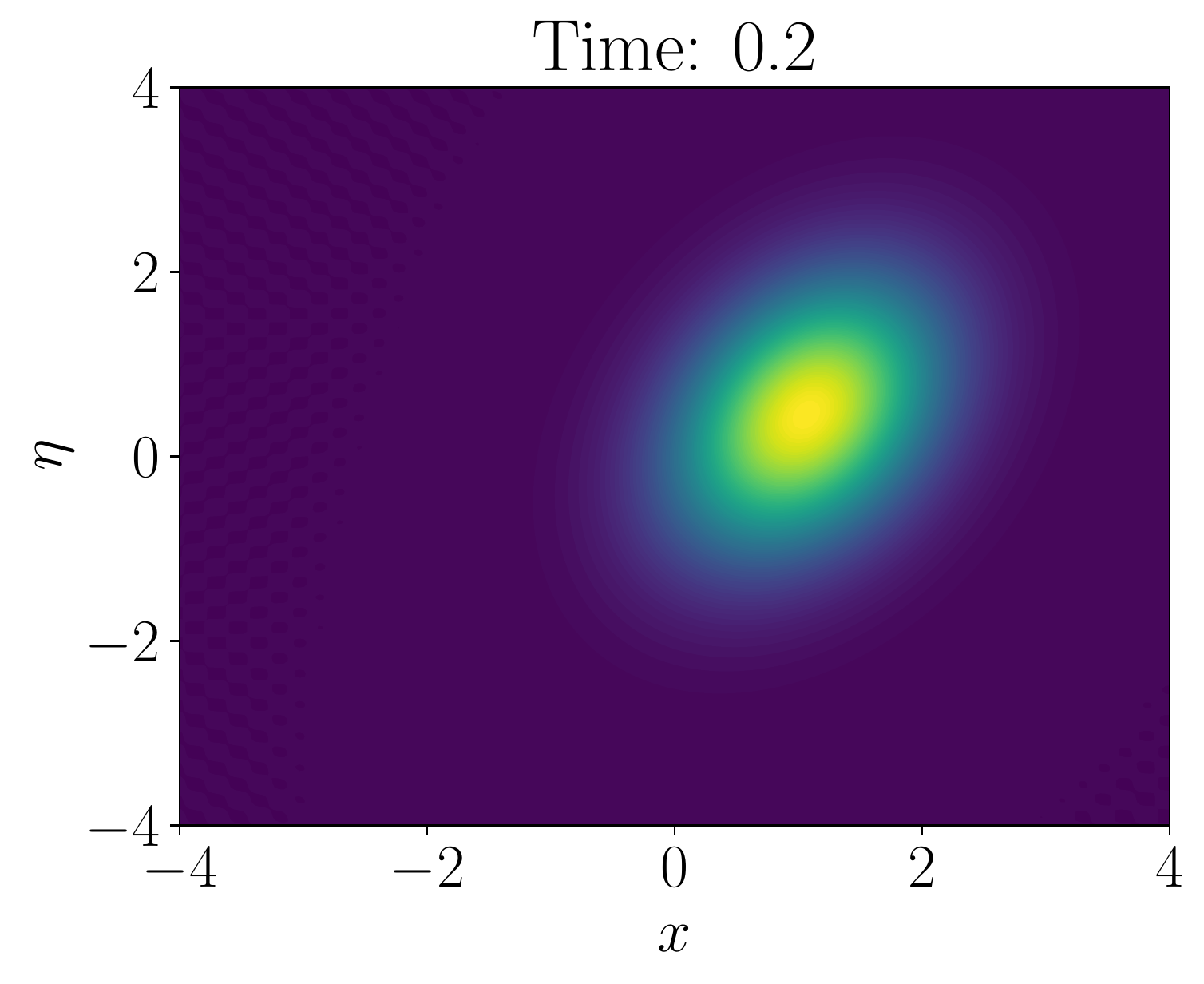}
    \includegraphics[width=0.24\linewidth]{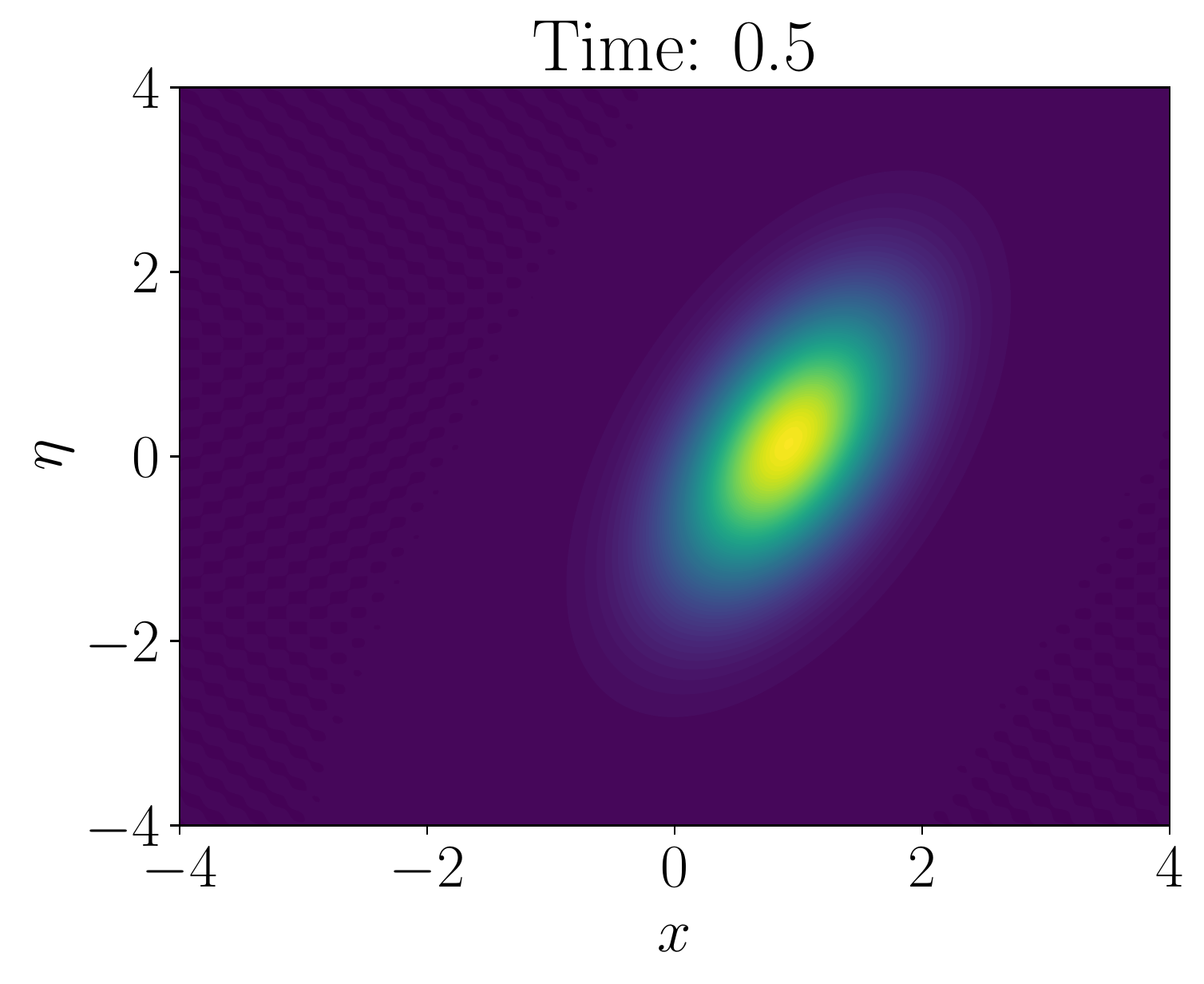}
    \includegraphics[width=0.24\linewidth]{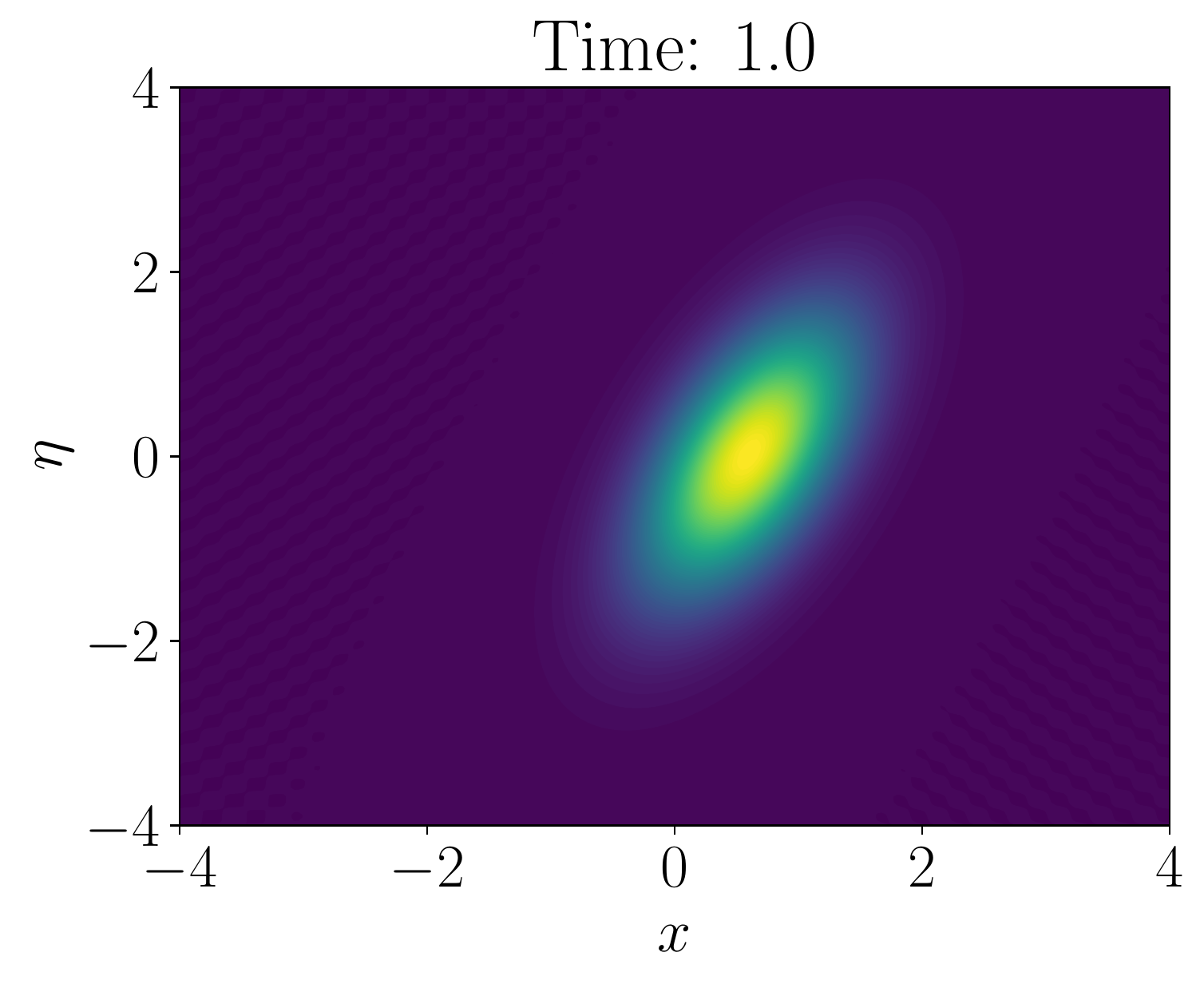}
    \caption{Probability density solution of \cref{eq:mckean_vlasov_colored_noise} (obtained using the spectral method) at times $0, 0.2, 0.5, 1$.}
    \label{fig:mckean_vlasov_colored:snapshots_of_the_solution}
\end{figure}

\begin{figure}[ht]
    \centering
    \includegraphics[width=0.6\linewidth]{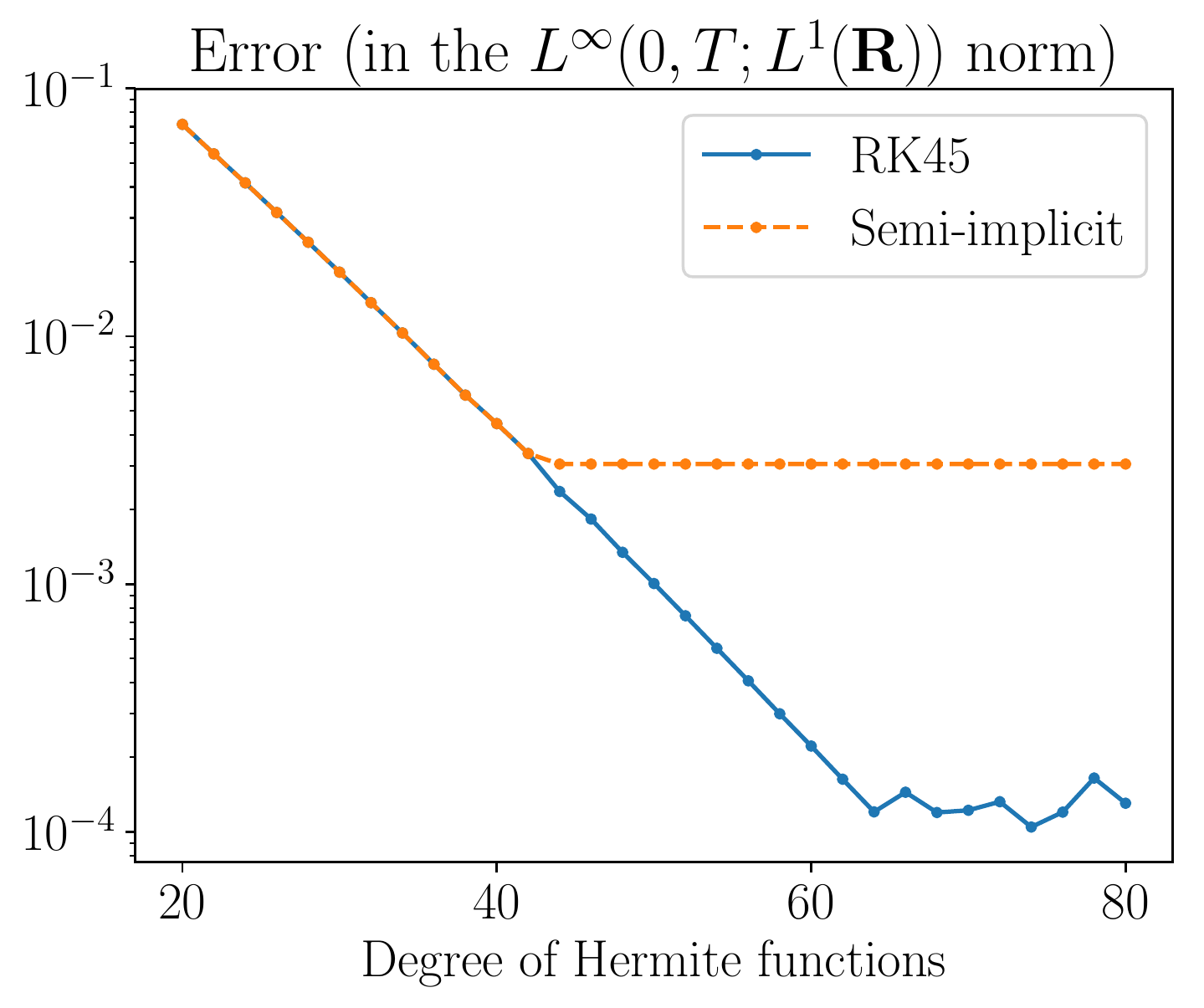}%
    \caption{%
        Convergence of the Hermite spectral method for the nonlinear McKean--Vlasov equation with colored noise.
        The error decreases exponentially for low enough values of $d$,
        and then reaches a plateau when
        it becomes dominated by the error induced by the time discretization.
    }
    \label{fig:mckean_vlasov_colored:convergence}
\end{figure}

\section*{Acknowledgments}
The authors are grateful to J.A. Carrillo for making available the finite volume code employed in \cref{sub:white_noise_case} for
the comparison with our spectral method for the McKean--Vlasov dynamics.
This work was supported by EPSRC through grants number EP/P031587/1, EP/L024926/1, EP/L020564/1 and EP/K034154/1.
The work of SG is supported by the Leverhulme Trust through Early Career Fellowship ECF-2018-536.

\bibliographystyle{abbrv}
\bibliography{references}

\end{document}